\documentclass{siamltex}
\usepackage{graphicx, latexsym}
\usepackage{amssymb}
\usepackage{euscript,color}
\usepackage[bb=dsfontserif]{mathalpha}

\def\cal{\EuScript}

\newcommand{\eps}{\varepsilon}
\newcommand{\norm}[1]{\ensuremath{\| #1 \|_2}}

\def\CP{\cal P}
\def\CPs{\CP{\kern-0.8pt}}
\def\mingmres{\min_{\stackrel{\scriptstyle{p\in \CPs_k}}
                            {\scriptstyle{p(0)=1}}}}

\def\mingmresq{\min_{\stackrel{\scriptstyle{q\in \CPs_{k-2}}}
                            {\scriptstyle{q(0)=1}}}}

\newcommand{\C}{\mathbb{C}}
\newcommand{\Z}{\mathbb{Z}}
\newcommand{\Cn}{\ensuremath{\C^n}}
\newcommand{\Cnn}{\ensuremath{\C^{n\times n}}}
\newcommand{\R}{\mathbb{R}}

\newcommand{\POLY}{{\cal P}}
\newcommand{\KRY}{{\cal K}}
\newcommand{\ORDER}{{\cal O}}

\newcommand{\COND}{\kappa}
\newcommand{\DISK}[1]{\Delta_{#1}}
\newcommand{\CDISK}[1]{\overline{\Delta}_{#1}}
\newcommand{\PSA}{{\sigma_\eps}}
\newcommand{\SPEC}{\sigma}
\newcommand{\FOV}{{W}}
\newcommand{\CG}{{\Omega_{\rm CG}}}
\newcommand{\NRAD}{\mu}
\newcommand{\SRAD}{\rho}
\newcommand{\RANK}{{\rm rank}}
\newcommand{\RANGE}{{\rm Ran}}
\newcommand{\SPAN}{{\rm span}}

\def\wh#1{\widehat{#1}}
\def\BA{{\bf A}}
\def\BAs{{\bf A}\kern-.7pt}

\def\BE{{\bf E}}
\def\BH{{\bf H}}
\def\BHtil{\widetilde{{\bf H}}}
\def\BI{{\bf I}}
\def\BJ{{\bf J}}
\def\BItil{\widetilde{{\bf I}}}
\def\BP{{\bf P}}
\def\BU{{\bf U}}
\def\BV{{\bf V}}
\def\BX{{\bf X}}
\def\BLambda{\mbox{\boldmath$\Lambda$}}
\def\BPi{\mbox{\boldmath$\Pi$}}
\def\Bb{{\bf b}}
\def\Bc{{\bf c}}
\def\Be{{\bf e}}
\def\Bf{{\bf f}}
\def\Br{{\bf r}}

\def\Bv{{\bf v}}
\def\Bx{{\bf x}}
\def\adj{^*}
\def\Bzero{{\bf 0}}

\def\dop{{\rm d}}
\def\iop{{\rm i}}
\def\evp{EV$'$}

\def\cev{C_{\mbox{\scriptsize{\sc ev}}}}
\def\cfov{C_{\mbox{\scriptsize{\sc fov}}}}
\def\cpsa{C_{\mbox{\scriptsize{\sc psa}}}}
\def\Cro{\cfov}

\def\rev{\rho_{\mbox{\scriptsize{\sc ev}}}}
\def\rfov{\rho_{\mbox{\scriptsize{\sc fov}}}}
\def\rpsa{\rho_{\mbox{\scriptsize{\sc psa}}}}

\newcommand{\exhead}[1]{\subsection*{$\bullet$\ #1}}

\font\sevenex=cmex7
\def\sepshat{{\mbox{\sevenex\symbol{"62}}\!\!\!\eps}}

\mathcode`\@="8000
{\catcode`\@=\active \gdef@{\mkern1mu}} % See page 641 of Knuth's "Digital Typography"

\def\mydate{\number\day\ {\ifcase\month \or January\or February\or
              March\or April\or May\or June\or July\or August\or
              September\or October\or November\or December\fi} \number\year}

\def\mymy{{\ifcase\month \or January\or February\or
              March\or April\or May\or June\or July\or August\or
              September\or October\or November\or December\fi} \number\year}

\title{How Descriptive are GMRES Convergence Bounds?}
\author{Mark Embree\thanks{Department of Mathematics,
Virginia Tech, Blacksburg, VA 24061 ({\tt embree@vt.edu}).
The original version of this work appeared as 
Oxford University Computing Laboratory Technical Report TR 99/08 in June 1999.
This updated version corrects some errors, improves the presentation,
 and includes additional illustrations.}}
\begin{document}

\maketitle

%%%%%%%%%%%%%%%%%%%%%%%%%%%%%%%%%%%%%%%%%%%%%%%%%%%%%%%%%%%%%%%%%%%%%%%%%%%%%%%%
\begin{abstract}
GMRES is a popular Krylov subspace method for solving linear systems 
of equations involving a general non-Hermitian coefficient matrix.
The conventional bounds on GMRES convergence involve 
polynomial approximation problems in the complex plane.  
Three popular approaches pose this approximation problem on the spectrum, 
the field of values, or pseudospectra of the coefficient matrix. 
We analyze and compare these bounds,
illustrating with six examples the success and failure of each.
When the matrix departs from normality due only to a low-dimensional 
invariant subspace, we discuss how these bounds can be adapted 
to exploit this structure.
Since the Arnoldi process that underpins GMRES provides approximations 
to the pseudospectra, one can estimate the GMRES convergence bounds 
as an iteration proceeds.
\end{abstract}


%\begin{AMS}
%15A06, 65F10, 15A18, 15A60, 31A15
%\end{AMS}
%%%%%%%%%%%%%%%%%%%%%%%%%%%%%%%%%%%%%%%%%%%%%%%%%%%%%%%%%%%%%%%%%%%%%%%%%%%%%%%%

\pagestyle{myheadings}
\thispagestyle{plain}
\markboth{MARK EMBREE}{GMRES CONVERGENCE BOUNDS}

%%%%%%%%%%%%%%%%%%%%%%%%%%%%%%%%%%%%%%%%%%%%%%%%%%%%%%%%%%%%%%%%%%%%%%%%%%%%%%%%
\section{Introduction}
%%%%%%%%%%%%%%%%%%%%%%%%%%%%%%%%%%%%%%%%%%%%%%%%%%%%%%%%%%%%%%%%%%%%%%%%%%%%%%%%

Many algorithms for solving large, sparse systems of linear equations 
construct iterates that attempt to minimize the residual norm over 
all candidates in an affine Krylov subspace whose dimension grows at each step.
For non-Hermitian matrices, the GMRES algorithm of Saad and Schultz~\cite{ss86}
generates such optimal iterates.
This method, which uses the Arnoldi process to generate an orthonormal basis 
for the Krylov subspace, becomes intractable for problems that converge slowly.
Practical algorithms, such as restarted GMRES, BiCGSTAB, QMR 
(see, e.g., \cite[Ch.~5]{greenbaum97}, \cite[Ch.~7]{Saa03}),
reduce this computational expense but compromise the optimality,
and it is tough to characterize the convergence that results.
(For an indication of the complexity of restarted GMRES, see~\cite{Emb03}.)
The residual norms cannot be smaller 
than those produced by GMRES, since the algorithms choose iterates from
the same Krylov subspace; sometimes they can be related to the 
GMRES residual norm, as in the case of~QMR~\cite{fn91}.
Understanding GMRES convergence, facilitated by its optimality property, 
can thus be a step toward convergence analysis for other algorithms. 
Insight about GMRES convergence can also
inform the construction and evaluation of preconditioners
for non-Hermitian matrices.

Given a system of linear equations $\BA\Bx=\Bb$, with $\BA\in\Cnn$ and
$\Bx,\Bb\in\Cn$, the GMRES algorithm~\cite{ss86} iteratively generates solution
estimates $\Bx_k \approx \Bx$ based on an initial guess $\Bx_0$.
The residuals induced by these iterates, $\Br_k:=\Bb-\BA\Bx_k$, satisfy 
the minimum residual property,
\begin{equation} \label{opt}
\norm{\Br_k} = \mingmres \norm{p(\BA)\Br_0},
\end{equation}
where $\POLY_k$ denotes the set of polynomials of degree $k$ or less.

What properties of the coefficient matrix $\BA$ govern convergence?
In this work, we examine three proposed answers
to this question (see~\cite{greenbaum97}): 
eigenvalues with eigenvector condition number,
the field of values, and pseudospectra. 
When $\BA$ is normal (i.e., it has an orthogonal basis
of eigenvectors or, equivalently, it commutes with its adjoint),
convergence can be accurately bounded using the eigenvalues alone.
This is not the case for nonnormal matrices, as the 
construction of Greenbaum, Pt\'ak, and Strako\v{s}
vividly illustrates~\cite{gps96}:
given any set of eigenvalues $\SPEC(\BA)$, one can construct 
$\BA$ and $\Br_0$ to produce any (monotonically decreasing)
GMRES residual norms.
When $\BA$ is far from normal, the residual norms 
$\{\norm{\Br_k}\}$ often exhibits a period of initial stagnation 
before converging at a quicker asymptotic rate.  
The bounds we study here essentially differ in two ways:
how they account for this delay due to nonnormality, 
and the sets upon which they base the asymptotic rate convergence.

Section~2 describes three standard GMRES convergence bounds.
These characterizations can significantly overestimate the norm of
the residual when nonnormality is only associated with a few eigenvalues 
(e.g., several nearly aligned eigenvectors orthogonal to all other eigenvectors).  
To circumvent this shortcoming, we apply spectral 
projectors to modify the traditional formulations.  
This strategy bounds GMRES convergence using the condition numbers
of individual eigenvalues, and also leads to flexible generalizations
of the field of values and pseudospectra bounds.
In Section~3, we present six examples to illustrate that 
the three standard bounds can each give significant overestimates, 
but each can also be rather descriptive.
The bounds are also compared via the relationships between the 
eigenvectors, field of values, and pseudospectra.  
These examples highlight a strength of pseudospectral bounds
(elaborating an observation of Driscoll, Toh, and Trefethen~\cite[p.~564]{dtt98}):
by sampling the bound based on the $\eps$-pseudospectrum over a range of
$\eps$ values, one can potentially capture \emph{different phases of convergence}
via the envelope of these bounds.
While the cost of computing pseudospectra can deter the use of this bound
for large-scale problems, in Section~\ref{sec:adaptive} we suggest an approach 
for obtaining GMRES convergence \emph{estimates} at a lower computational 
expense based on approximate pseudospectra taken from the Hessenberg matrix
generated by the Arnoldi process within the standard GMRES implementation.

Though we are concerned with GMRES convergence for a linear system with 
a specific initial residual, all the analysis described here
first employs the inequality
\begin{equation}
\label{ideal}
 \norm{\Br_k} \le \mingmres \norm{p(\BA)}\,\norm{\Br_0}, 
\end{equation}
and then studies $\norm{p(\BA)}$ independent from $\Br_0$.
This approach leads to upper bounds for worst case GMRES convergence.
With a carefully crafted example, Toh proved that this inequality 
can be arbitrarily misleading for some nonnormal matrices~\cite{toh97}.  
There may be \emph{no} vector $\Br_0\in\Cn$ for which
$\norm{p(\BA)}$ equals $\norm{\Br_k}/\norm{\Br_0}$ at iteration $k$.
Examples of this extreme behavior are thought to be rare in 
practice~\cite[\S3.6]{toh96} and thus we are typically content to make
the inequality~(\ref{ideal}) and proceed with the analysis of 
$\norm{p(\BA)}$ that follows from it.

%%%%%%%%%%%%%%%%%%%%%%%%%%%%%%%%%%%%%%%%%%%%%%%%%%%%%%%%%%%%%%%%%%%%%%%%%%%%%%%%
\section{Three Convergence Bounds and Variations} \label{sec:survey}
%%%%%%%%%%%%%%%%%%%%%%%%%%%%%%%%%%%%%%%%%%%%%%%%%%%%%%%%%%%%%%%%%%%%%%%%%%%%%%%%

In this section, we describe three common convergence bounds for GMRES 
based on eigenvalues with the eigenvector condition number,
the field of values, and pseudospectra.  These bounds all fail to
describe convergence accurately when nonnormality is primarily 
associated with a few eigenvalues.  One can use spectral projectors 
to decouple sets of eigenvalues, leading to localized versions of 
these bounds that can be sharper than the conventional versions.

%%%%%%%%%%%%%%%%%%%%%%%%%%%%%%%%%%%%%%%%%%%%%%%%%%%%%%%%%%%%%%%%%%%%%%%%%%%%%%%%
\subsection{Eigenvalues with Eigenvector Conditioning}
%%%%%%%%%%%%%%%%%%%%%%%%%%%%%%%%%%%%%%%%%%%%%%%%%%%%%%%%%%%%%%%%%%%%%%%%%%%%%%%%

The first convergence bound suggested for GMRES predicts
convergence at a rate determined by the set of eigenvalues of $\BA$, 
denoted $\SPEC(\BA)$.  If $\BA$ is normal, $\SPEC(\BA)$ determines 
convergence, since for any polynomial $p$
\[ \norm{p(\BA)} = \max_{\lambda \in \SPEC(\BA)} |p(\lambda)|.\]
Nonnormality can impede the onset of convergence at this
spectral rate; to account for such delay, this bound scales the 
spectral convergence behavior by the condition number of the matrix
having the eigenvectors of $\BA$ as its columns~\cite{ees83,ss86}.  
Provided that $\BA$ is diagonalizable, $\BA=\BV\BLambda \BV^{-1}$, 
\begin{eqnarray*}
 \norm{\Br_k} = \mingmres \norm{p(\BA)\Br_0} 
   &\le& \norm{\BV p(\BLambda) \BV^{-1}}\,\norm{\Br_0} \\[-8pt]
   &\le& \norm{\BV} \norm{\BV^{-1}} \norm{p(\BLambda)} \norm{\Br_0},
\end{eqnarray*}
implying the bound
$$ \frac{\norm{\Br_k}}{\norm{\Br_0}} \le 
   \COND(\BV) \mingmres \max_{\lambda\in\SPEC(\BA)} |p(\lambda)|. 
    \leqno{{\rm (EV)}}$$
Here, $\COND(\BV) := \norm{\BV}\,\norm{\BV^{-1}}$ is the 2-norm condition
number of the eigenvector matrix~$\BV$.\ \ If $\BA$ is normal, then $\COND(\BV)=1$;
if, in addition, the eigenvalues are real, then
(EV) reduces to the standard convergence bound for MINRES~\cite{fischer96}.\ \ 
If $\BA$ is nonnormal, then $\COND(\BV)>1$ and determining the optimal 
value of $\COND(\BV)$ can be a challenge~\cite{gs94}; this task is
further complicated if $\BA$ has repeated eigenvalues.  
Throughout this work, we scale the columns of $\BV$ 
to have unit 2-norm; provided each eigenvalue of $\BA$ is simple, 
this scaling ensures that $\COND(\BV)$ is no more than $\sqrt{n}$ times 
its optimal value, where $n$ is the matrix dimension~\cite{vds69}.

Since $\SPEC(\BA)$ is a discrete point set for finite-dimensional
matrices, the polynomial approximation problem in~(EV) will be 
zero when $k$ reaches the matrix dimension, $k=n$: 
\emph{thus the bound (EV) captures the finite termination of GMRES}.\ \ 
However (despite the small size of most of our examples),
we want to apply GMRES for large $n$,
in the hope of obtaining convergence after $k\ll n$ iterations.
Thus the bound (EV) is often employed with $\SPEC(\BA)$ replaced
by some compact set $\Omega \supset \SPEC(\BA)$.
For example, if all eigenvalues of $\BA$ are real and positive,
one might take $\Omega$ to be the real interval connecting the
extreme eigenvalues, as is common in analysis of the conjugate
gradient method; see, e.g.,~\cite[chap.~38]{TB97}.
(Better bounds result from including outlying eigenvalues as
singletons, and bounding the rest of the spectrum in aggregate~\cite{vdSV86,vdvv93}.)

The constant $\COND(\BV)$ in (EV) reflects the departure of $\BA$ 
from nonnormality.  The normalized residual norms $\norm{\Br_k}/\norm{\Br_0}$
form a nonincreasing sequence starting with the value~1 at $k=0$.
Thus if $\COND(\BV)$ is large, the bound (EV) cannot describe 
convergence at least until the iteration $k$ at which the polynomial minimization 
term is as small as $1/\COND(\BV)$.\ \ 
Even then, (EV)~can be grossly inaccurate.
For example, $\COND(\BV)$ can be large because all the eigenvectors 
are ill-conditioned, or if only two eigenvectors are nearly aligned.  
In the latter case, the bound usually fails to predict convergence, 
while it may be more appropriate in the former case.  
These situations are illustrated in Examples~B and~E of Section~3.
(For a discussion about the shortcomings of scalar measures of 
nonnormality see~\cite[chap.~48]{TE05}.)

In an effort to avoid this difficulty,
the bound (EV) can be adapted by considering the conditioning 
of individual eigenvalues.  
Suppose $\lambda\in\SPEC(\BA)$ is simple, 
having left and right eigenvectors $\wh{\Bv}$ and $\Bv$.
Then the condition number of $\lambda$~\cite[\S2.8]{aep} is
\[ \COND(\lambda) := \frac{\norm{\wh{\Bv}}\,\norm{\Bv}}{|\wh{\Bv}^*\Bv|} 
                  = {1\over \cos\angle(\wh{\Bv},\Bv)}.\]
Using these condition numbers leads to a bound that can be  
much sharper than~(EV).  This result can be seen as a special case 
of a theorem of Joubert~\cite[Thm.~3.2(4)]{Jou94a}, who presents
it in the context of the Jordan canonical form.

\medskip
\begin{theorem} \label{ewbound}
Suppose every eigenvalue $\lambda_j$ of $\BA$ is simple.  Then 
for any $p\in\POLY_k$,
\begin{equation} \label{ewcond}
\norm{p(\BA)} \le \sum_{j=1}^n \COND(\lambda_j)\,|p(\lambda_j)|.
\end{equation}
\end{theorem}
\begin{proof}
Since $\BA$ has simple eigenvalues, it is diagonalizable,
$\BA=\BV\BLambda \BV^{-1}.$ Let $\{\wh{\Bv}_j\}_{j=1}^n$ be the left eigenvectors
($\wh{\Bv}_j^*$ is the $j$th row of $\BV^{-1}$) and $\{\Bv_j\}_{j=1}^n$
the corresponding right eigenvectors (columns of $\BV$), with
$\BLambda_{jj} = \lambda_j$. Then
$$ \norm{p(\BA)} = \norm{\BV p(\BLambda) \BV^{-1}}
               = \biggm\| \sum_{j=1}^n p(\lambda_j) \Bv_j \wh{\Bv}_j^*\biggr\|_2
               \le \sum_{j=1}^n |p(\lambda_j)|\, \norm{\Bv_j \wh{\Bv}_j^*}.$$
The result follows from noting that since $\wh{\Bv}_j^*\Bv_j^{} = 1$ by
construction,
 $\norm{\Bv_j \wh{\Bv}_j^*} = \norm{\wh{\Bv}_j}\,\norm{\Bv_j} = 
\norm{\wh{\Bv}_j}\,\norm{\Bv_j}/|\wh{\Bv}_j^*\Bv_j| = \COND(\lambda_j)$.
\quad
\end{proof}

\medskip
The quantity on the right of equation~(\ref{ewcond})
is the 1-norm of $p(\BLambda) \Bc$ where $c_j = \kappa(\lambda_j)$.
Using a norm equivalence, we can obtain a conventional 
GMRES problem involving a normal matrix, but with a special right hand side.

\medskip

\begin{corollary} \label{evprime}
Define the components of $\Bc\in\Cn$ to be $c_j = \kappa(\lambda_j)$. Then
$$ \frac{\norm{\Br_k}}{\norm{\Br_0}} 
 \le \sqrt{n} \mingmres \norm{p(\BLambda)\Bc}.
   \leqno{{\rm (EV')}}$$
\end{corollary}

Note that $\|\Bc\|_2 \ge \sqrt{n}$, since each $c_j \ge 1$.
If $c_j$ is large, the corresponding eigenvalue $\lambda_j$ is
ill-conditioned.  If only a few eigenvalues of $\BA$ are ill-conditioned,
then the GMRES problem on the right-hand side of (EV$'$) can exhibit 
an initial phase of rapid convergence 
(as one could choose a polynomial $p\in\CP_k$
with roots at the ill-conditioned eigenvalues, eliminating those
components from $\Bc$), leading to more rapid convergence than
one would expect for a typical initial residual of similar 
magnitude.  Figure~\ref{supg:evprime} shows the ability 
of (EV$'$) to describe convergence for a highly nonnormal
matrix from a convection-diffusion problem.

When $\COND(\BV)$ is large (or $\BA$ is nondiagonalizable,
which we regard as $\COND(\BV)=\infty$),
it may be appealing to obtain GMRES bounds with smaller leading
constants, at the expense of polynomial approximation problems on
larger sets in the complex plane.

%%%%%%%%%%%%%%%%%%%%%%%%%%%%%%%%%%%%%%%%%%%%%%%%%%%%%%%%%%%%%%%%%%%%%%%%%%%%%%%%
\subsection{Field of Values}
%%%%%%%%%%%%%%%%%%%%%%%%%%%%%%%%%%%%%%%%%%%%%%%%%%%%%%%%%%%%%%%%%%%%%%%%%%%%%%%%

The \emph{field of values} (or \emph{numerical range}),
\begin{equation} \label{eq:fovdef}
 \FOV(\BA) := \left\{\,\frac{\Bx\adj\BA\Bx}{\Bx\adj\Bx}:
                 \Bx\in\Cn,\, \Bx\ne0 \right\},
\end{equation}
is a popular alternative to eigenvalues
for understanding the behavior of functions of nonnormal matrices. 
The field of values is always a closed, convex set that contains $\SPEC(\BA)$,
and so it is possible that $0\in\FOV(\BA)$ even when $\BA$ is nonsingular.
When $\BA$ is far from normal, $\FOV(\BA)$ can contain points far beyond the
convex hull of $\SPEC(\BA)$.
While the eigenvalues of $\BA$ can be sensitive to perturbations,
the field of values is robust:
\[ \FOV(\BA+\BE) \subseteq \FOV(\BA) + \{ z\in\C: |z| \le \norm{\BE}\}.\]
Moreover, the extreme eigenvalues of the Hermitian 
($(\BA+\BAs^*)/2$) and skew-Hermitian $(\BA-\BAs^*)/2$ parts of $\BA$
give tight bounds on the real and imaginary extent of $\FOV(\BA)$;
see~\cite[chap.~2]{hojo2} for the basic algorithm for computing $\FOV(\BA)$, 
and Bracconier and Higham~\cite{bh96} for an approach for large-scale problems.
For many of the examples that follow, we use Higham's {\tt fv} code~\cite{HighTMT}. 
For additional properties of $\FOV(\BA)$, 
see~\cite{gustrao}, \cite[chap.~1]{hojo2}, and \cite[chap.~17]{TE05}.

To develop a GMRES bound based on the field of values, 
we seek an alternative to~(EV) that replaces maximization 
over $\SPEC(\BA)$ by maximization over $\FOV(\BA)$.
Enlarging this maximizing set generally increases the polynomial 
approximation term, which will hopefully be counterbalanced by 
a leading constant that is smaller than $\kappa(\BV)$ in~(EV).

How does $\|p(\BA)\|$ relate to $|p(z)|$ for $z\in\FOV(\BA)$?
This question dates back at least to the 1960s, but major progress 
has been made in recent years by Crouzeix and collaborators.
Indeed, Crouzeix has shown that for any $p\in\CP_k$, 
there exists a constant $\Cro$ (independent of $\BA$, the degree $k$, 
and the dimension $n$) such that
\begin{equation} \label{eq:cro}
 \|p(\BA)\| \le \Cro \max_{z\in\FOV(\BA)} |p(z)|.
\end{equation}
How large is $\Cro$?
Crouzeix~\cite{Cro07} proved that $2 \le \Cro \le 11.08$ and conjectured that $\Cro = 2$.
More recently, Crouzeix and Palencia~\cite{CP17} showed that $2 \le \Cro \le 1+\sqrt{2}$.
We thus have
$$ \frac{\norm{\Br_k}}{\norm{\Br_0}} 
\le \Cro \mingmres \max_{z\in\FOV(\BA)} |p(z)|,
%\le (1+\sqrt{2}) \mingmres \max_{z\in\FOV(\BA)} |p(z)|.
   \leqno{{\rm (FOV)}} $$
with $\Cro \le 1+\sqrt{2}$.
We note that this bound complements important earlier work by Eiermann.%
\footnote{To the best of our knowledge, Eiermann was the first to use
the field of values to bound the convergence
of iterative linear solvers~\cite{eiermann93,eiermann97}.
Working before Crouzeix's analysis, he bounded $\FOV(\BA)$ by an ellipse 
in $\C$, then used Chebyshev polynomials to bound the optimal polynomial
on this ellipse.
Eiermann and Ernst proposed a different field of values bound
involving $\FOV(\BA^{-1})$ in~\cite[sect.~6]{EE01}.} 
While Crouzeix's conjecture is known to hold for certain classes of matrices,
in all the illustrations that follow we use $\cfov=1+\sqrt{2}=2.4142\ldots.$

\medskip
The bound~(FOV) has a small leading constant and is appealingly simple,
but it suffers from several notable limitations.

\begin{remunerate}
\item
It is possible that $0\in\FOV(\BA)$ even when $\BA$ is nonsingular
($0 \not\in \SPEC(\BA)$).
Then
\[ \mingmres \max_{z\in\FOV(\BA)} |p(z)|
   \ge \mingmres |p(0)| = 1,\]
and so the bound~(FOV) fails to describe any convergence,
despite the fact that the full GMRES algorithm must converge.
This limitation makes (FOV)  unsuitable for Hermitian indefinite $\BA$,
not to mention nonnormal problems for which $0$ is embedded
within $\FOV(\BA)$ despite the spectrum being relatively well 
separated from the origin. 

\item Since $\FOV(\BA)$ is a convex set in $\C$, it hides
information about the \emph{distribution of the eigenvalues} within $\FOV(\BA)$.  
Suppose $\BA$ is Hermitian with 
$\FOV(\BA) = [\alpha,\beta] \subset \R$ for $0 < \alpha < \beta$.
GMRES will converge very differently if $\SPEC(\BA)$ is uniformly distributed
throughout $[\alpha,\beta]$, or if $\SPEC(\BA)$ consists of a cluster near $\alpha$ 
and a single eigenvalue near $\beta$, yet both scenarios give the same $\FOV(\BA)$.
The bound~(FOV) cannot capture so-called \emph{superlinear} convergence effects
associated with isolated outlying eigenvalues~\cite{dtt98,vdvv93}.

\item Similarly, suppose $\FOV(\BA)$ contains points far from $\SPEC(\BA)$ due to 
a large departure from normality associated with a low-degree invariant 
subspace of $\BA$.
In the GMRES approximation problem, the optimizing polynomial $p$ could 
target these ill-conditioned eigenvalues, effectively eliminating the nonnormality 
from the problem at an early iteration: later iterations can focus on the 
rest of the spectrum.  In contrast, the bound~(FOV) must continue to 
optimize $p$ over all $\FOV(\BA)$ at every iteration. 

\item As the constant term $\cfov$ in~(FOV) is small and the approximation
problem on a convex region predicts asymptotic linear convergence 
(see equation~(\ref{convex})), the bound~(FOV)
cannot be entirely descriptive for iterations that initially stagnate
before converging at a more rapid asymptotic rate.
This behavior is observed by Higham and Trefethen
for matrix powers~\cite{ht93}, and identified
by Ernst in the context of GMRES applied to convection-diffusion 
problems~\cite{Ern00}, as we see in Section~\ref{sec:convdiff}.
\end{remunerate}

\medskip
We can remedy some limitations of (FOV) by
working with projectors onto invariant subspaces of $\BA$.  
Partition the spectrum of $\BA$ into disjoint sets $\Lambda_j$, 
such that $\SPEC(\BA) = \cup_{j=1}^m \Lambda_j$.  Define the spectral projector
\[ \BP_j :=  \frac{1}{2\pi@\iop} \int_{\Gamma_j} (z\BI-\BA)^{-1}\,\dop z,\]
where $\Gamma_j$ is the union of Jordan curves containing the eigenvalues
$\Lambda_j$ in their collective interior, but not enclosing any other
eigenvalues.  Then $\BP_j$ is a projector onto the invariant
subspace of $\BA$ associated with the eigenvalues 
$\Lambda_j$ (see, e.g.,~\cite[sect.~I.5.3]{kato}).

\medskip
\begin{theorem} \label{gensp} 
Let $\{\Lambda_j\}_{j=1}^m$ be a partition of $\SPEC(\BA)$
into $m$ disjoint sets.  For each $1\le j\le m$,  let $\BP_j$
be the spectral projector onto the invariant subspace 
associated with $\Lambda_j$, and let the columns of
$\BU_j^{n\times\RANK(\BP_j)}$ be an orthonormal basis for
$\RANGE(\BP_j)$.
Then for any polynomial $p\in\POLY_k$, 
\[ \norm{p(\BA)} \le \sum_{j=1}^m 
   \norm{\BP_j}\,\norm{p(\BU_j\adj\BA\BU_j)}.\]
\end{theorem}
\begin{proof}
Note that $\BPi_j := \BU_j\BU_j\adj$
is the orthogonal projector onto $\RANGE(\BP_j)$, an
invariant subspace of $\BA$.
(If $\BA$ is normal, then $\BPi_j = \BP_j$; for nonnormal $\BA$, 
we generally have $\BPi_j \ne \BP_j$ though both $\BPi_j$ and $\BP_j$ 
are projectors onto the same subspace.)
We apply several important identities for spectral projectors:
$\sum_{j=1}^m \BP_j = \BI$, 
$\BA\BP_j = \BA\BPi_j\BP_j$,
and $\BA \BPi_j = \BPi_j \BA \BPi_j$
(see, e.g., \cite[\S I.5.3]{kato},\cite[\S~3.1]{Saa11}).
Substituting the first identity into $\norm{p(\BA)}$ yields
\begin{eqnarray} 
 \norm{p(\BA)} = \biggl\|\, p(\BA) \sum_{j=1}^m \BP_j\, \biggr\|_2 
               &\le& \sum_{j=1}^m \norm{p(\BA)\BP_j} 
                 = \sum_{j=1}^m \norm{\BPi_j\,p(\BA)\,\BPi_j\BP_j}  \nonumber \\
               &\le& \sum_{j=1}^m \norm{\BPi_j\,p(\BA)\,\BPi_j}\,\norm{\BP_j}.
               \label{polybound}
\end{eqnarray}
Notice that for each $j$, 
$\norm{\BPi_j p(\BA) \BPi_j} = \norm{p(\BPi_j\BA\BPi_j)}
 = \norm{\BU_j\,p(\BU_j\adj \BA\BU_j)\,\BU_j\adj}
             \le \norm{p(\BU_j\adj \BA\BU_j)}$. 
Using this bound in~(\ref{polybound}) completes the proof.\qquad
\end{proof}

\medskip
Joubert~\cite[Thm.~3.2(4)]{Jou94a} presents a similar result that is
presented using the language of the Jordan canonical form 
rather than spectral projectors. 

Theorem~\ref{gensp} provides a natural tool for transitioning
between global statements like~(EV) and localized statements like~(\evp).  
In the former case, Theorem~\ref{gensp} is vacuous since the spectrum
is partitioned into a single set; in the latter case,
Theorem~\ref{gensp} reduces to Theorem~\ref{ewbound}:
each set $\Lambda_j$ is a single eigenvalue,
which implies $\norm{\BP_j} = \COND(\lambda_j)$.
In summary, \emph{one can eliminate the nonnormal coupling between sets of eigenvalues
at the cost of scaling by the norm of the associated spectral projectors.}

Theorem~\ref{gensp} can be combined with the analysis 
leading to the bound (FOV) to provide a field of values
analogue to (\evp).

\medskip

\begin{corollary} \label{cor:fovp}
   Partition the eigenvalues $\SPEC(\BA)$ into disjoint sets,
$\{\Lambda_j\}_{j=1}^m$, with the columns of
each $\BU_j$ giving an orthonormal basis for the invariant 
subspace of $\BA$ associated with $\Lambda_j$.  Then
$$ \frac{\norm{\Br_k}}{\norm{\Br_0}} 
        \le (1+\sqrt{2}) \mingmres 
           \sum_{j=1}^m 
\bigg(\norm{\BP_j} \max_{z\in \FOV(\BU_j\adj\BA\BU_j)} |p(z)|\bigg).
\leqno{({\rm FOV}')}$$
\end{corollary}

This localization procedure enables one to work around numerous
scenarios in which $0\in\FOV(\BA)$,
but some matrices remain out of reach.  
For example, if $\BA$ is a matrix with a single eigenvalue,
Corollary~\ref{cor:fovp} does not permit any splitting of the invariant subspaces
(though similar ideas could be applied if a nontrivial block diagonalization
is possible~\cite[Thm.~3.2(4)]{Jou94a}).
If additionally $0\in\FOV(\BA)$, analysis based on the field of values cannot 
predict any convergence.

Recent analysis by Crouzeix and Greenbaum~\cite{CG19} provides a different
way to handle cases where $0\in\FOV(\BA)$.\ \ 
Let $\mu(\BA) := \max_{z\in \FOV(\BA)} |z|$ denote the \emph{numerical radius}
of $\BA$.\ \ Suppose $0\in\FOV(\BA)$ and let us denote by $\CG$ 
the region formed by the intersection of $\FOV(\BA)$ with the exterior of the disk centered at the origin with radius $1/\mu(\BAs^{-1})$:
\begin{equation} \label{eq:CG}
 \CG  := \FOV(\BA) \cap \{ z\in\C: |z|\ge 1/\mu(\BAs^{-1})\}.
\end{equation}
This set effectively results from ``carving out'' from $\FOV(\BA)$ 
the disk of radius $1/\mu(\BAs^{-1})$
centered at the origin.
Crouzeix and Greenbaum~\cite[p.~1098]{CG19} show that $\Omega_{\rm CG}$ 
is a $(3+2\sqrt{3})$-spectral set, giving the GMRES bound
$$ {\|\Br_k\|_2 \over  \|\Br_0\|_2} 
    \le \big(3+2\sqrt{3}\big) \mingmres \max_{z\in\CG} |p(z)|.
   \leqno{{\rm (CG)}}$$
As noted in~\cite{CG19}, when $\CG$ surrounds the origin 
the polynomial approximation problem in~(CG) must equal~1
(due to the maximum-modulus theorem);
see the left plot in Figure~\ref{fig:CGex2} for an example.
If $\CG$ does not surround the origin
-- as can occur even in some cases where $\BA$ has just one eigenvalue,
as in the right plot in Figure~\ref{fig:CGex2} -- the bound~(CG) gives convergence at an asymptotic 
rate determined by $\CG$.\ \ 
However, since $3+2\sqrt{3}\approx 6.464$ is still a small constant,
like (FOV) the bound~(CG) will not describe iterations that initially
converge slowly, before accelerating at later iterations.
The bound~(CG) seems to be especially useful when $0\in\FOV(\BA)$
because of the requisite convexity of $\FOV(\BA)$, rather than
nonnormality.  For example, suppose 
\begin{equation} \label{eq:CG1}
 \BA = \BA_{-1} \oplus \BA_{+1} := 
     \left[\begin{array}{cccc}
         -1 & 1/4 \\ & -1 & \ddots \\ & & \ddots & 1/4 \\ & & & -1
     \end{array}\right] 
      \oplus
     \left[\begin{array}{cccc}
         +1 & 1/2 \\ & +1 & \ddots \\ & & \ddots & 1/2 \\ & & & +1
     \end{array}\right],
\end{equation}
where $\BA_{-1}$ and $\BA_{+1}$ are bidiagonal matrices of dimension $n/2=32$.
Though the matrices in the direct sum are both far from normal,
$0 \not \in \FOV(\BA_{\pm1})$. However, since $\FOV(\BA)$ is
convex, $0\in \FOV(\BA)$; see Figure~\ref{fig:CGex1}.
The set $\CG$ remedies this ``incidental'' inclusion of the origin
in $\FOV(\BA)$, and (CG) gives a convergent bound.
(Figure~\ref{fig:CGex1} shows several interesting features:
the left lobe of $\CG$ is larger than $\FOV(\BA_{-1})$,
and moving the eigenvalue $-1$ further to the left would 
enlarge this component; the right lobe of $\CG$ omits some 
points in $\FOV(\BA_{+1})$.)
In contrast, for this example the bound (FOV$'$) would give
\[ {\|\Br_k\|_2 \over  \|\Br_0\|_2} 
    \le \big(1+\sqrt{2}\big) 
         \mingmres
         \bigg( \Big(\max_{z\in \FOV(\BA_{-1})} |p(z)|\Big)
               + \Big(\max_{z\in \FOV(\BA_{+1})} |p(z)|\Big) \bigg),
\]
since in this case the spectral projectors for the two distinct eigenvalues are orthogonal.

%%%%%%%%%%%%%%%%%%%%%%%%%%%%%%%%%%%%%%%%%%%%%%%%%%%%%%%%%%%%%%%%%%%%%%%%%%%%%%%%
% FIGURES FOR CROUZEIX-GREENBAUM example
%%%%%%%%%%%%%%%%%%%%%%%%%%%%%%%%%%%%%%%%%%%%%%%%%%%%%%%%%%%%%%%%%%%%%%%%%%%%%%%%

\begin{figure}[t!]
\begin{center}
\includegraphics[width=2.45in]{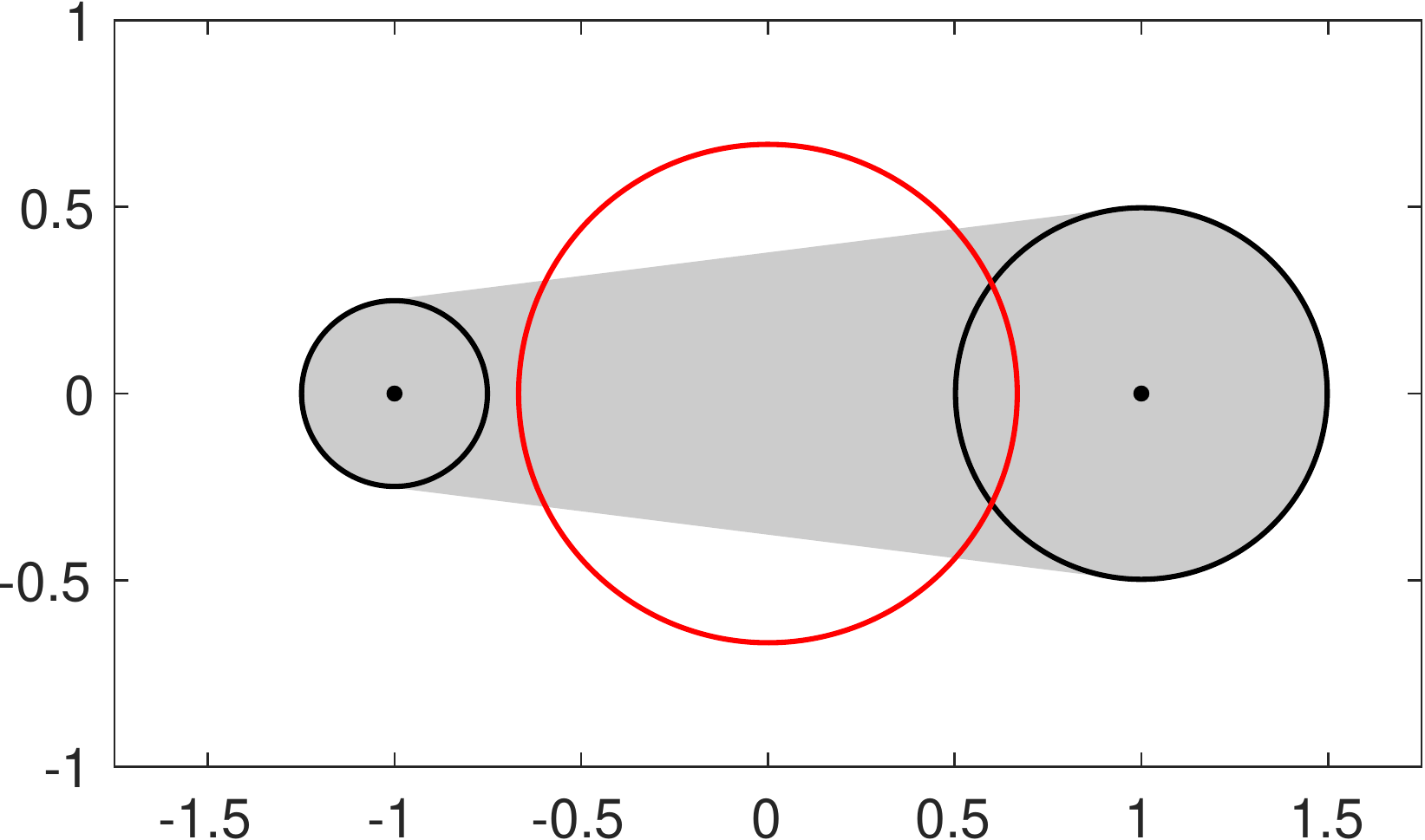}\quad
\includegraphics[width=2.45in]{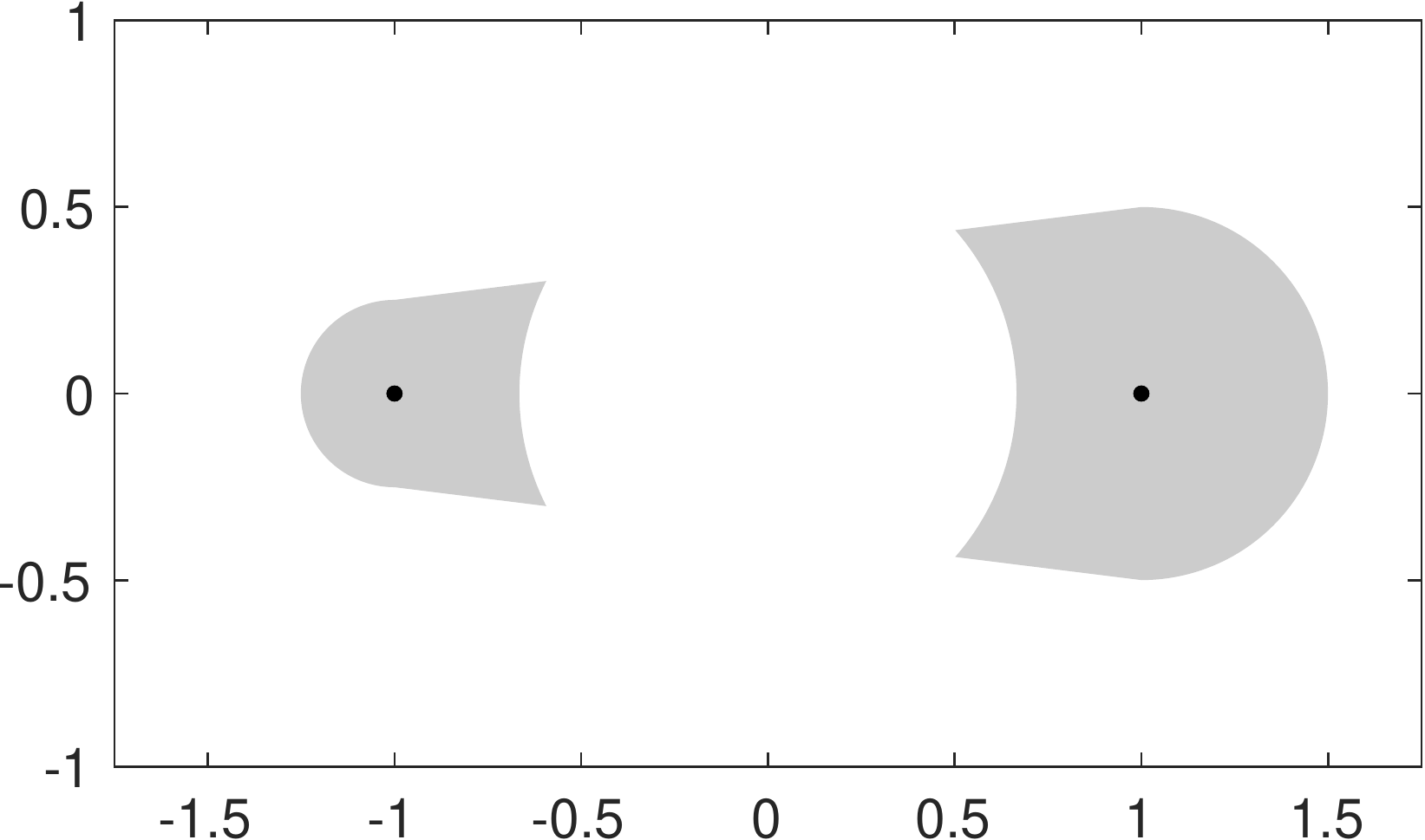}
\end{center}

\caption{\label{fig:CGex1}
For the matrix $\BA = \BA_{-1}\oplus \BA_{+1}$ in~(\ref{eq:CG1}):
on the left, $\SPEC(\BA) = \{\pm 1\}$ (dots), 
$\FOV(\BA)$ (gray region), 
the boundaries of $\FOV(\BA_{-1})$ and $\FOV(\BA_{+1})$ (black lines),
and the circle of radius $1/\mu(\BA^{-1})$ (red line);
on the right, the Crouzeix--Greenbaum region $\CG$.
Note that $0\in\FOV(\BA)$ but $0\not\in\CG$.
}
\end{figure}

%%%%%%%%%%%%%%%%%%%%%%%%%%%%%%%%%%%%%%%%%%%%%%%%%%%%%%%%%%%%%%%%%%%%%%%%%%%%%%%%
\subsection{Pseudospectra}
%%%%%%%%%%%%%%%%%%%%%%%%%%%%%%%%%%%%%%%%%%%%%%%%%%%%%%%%%%%%%%%%%%%%%%%%%%%%%%%%

Pseudospectra provide another generalization of $\SPEC(\BA)$ upon which 
to base GMRES convergence bounds.
The $\eps$-pseudospectrum~\cite{lnt92,TE05} $\PSA(\BA)$ has as its boundary the
$\eps$-level set of the norm of the resolvent:
$$\PSA(\BA) := \{z\in\C:  \norm{(z\BI-\BA)^{-1}} > \eps^{-1}\}. $$
The $\eps$-pseudospectrum can equivalently be defined in terms of eigenvalues
of perturbations of $\BA$:
$\PSA(\BA) = \{z\in\C:  z\in\SPEC(\BA+\BE), \norm{\BE} < \eps\}.$
Like the field of values, $\PSA(\BA)$ is robust to perturbations,
in the sense that $\PSA(\BA+\BE) \subseteq \SPEC_{\eps+\|\BE\|}(\BA)$.
Indeed, the pseudospectra can be regarded as a bridge 
between the eigenvalues and the field of values~\cite{ht93}.
Pseudospectra are more expensive to compute than $\SPEC(\BA)$ and $\FOV(\BA)$,
but robust algorithms exist for their computation and approximation in the 
large-scale case; see~\cite{lnt99,Wri02b,WT01b}.  For many of the examples 
below, we use Trefethen's code in~\cite{lnt99} or 
Wright's EigTool software~\cite{Wri02a}.

In an early application of pseudospectral theory, Trefethen~\cite{lnt90}
developed GMRES bounds by working from the Dunford--Taylor 
integral~\cite[\S I.5.6]{kato} 
\begin{equation}\label{dunford}
 p(\BA) = \frac{1}{2\pi\iop}\int_\Gamma p(z) (z\BI-\BA)^{-1}\, \dop z
\end{equation}
for $p\in\POLY_k$,
where $\Gamma$ is any union of Jordan curves containing $\SPEC(\BA)$
in its collective interior.  For a fixed $\eps>0$, we can take the
contour $\Gamma_\eps$ to be the boundary of $\PSA(\BA)$.
(If this is not the union of Jordan curves, take $\Gamma_\eps$
to be slightly exterior.) 
Coarsely approximating the resolvent norm about $\Gamma_\eps$ yields
$$ \norm{p(\BA)} \le  \frac{1}{2\pi} 
          \int_{\Gamma_\eps} |p(z)|\, \norm{(z\BI-\BA)^{-1}} \, \dop|z|
          \le \frac{{\cal L}(\Gamma_\eps)}{2\pi\eps} 
          \max_{z\in\PSA(\BA)} |p(z)|,$$
where ${\cal L}(\Gamma_\eps)$ is the contour length of $\Gamma_\eps$.
When applied to the GMRES problem, this inequality gives 
$$ \frac{\norm{\Br_k}}{\norm{\Br_0}} \le
   \frac{{\cal L}(\Gamma_\eps)}{2\pi\eps} \mingmres \max_{z\in\PSA(\BA)} |p(z)|.
   \leqno{{\rm (PSA)}}$$
What value $\eps$ should one use in~(PSA) to get a concrete bound for a given problem?
Notice that (PSA) is actually a \emph{family of bounds}
when sampled over all $\eps>0$~\cite[p.~564]{dtt98}).  
This crucial aspect of~(PSA) is often overlooked.
Larger values of $\eps$ tend to give smaller leading constants
${\cal L}(\Gamma_\eps)/(2\pi\eps)$ but larger sets $\PSA(\BA)$ over which to maximize $|p(z)|$:
such bounds can potentially capture the slow initial convergence 
of GMRES commonly observed for nonnormal $\BA$.
Smaller values of $\eps$ tend to give larger values of 
${\cal L}(\Gamma_\eps)/(2\pi\eps)$ but small $\PSA(\BA)$:
this regime can describe the faster later stage of GMRES convergence, 
with the large constant corresponding to the delayed onset of the fast convergence phase.
To see an illustration of this phenomenon, look ahead to Figure~\ref{Fbounds}.

Several different strategies lead to localized versions of~(PSA).
One can use Theorem~\ref{gensp} to decompose $\BA$ into 
invariant subspaces (at the cost of multiplication by the norms
of the spectral projectors), or one can simply replace 
$\PSA(\BA)$ with a conglomeration of disjoint components of 
pseudospectra using several different values of $\eps$.  

\medskip
\begin{corollary} \label{cor:psap}
   Partition the eigenvalues $\SPEC(\BA)$ into disjoint sets,
$\{\Lambda_j\}_{j=1}^m$, with the columns of
each $\BU_j$ giving an orthonormal basis for the invariant 
subspace of $\BA$ associated with $\Lambda_j$.  
Then for any $\eps>0$,
$$ \frac{\norm{\Br_k}}{\norm{\Br_0}} 
     %    \le \mingmres \norm{p(\BA)} 
        \ \le \ \mingmres\ \sum_{j=1}^m\ 
        \bigg( \norm{\BP_j} {{\cal L}(\Gamma_\eps(\BU_j^*\BA\BU_j)) \over 2\pi\eps} 
                \max_{z\in \PSA(\BU_j\adj\BA\BU_j)} |p(z)|\bigg),
\leqno{({\rm PSA}')}$$
where $\Gamma_\eps(\BU_j^*\BA\BU_j)$ is a Jordan curve enclosing $\PSA(\BU_j^*\BA\BU_j)$.
\end{corollary}

\medskip
\begin{corollary} \label{cor:psapp}
Partition the eigenvalues $\SPEC(\BA)$ into disjoint sets $\{\Lambda_j\}_{j=1}^m$,
and let $\{\Gamma_j\}_{j=1}^m$ be a set of non-intersecting Jordan curves and $\{\eps_j\}_{j=1}^m$ 
positive constants such that for each~$j$:
\begin{itemize}
\item The interior of $\Gamma_j$ contains all eigenvalues in $\Lambda_j$, and
      all other eigenvalues of $\BA$ are exterior to $\Gamma_j$; 
\item $\|(z\BI-\BA)^{-1}\| \le 1/\eps_j$ for all $z\in\Gamma_j$.
\end{itemize}
Then
$$ \frac{\norm{\Br_k}}{\norm{\Br_0}} 
     %    \le \mingmres \norm{p(\BA)} 
        \ \le \ \mingmres \ 
      \sum_{j=1}^m  \bigg({{\cal L}(\Gamma_j) \over 2\pi\eps_j} 
                            \max_{z\in \Gamma_j}\ |p(z)|\bigg).
\leqno{({\rm PSA}'')}$$
\end{corollary}

\noindent
This bound follows from choosing $\Gamma = \cup_{j=1}^m \Gamma_j$ 
in the integral~(\ref{dunford}).  Notice that when $\eps_1=\cdots=\eps_m$,
this bound reduces to (PSA).\ \  However, in some situations it might be advantageous
to choose, for example, a very small $\eps_j$ for a few eigenvalues near
the origin, but larger $\eps_j$ for the remaining eigenvalues farther from the origin.

%%%%%%%%%%%%%%%%%%%%%%%%%%%%%%%%%%%%%%%%%%%%%%%%%%%%%%%%%%%%%%%%%%%%%%%%%%%%%%%%
\subsection{Computing the Convergence Bounds}
%%%%%%%%%%%%%%%%%%%%%%%%%%%%%%%%%%%%%%%%%%%%%%%%%%%%%%%%%%%%%%%%%%%%%%%%%%%%%%%%

With each of the bounds (EV), (FOV), and (PSA) 
is associated a constant, defined as 
\begin{equation} \label{eq:constants}
 \cev := \COND(\BV),\qquad 
   \cfov := 1+\sqrt{2}, \qquad\mbox{and}\quad
   \cpsa(\eps) := \frac{{\cal L}(\Gamma_\eps)}{2\pi\eps}.
\end{equation}
The asymptotic behavior of each bound is determined 
by the associated complex approximation problem over
$\SPEC(\BA)$, $\FOV(\BA)$, or $\PSA(\BA)$.  

Let $\Omega\subset\C$ be a compact domain (without isolated points)
that tightly bounds $\FOV(\BA)$, $\PSA(\BA)$, or the clustered eigenvalues 
of $\SPEC(\BA)$.%
\footnote{If $\BA$ has finite dimension, $\SPEC(\BA)$ is a discrete point set 
with no finite asymptotic convergence rate. If the eigenvalues are clustered, 
the asymptotic convergence rate of the bounding set $\Omega$
typically describes convergence. Outlying eigenvalues 
do not affect this convergence \emph{asymptotic} rate; 
see~\cite{dtt98} for details.}
Provided $0\not\in\Omega$,
the error of the approximation problem
\[ \mingmres \max_{z\in\Omega} |p(z)|\]
decreases at an asymptotically linear
rate in the degree $k$ (see, e.g.,~\cite[Ch.~16]{hille2}):
\begin{equation} \label{eq:rho}
    \limsup_{k\to\infty} \bigg(\mingmres \max_{z\in\Omega} |p(z)|\bigg)^{1/k} = \rho.
\end{equation}
We call $\rho\in(0,1)$ the \emph{asymptotic (linear) convergence rate}
for $\Omega$.\ \ Driscoll, Toh, and Trefethen demonstrate how this constant can
be computed via conformal mapping~\cite{dtt98}.
When the set is a line segment or a disk, the rate is simple to 
compute.  For arbitrary polygons, the rate can be 
computed in MATLAB using Driscoll's Schwarz--Christoffel Toolbox
for numerical conformal mapping~\cite{driscoll96};
recent versions Chebfun~\cite{chebfun,GT19} also
include conformal mapping capabilities.
One might prefer to take $\Omega$ to be a disconnected set if
$\SPEC(\BA)$ has outliers or for $\PSA(\BA)$ with sufficiently small $\eps$.  
If each connected component of a disconnected set 
is a polygon on the real axis and is symmetric about the real 
axis, the rate can still be computed~\cite{ET99b}.  More general 
sets present greater difficulty, and it may in practice
 be necessary to bound $\SPEC(\BA)$ or $\PSA(\BA)$ with a 
single over-sized polygon.

The asymptotic characterization~(\ref{eq:rho}) does not directly provide a 
convergence bound at a specific iteration number, $k$: 
effectively, the $k$th root in~(\ref{eq:rho}) obscures 
a leading constant relating the minimization problem to $\rho^k$.
We are guaranteed that \begin{equation} \label{bestrate}
 \mingmres \max_{z\in\Omega} |p(z)| \ge \rho^k
\end{equation}
(see, e.g.,~\cite{dtt98}),  and if $\Omega$ is a disk, 
equation~(\ref{bestrate}) holds with equality.
When $\Omega$ is a segment $[a,b]$ of a line passing through
the origin, shifted and scaled Chebyshev polynomials are optimal.
In this case, the minimax error is bounded above by $2\rho^k$ and
known explicitly for each $k$ (see, e.g.,~\cite[\S6.11]{Saa03}).
If $\Omega$ is convex, Eiermann~\cite{eiermann89,eiermann97}
uses Faber polynomial analysis based on an approximation
theorem of K\"ovari and Pommerenke~\cite{kp67} to show that 
\begin{equation} \label{convex}
 \mingmres \max_{z\in\Omega} |p(z)| \le \frac{2\rho^k}{1-\rho^k}.
\end{equation}
In particular, this bound always holds when $\Omega = \FOV(\BA)$~\cite{eiermann97}.
In other circumstances, one can construct some polynomial $\phi\in\CP_k$
(e.g., by interpolating at well-chosen points on $\Omega$, or by
constructing the Faber polynomials associated with $\Omega$ from 
the conformal map that determines $\rho$) to obtain an upper bound:
\[ \mingmres \max_{z\in\Omega} |p(z)|
   \le \max_{z\in\Omega} |\phi(z)|.
\]
To unify notation,
we label the rate associated with each of the three sets
$\SPEC(\BA)$, $\FOV(\BA)$, and $\PSA(\BA)$ as
$\rev$, $\rfov$, and $\rpsa(\eps)$.  
In most situations we will base these constants on 
compact domains $\Omega$ that contain $\SPEC(\BA)$, 
$\FOV(\BA)$, and $\PSA(\BA)$.  (For a few examples
we will handle outlying eigenvalues separately.)

%%%%%%%%%%%%%%%%%%%%%%%%%%%%%%%%%%%%%%%%%%%%%%%%%%%%%%%%%%%%%%%%%%%%%%%%%%%%%%%%
\section{Which Bounds are Most Useful\hspace*{0.1em}}\hspace*{-10pt}{\bf ?} \label{sec:useful}
%%%%%%%%%%%%%%%%%%%%%%%%%%%%%%%%%%%%%%%%%%%%%%%%%%%%%%%%%%%%%%%%%%%%%%%%%%%%%%%%
The previous section described the standard
bounds (EV), (FOV), and (PSA), along with ``localized'' variants.
How do these bounds compare?
Are they redundant, or can each provide specific insight?
We begin by exploring the relationships
between the sets $\SPEC(\BA)$, $\FOV(\BA)$, and $\PSA(\BA)$,
and then turn to concrete examples illustrating the relative
merits of the three standard bounds.
We denote the spectral radius as
\[ \rho(\BA) := \max_{z\in\SPEC(\BA)} |z| \]
and the numerical radius as
\[ \mu(\BA) := \max_{z\in\FOV(\BA)} |z|.\]
Let $\DISK{r}:=\{z\in\C : |z|< r\}$ 
and $\CDISK{r}:=\{z\in\C : |z|\le r\}$ 
denote the open and closed disks of radius $r$.

If for a small $\eps>0$ the $\eps$-pseudospectrum contains points far 
from any eigenvalue, then the field of values must also contain points
far from $\SPEC(\BA)$, and the 
condition number $\COND(\BV)$ of the eigenvector matrix must also be large, 
as made precise by the following theorems.
The first, a version of the Bauer--Fike 
theorem~\cite{bf60},\cite[Thm.~3.2]{TE05}, 
bounds the $\eps$-pseudospectrum by the union of balls of
radius $\eps@\COND(\BV)$ centered at the eigenvalues.
The second (known already to Stone in 1932~\cite[Thm.~4.20]{Sto32}; 
see also~\cite[\S4.6]{gustrao}, \cite[Thm.~17.2]{TE05})
bounds the $\eps$-pseudospectrum in terms of the field of values.
The third relates the field of values to the eigenvector
condition number, a consequence of the basic
inequality $\NRAD(\BA)\le \norm{\BA}$.

\medskip
\begin{theorem}[Bauer--Fike] \label{bauerfike}
Let $\BA$ be diagonalizable, $\BA=\BV\BLambda\BV^{-1}$.  
Then for any $\eps>0$, $\PSA(\BA) \subseteq \SPEC(\BA) + 
\DISK{\eps\COND(\BV)}.$
\end{theorem}

\medskip
\begin{theorem} \label{psafov}
For any $\eps>0$,
$ \PSA(\BA)\subseteq \FOV(\BA) + \DISK{\eps}.$
\end{theorem}

\medskip
 \begin{theorem} \label{evfov}
 Let $\BA$ be diagonalizable, $\BA=\BV\BLambda\BV^{-1}$.  
 Then $\NRAD(\BA) \le \COND(\BV)\,\SRAD(\BA)$.
 \end{theorem}

Theorems~\ref{bauerfike} and~\ref{evfov} hold with equality 
when $\BA$ is normal.
Theorem~\ref{psafov} is sharp if $\BA$ is a multiple of 
the identity, or if $\BA$ is a Jordan block in the
limit $n=\infty$~(see Example~D).\ \   
For nonnormal matrices, all three bounds can be far from equality.

\medskip
Theorems~\ref{bauerfike} and~\ref{psafov} shed light on 
the constant $\cpsa(\eps) = {\cal L}(\Gamma_\eps)/(2\pi\eps)$
defined in~(\ref{eq:constants}).
More precisely, they permit insight about $\cpsa(\eps)$ 
when $\Gamma_\eps$ is taken to be a convenient
curve that encloses $\PSA(\BA)$ in its interior,
rather than the generally more complicated boundary of $\PSA(\BA)$.
(The bound (PSA) holds when $\PSA(\BA)$ 
is replaced by such a larger set: $\cpsa(\eps)$ is now
defined using $\Gamma_\eps$, and the polynomial approximation problem
associated with $\rpsa(\eps)$ in~(PSA) is now posed over the 
interior of $\Gamma_\eps$, rather than $\PSA(\BA)$ itself.)

By Theorem~\ref{bauerfike}, $\PSA(\BA)$ is 
bounded by the union of $n$ disks each with radius 
$\eps@\COND(\BV)$.  
Taking $\Gamma_\eps$ to be the boundary
of this union, ${\cal L}(\Gamma_\eps)$
can be no larger than $2\pi@n@\eps@\COND(\BV)$, so
$\cpsa(\eps) \le n@\COND(\BV)$ for this $\Gamma_\eps$.
But since $\SPEC(\BA)\subset\PSA(\BA)\subseteq {\rm interior}(\Gamma_\eps)$, 
we must have $\rev < \rpsa(\eps)$
for all $\eps>0$, and thus (PSA) is generally only useful 
for those $\eps$ for which $\cpsa(\eps) \ll n@\COND(\BV)$.

By Theorem~\ref{psafov}, if one takes $\Gamma_\eps$ to be the boundary of
the disk centered at the origin having radius $\NRAD(\BA)+\eps$, then
$\cpsa(\eps) = 1 + \NRAD(\BA)/\eps$.
Notice then that for such $\Gamma_\eps$ we have $\cpsa(\eps)\to1$ as $\eps\to\infty$.
(Of course, $\rpsa(\eps)=1$ for such sets.)

When the containment $\PSA(\BA)\subset\FOV(\BA)+\Delta_\eps$
is nearly equality even for small values of $\eps$,
the bound (FOV) can be slightly sharper than (PSA),
as is seen in Figures~\ref{fig:Abounds} and~\ref{Dbounds}.
In cases where the bound in Theorem~\ref{psafov} is far from equality
(i.e., $\PSA(\BA)$ does not contain points near the boundary of 
$\FOV(\BA)+\Delta_\eps$ for moderate values of $\eps$),
one often finds that (FOV) suggests slow, consistent
convergence, while (PSA) reflects convergence that eventually
accelerates to a more rapid rate, as in Figures~\ref{Ebounds},
\ref{Fbounds}, and~\ref{supg:bounds}.

%%%%%%%%%%%%%%%%%%%%%%%%%%%%%%%%%%%%%%%%%%%%%%%%%%%%%%%%%%%%%%%%%%%%%%%%%%%%%%%%
\subsection{The Examples} \label{sec:ex}
%%%%%%%%%%%%%%%%%%%%%%%%%%%%%%%%%%%%%%%%%%%%%%%%%%%%%%%%%%%%%%%%%%%%%%%%%%%%%%%%

The bounds (EV), (FOV), and (PSA) are descriptive in different
situations.  We demonstrate with six examples where the bounds 
succeed together, fail together, and, in turn,
fail and succeed alone.  These examples are summarized in
Table~\ref{table:all}.  We only discuss the
standard bounds, though in some instances a localized
version would fix the flaw that causes the corresponding
standard bound to fail.
It is difficult to show the failure of (PSA) with the 
simultaneous success of (EV) or (FOV). 
Example~C, showing success of (EV) with pessimistic 
(PSA) bounds, is the least satisfying of our six examples.
We also discuss why (FOV) cannot significantly outperform (PSA).

%%%%%%%%%%%%%%%%%%%%%%%%%%%%%%%%%%%%%%%%%%%%%%%%%%%%%%%%%%%%%%%%%%%%%%%%%%%%%%%%
\begin{table}[b!]
\caption{Predicted iterations for the six main examples in Section~\ref{sec:ex}.}
\label{table:all}

\begin{center}
\begin{tabular}{l l | c c c | c}
\multicolumn{2}{c|}{\emph{example}}  &  (EV) & (FOV) & (PSA) & \emph{true
iterations}\\[
0.2em] \hline
&&&& \\[-0.7em]
A  & all descriptive & 1 & 1 & 1 & 1 \\[0.4em]
B  & none descriptive & $\infty$ & $\infty$ & $\infty$ &
                                 \emph{see note 1} \\[0.4em]
C  & (EV) wins         & \multicolumn{3}{c|}{\emph{see note 2}} &  \\[0.4em]
D  & (EV) loses        & $\infty$ & $1$      & $1$      & 1 \\[0.4em]
   & (FOV) wins  & \multicolumn{3}{c|}{\emph{see note 3}}  & \\[0.4em]
E  & (FOV) loses       & $2$      & $\infty$ & $2$      & 2 \\[0.4em]
F  & (PSA) wins        & $\infty$ & $\infty$ & $2$      & 2 \\[0.4em]
   & (PSA) loses & \multicolumn{3}{c|}{\emph{see note 3}} & \\
\multicolumn{6}{l}{\rule[0.25em]{5em}{0.5pt} }\\
\multicolumn{6}{l}{1. $\lceil 2+\log(\mbox{\sc tol})/\log(\rho)\rceil$ iterations, for
                      parameters $\mbox{\sc tol}<1$ and $\rho\in(0,1)$.}\\
\multicolumn{6}{l}{2. ``(EV) wins'' is more involved; 
                       details are given in the text.} \\
\multicolumn{6}{l}{3. (FOV) cannot significantly beat (PSA);
                       see the ``No Example'' sections.} 
\end{tabular}
\end{center}
\end{table}
%%%%%%%%%%%%%%%%%%%%%%%%%%%%%%%%%%%%%%%%%%%%%%%%%%%%%%%%%%%%%%%%%%%%%%%%%%%%%%%%

In the illustrations that follow, the Ideal GMRES value
\[ \mingmres \norm{p(\BA)}\]
is drawn as a solid black line with dots superimposed at each iteration $k$.  
The bound (EV) is drawn as a solid red line, (FOV) with a broken blue line, 
and (PSA) with gray lines for various values of $\eps$.

%%%%%%%%%%%%%%%%%%%%%%%%%%%%%%%%%%%%%%%%%%%%%%%%%%%%%%%%%%%%%%%%%%%%%%%%%%%%%%%%
\exhead{Example A: All descriptive}  
%%%%%%%%%%%%%%%%%%%%%%%%%%%%%%%%%%%%%%%%%%%%%%%%%%%%%%%%%%%%%%%%%%%%%%%%%%%%%%%%

All bounds accurately describe GMRES convergence for a scalar
multiple of the identity,
\[ \BA = \alpha @\BI, \rlap{$\quad \alpha\in\C\setminus\{0\}.$}\]
Since $\BA$ is normal with a single eigenvalue, 
$\SPEC(\BA) = \FOV(\BA) = \{\alpha\}$ and 
$\PSA(\BA) = \{\alpha\}+\DISK{\eps}$.  The approximation
problems in (EV) and (FOV) are on singleton sets, and thus
these bounds ensure convergence in one iteration;
the associated constants are $\cev=1$ and $\cfov = 1+\sqrt{2}$.  
The pseudospectral bound also gives convergence in a single 
iteration, but the explanation is a bit more elaborate.
The constant term in (PSA) is $\cpsa=1$ for all $\eps$;
with its approximation problem on the disk $\alpha+\DISK{\eps}$, 
(PSA) gives
\[ \mingmres \norm{p(\BA)} \le (\eps/\alpha)^k, 
   \rlap{\qquad for all $\eps\in(0,\alpha)$,}\]
implying, as $\eps\to 0$, 
convergence to any given tolerance in a single iteration.

%%%%%%%%%%%%%%%%%%%%%%%%%%%%%%%%%%%%%%%%%%%%%%%%%%%%%%%%%%%%%%%%%%%%%%%%%%%%%%%%
% FIGURES AND TABLE FOR EXAMPLE A
%%%%%%%%%%%%%%%%%%%%%%%%%%%%%%%%%%%%%%%%%%%%%%%%%%%%%%%%%%%%%%%%%%%%%%%%%%%%%%%%
\begin{figure}[b!]
\begin{center}
\includegraphics[scale=0.57]{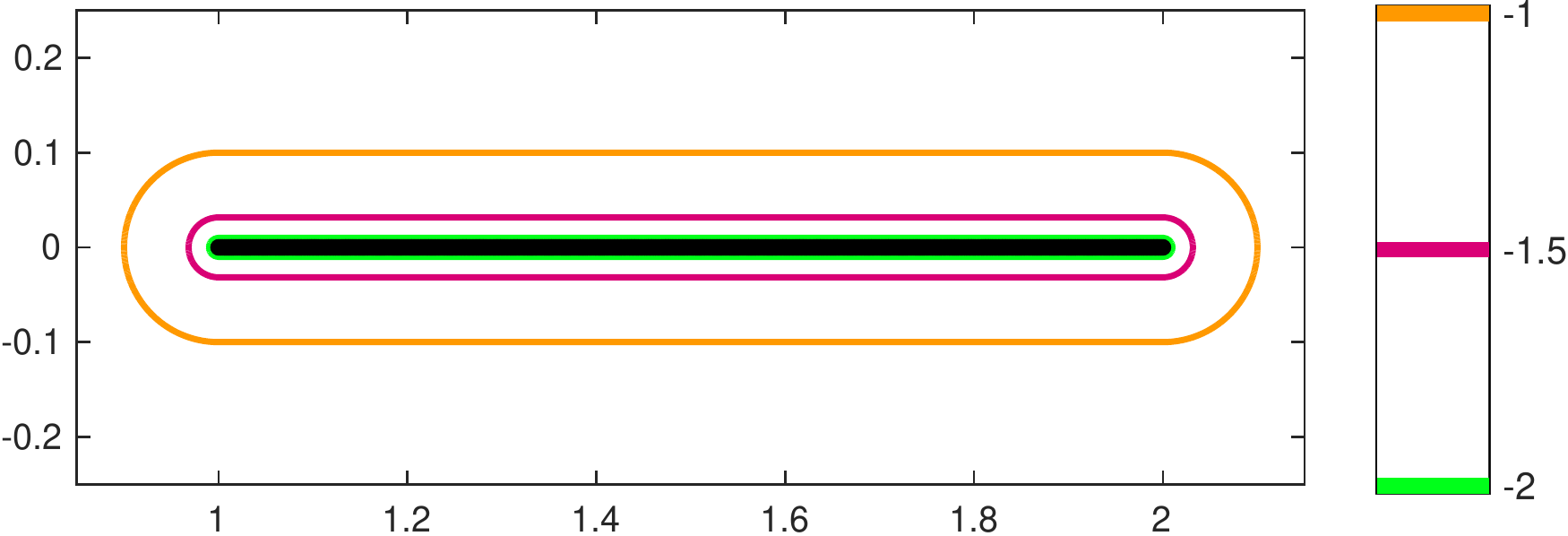}

\vspace*{.5em}
\includegraphics[scale=0.57]{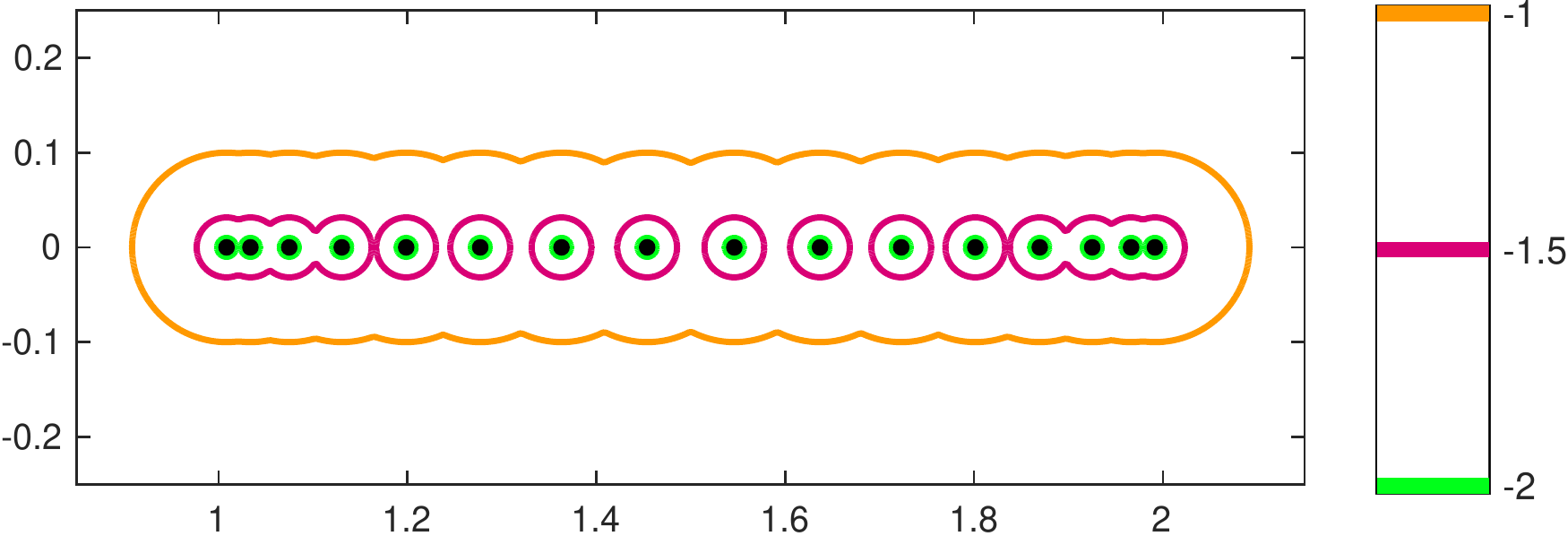}

\begin{picture}(0,0)
\put(-120,200){\small $N=\infty$}
\put(-120,95){\small $N=16$}
\end{picture}
\end{center}

\vspace*{-15pt}
\caption{\label{fig:psaA}
Spectra $($black dots in {$[1,2]$}$)$ and $\eps$-pseudospectra $(${$\eps=10^{-1}$, 
$10^{-1.5}$, $10^{-2}$}$)$ for the operator~(\ref{eq:Atoep}) $($top$)$
with $[a,b]=[1,2]$, and its finite section of dimension $N=16$ $($bottom$)$.
(In this and all following plots of $\PSA(\BA)$, the color bar denotes 
$\log_{10}(\eps)$, and the corresponding colored lines show the boundary
of $\PSA(\BA)$.)
}
\end{figure}
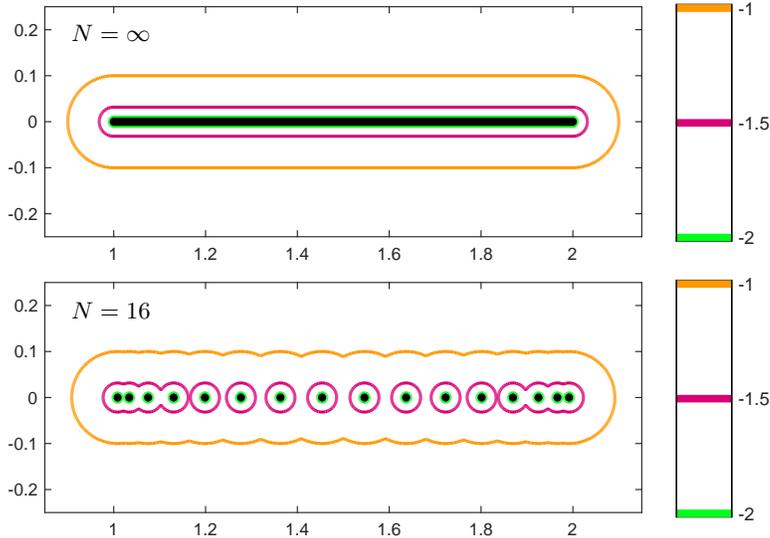

%%%%%%%%%%%%%%%%%%%%%%%%%%%%%%%%%%%%%%%%%%%%%%%%%%%%%%%%%%%%%%%%%%%%%%%%%%%%%%%%
\begin{figure}[t!]
\begin{center}
\includegraphics[scale=0.57]{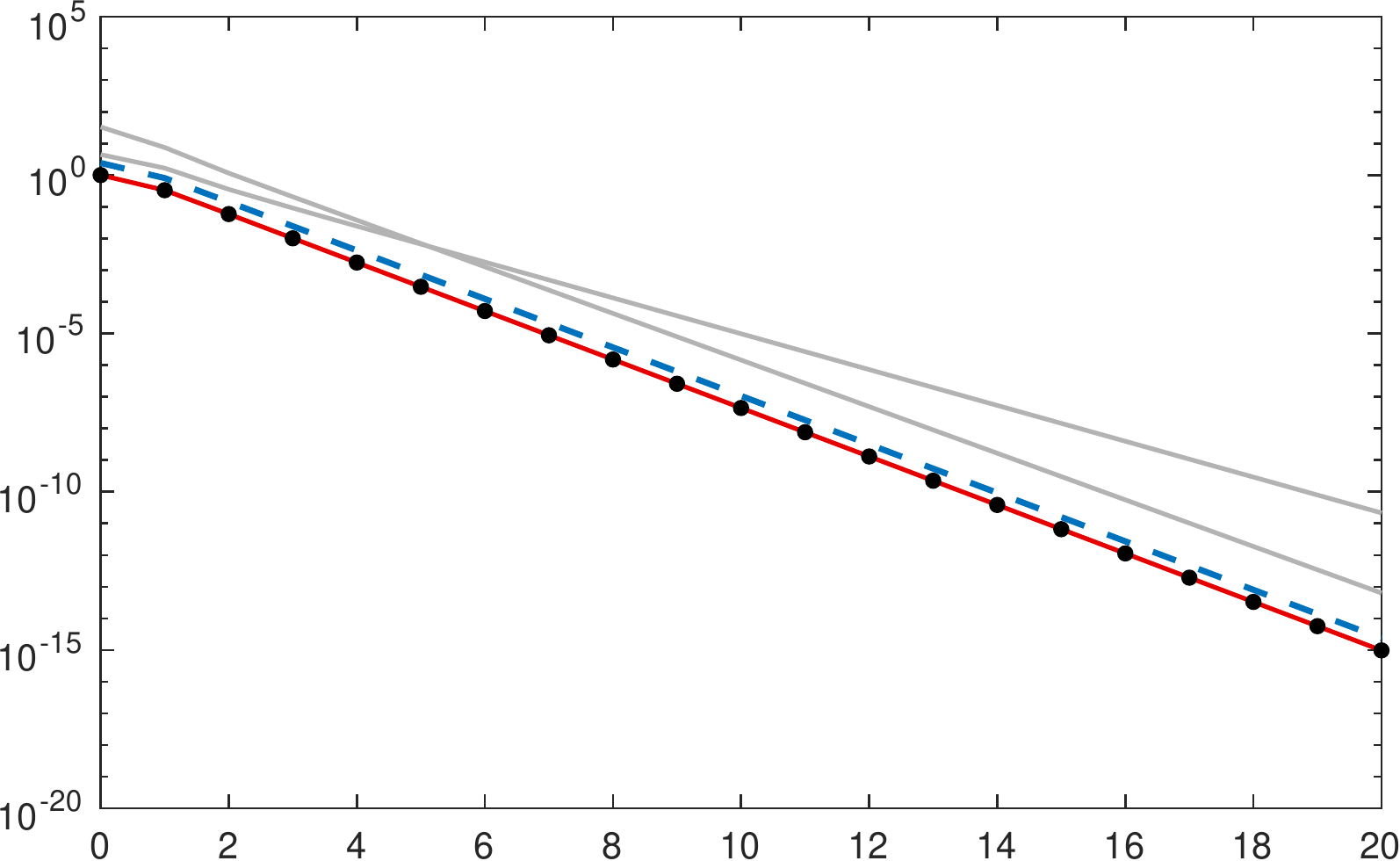}

\begin{picture}(0,0)
 \put(-184,125){\small $\displaystyle{\mingmres\! \norm{p(\BA)}}$}
 \put(-14,0){\small iteration, $k$}
 \put(130,75){\small $\eps=10^{-1}$}
 \put(130,61){\small $\eps=10^{-2}$}
 \put(130,52){\small (FOV)}
 \put(130,42){\small (EV), exact}
\end{picture}
\end{center}

\vspace*{-5pt}
\caption{\label{fig:Abounds}
Convergence bounds for the normal Toeplitz operator~$(\ref{eq:Atoep})$
with $\SPEC(\BA)=[1,2]$.}
\end{figure}
%%%%%%%%%%%%%%%%%%%%%%%%%%%%%%%%%%%%%%%%%%%%%%%%%%%%%%%%%%%%%%%%%%%%%%%%%%%%%%%%

All three bounds also perform well if the single eigenvalue is
expanded to a positive real interval $[a,b]\not\ni 0$, provided $\BA$
remains normal.  To get $\SPEC(\BA)=[a,b]$ requires an 
operator on an infinite dimensional space.  For example, 
set $\alpha:=(a+b)/2$ and $\beta := (b-a)/4$.  Then 
the tridiagonal Toeplitz operator
\begin{equation} \label{eq:Atoep}
\BA = \left[\begin{array}{cccccc}
              \ddots & \ddots  \cr \ddots & \alpha & \beta \cr 
              & \beta & \alpha & \beta & \phantom{\ddots}\cr 
              & & \beta & \alpha & \ddots  \cr
              & &  & \ddots & \ddots 
       \end{array}\right]
\end{equation}
on $\ell^2(\Z)$ is self-adjoint (and hence normal) with the 
spectrum and field of values 
$\SPEC(\BA) = \FOV(\BA) = [a,b]$; see, e.g., \cite{BG05}.
(Finite sections of this operator will have eigenvalues 
distributed across $(a,b)$.)
For the infinite dimensional operator, (EV) and (FOV) give the 
same convergence rate,  
\[ \rev = \rfov = {\sqrt{b/a}-1 \over \sqrt{b/a}+1},\]
with $\cev=1$ and $\cfov=1+\sqrt{2}$ as before.  
However, the bound (PSA) becomes slightly less
accurate, as $\PSA(\BA) = [a,b] + \DISK{\eps}$ 
consists of the interval $[a,b]$ surrounded by a border of radius $\eps$
(see Figure~\ref{fig:psaA}, top).  
The constant $\cpsa$ now involves a length scale,
$\cpsa(\eps) = (b-a)/(\pi\eps) + 1$.  Descriptive pseudospectral bounds
require one to balance the more accurate convergence rates obtained for
small $\eps$ against the growth of the constant $\cpsa(\eps)$. 
Figure~\ref{fig:Abounds} illustrates this situation for $\SPEC(\BA)=[a,b] = [1,2]$.

The pseudospectra shown in Figure~\ref{fig:psaA} highlight a subtle practical
aspect of applying the (PSA) bound.  
To compute (PSA) for the case $N=16$, one might prefer to use the 
upper bound on $\PSA(\BA)$ obtained from the $N=\infty$ case 
(a geometrically simpler set).  
Replacing $\PSA(\BA)$ by this upper
bound will slow the asymptotic convergence rate ever so slightly,
but could reduce the boundary length, and hence $\cpsa(\eps)$.

%%%%%%%%%%%%%%%%%%%%%%%%%%%%%%%%%%%%%%%%%%%%%%%%%%%%%%%%%%%%%%%%%%%%%%%%%%%%%%%%
\exhead{Example B: None descriptive}  
%%%%%%%%%%%%%%%%%%%%%%%%%%%%%%%%%%%%%%%%%%%%%%%%%%%%%%%%%%%%%%%%%%%%%%%%%%%%%%%%

As described in Section~\ref{sec:survey}, each of the (EV), (FOV), and (PSA) bounds can be deceived by low-dimensional nonnormality.  
Example~B illustrates this shortcoming:%
\footnote{This example is essentially an extreme version of the 
diagonalizable example constructed by Greenbaum and Strako\v{s}~\cite{gs94} 
to demonstrate the failure of the pseudospectral bound (PSA).}
\begin{equation} \label{eq:B}
\BA = \left[\begin{array}{cc}1&\alpha\\ 0 & 1\end{array}\right]
       \oplus \widehat{\BLambda},\qquad \alpha\gg2,
\end{equation}
where $\widehat{\BLambda}$ is a diagonal matrix with 
uniformly distributed entries in the positive real
interval $[1,b]$ for $b>1$.
When $\alpha$ is large, the $2\times 2$ Jordan bock will dominate
$\FOV(\BA)$ and $\PSA(\BA)$ even for some fairly small values of $\eps$.
(Figure~\ref{fig:WpsaB} shows an example.)
A polynomial $p(z)$ with two roots at $z=1$ 
would annihilate this Jordan block, 
leaving a normal matrix with eigenvalues in the interval $[1,b]$ to 
handle at later iterations.
This structure allows one to bound the norm of the GMRES residual
\emph{independent of $\alpha$}.
To see this, replace the optimal GMRES polynomial $p(z)$ with 
$(1-z)^2q(z)$ at iteration $k\ge 2$ to get the upper bound 
\begin{eqnarray}
\mingmres \norm{p(\BA)} 
 &\le& \mingmresq \|(\BI-\BA)^2 q(\BA)\|_2  
 \ =\ \mingmresq \|\Bzero \oplus (\BI-\widehat{\BLambda})^2 q(\widehat{\BLambda})\|_2   \nonumber
      \\[.25em]
 &\le& \mingmresq \max_{z\in [1,b]} |1-z|^2 |q(z)|  \nonumber\\[.25em]
 &\le& |1-b|^2 \mingmresq \max_{z\in[1,b]} |q(z)|
 \ \le\  2@|1-b|^2 \left(\frac{\sqrt{b}-1}{\sqrt{b}+1}\right)^{k-2}.
\label{eq:exBgmres}
\end{eqnarray}
This last step just uses Chebyshev approximation on the interval $[1,b]$.
Now if $b \le 1+1/\sqrt{2}$ (so $2@|1-b|^2\le 1$), 
\[\mingmres \norm{p(\BA)} \le \left(\frac{\sqrt{b}-1}{\sqrt{b}+1}\right)^{k-2}.\] This bound ensures convergence to the tolerance \textsc{tol} in
$\lceil 2+\log(\mbox{\sc tol})/\log(\rho)\rceil $ iterations,
where $\rho = (\sqrt{b}-1)/(\sqrt{b}+1)$, as given in Table~\ref{table:all}.

%%%%%%%%%%%%%%%%%%%%%%%%%%%%%%%%%%%%%%%%%%%%%%%%%%%%%%%%%%%%%%%%%%%%%%%%%%%%%%%%

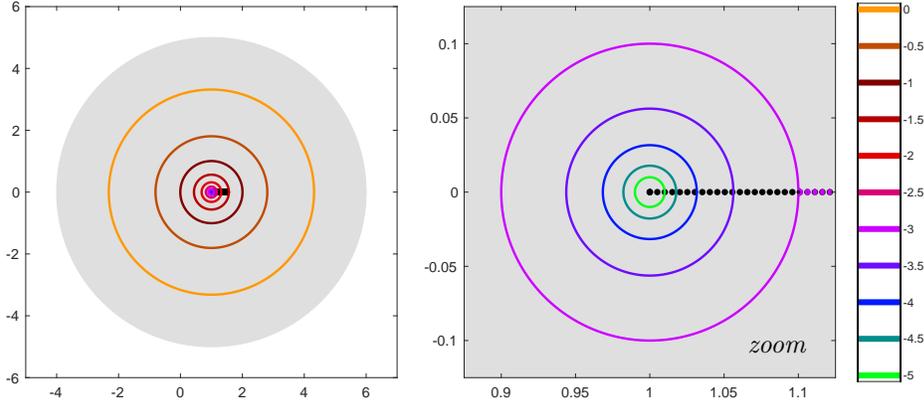
\begin{figure}[t!]
\begin{center}
\includegraphics[scale=0.45]{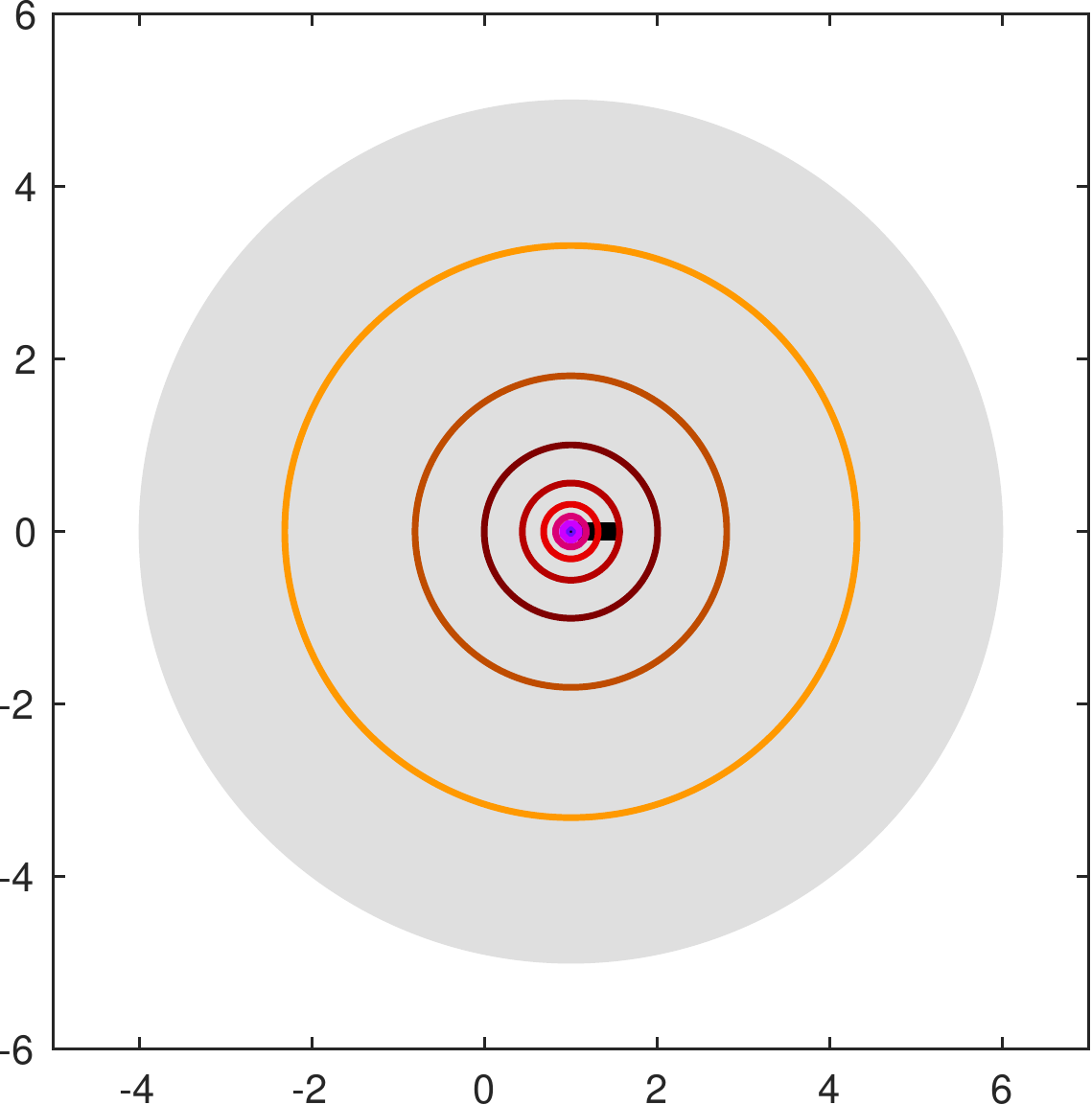}\quad
\includegraphics[scale=0.45]{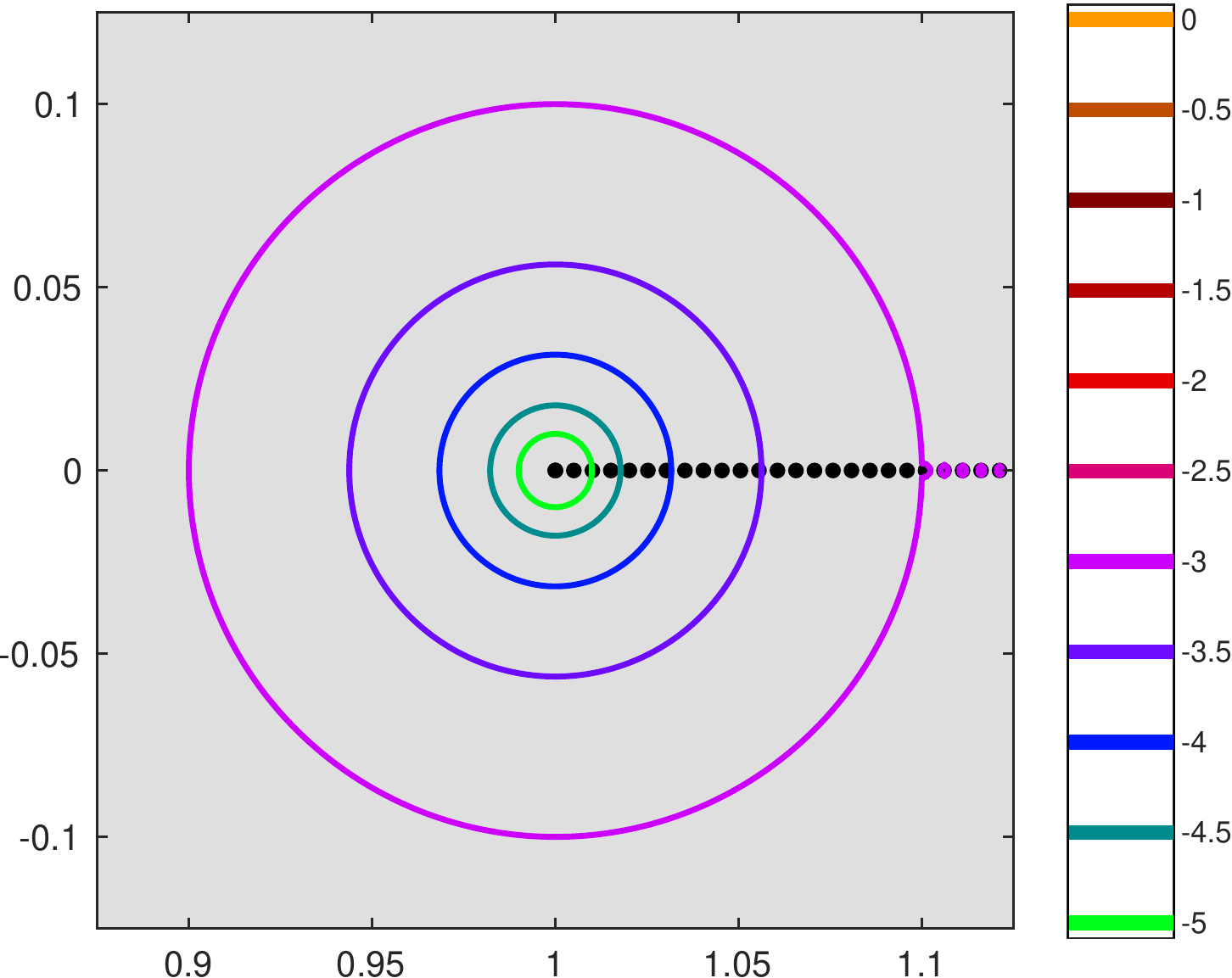}
\begin{picture}(0,0)
\put(-70,17){\emph{zoom}}
\end{picture}
\end{center}

\caption{\label{fig:WpsaB}
Field of values $($gray disk$)$, $\eps$-pseudospectra $(\eps=10^{-5},\ldots, 10^{0})$
and eigenvalues $($small dots$)$ for Example~B with $\alpha=10$,  $b=3/2$, and $N=102$.  
The $2\times 2$ Jordan block determines $\FOV(\BA)$ and dominates $\PSA(\BA)$ for
the $\eps$ values shown, exerting strong influence on the associated convergence bounds.
Nevertheless, this Jordan block does not much delay GMRES convergence.
}
\end{figure}
%%%%%%%%%%%%%%%%%%%%%%%%%%%%%%%%%%%%%%%%%%%%%%%%%%%%%%%%%%%%%%%%%%%%%%%%%%%%%%%%

Since $\BA$ is nondiagonalizable, $\cev=\infty$ and
(EV) does not apply.  
The field of values of $\BA$ grows ever larger with $\alpha$.  
Note that
\[ \FOV\bigg(\left[\begin{array}{cc} 1 & \alpha \\ 0 & 1 \end{array}\right]\bigg) = 1 + \CDISK{\alpha/2}.\]
Thus if $b \le \alpha/2$, the normal eigenvalues 
on the diagonal of $\widehat{\BLambda}$ are all embedded within
the field of values of the $2\times 2$ Jordan block, and hence
\[ \FOV(\BA) = 1 + \CDISK{\alpha/2}, \rlap{\qquad $b \in (1,\alpha/2]$.} \]
(The left plot of Figure~\ref{fig:WpsaB} shows this scenario.)
In any case, $1+\CDISK{\alpha/2} \subseteq \FOV(\BA)$, and so if
$\alpha\ge 2$, then $0\in \FOV(\BA)$ and (FOV)~cannot give convergence.

Analysis of the pseudospectral bound is more involved. 
In Example~F we shall see that (PSA) accurately predicts convergence
for a Jordan block by taking $\eps$ sufficiently small.
The $\widehat{\BLambda}$ component in the present example~(\ref{eq:B})
adds the crucial complication: 
as seen in Example~A, for Hermitian matrices
the bound (PSA) is best when $\eps$ is relatively large 
($10^{-1}$ and $10^{-2}$ in Figure~\ref{fig:Abounds}).
By explicitly computing the norm of the resolvent, one can show that
\begin{equation} \label{eq:psaJ2}
 1 + \Delta_{\sqrt{\alpha \eps + \eps^2}}
    \ =\  
   \PSA\bigg(\left[\begin{array}{cc} 1 & \alpha \\ 0 & 1 \end{array}\right]\bigg) 
    \ \subseteq\  
   \PSA(\BA).
\end{equation}
Since $\PSA(\BA)$ must contain the disk $1+\Delta_{\sqrt{\alpha\eps+\eps^2}}$, 
the boundary of $\PSA(\BA)$ must encircle this set, ensuring that, for any $\eps>0$,
\[ \cpsa(\eps)\ge {2\pi \sqrt{\alpha \eps + \eps^2} \over 2\pi \eps} \ge \sqrt{\alpha/\eps}.\]
Moreover, equation~(\ref{eq:psaJ2}) implies that 
\[ 0 \in \PSA(\BA) \rlap{\quad \mbox{for all $\eps>{1\over2}\big(\sqrt{\alpha^2+4}-\alpha\big)$},}\] 
i.e., $\eps>{1\over2}\big(\sqrt{\alpha^2+4}-\alpha\big) \approx 1/\alpha$ 
for large $\alpha$: 
thus ruling out the large values of $\eps$ that could control $\cpsa(\eps)$.
By increasing $\alpha$, we can make $\cpsa(\eps)$ arbitrarily large,
even though~(\ref{eq:exBgmres}) bounds GMRES convergence independent of $\alpha$.  
On the other hand, the $\widehat{\BLambda}$ block gives a 
nonzero asymptotic rate of convergence for any $\eps>0$ 
(unlike Example~F, where a single eigenvalue will permit 
arbitrarily fast convergence rates as $\eps\to 0$).
Thus in Table~\ref{table:all} we say that (PSA) predicts
infinitely many iterations.

%%%%%%%%%%%%%%%%%%%%%%%%%%%%%%%%%%%%%%%%%%%%%%%%%%%%%%%%%%%%%%%%%%%%%%%%%%%%%%%%
\exhead{Example C: Only (EV) descriptive}  
%%%%%%%%%%%%%%%%%%%%%%%%%%%%%%%%%%%%%%%%%%%%%%%%%%%%%%%%%%%%%%%%%%%%%%%%%%%%%%%%
Aside from trivial cases, the field of values and pseudospectral bounds 
both involve approximation problems on regions in the complex plane that give
nonzero asymptotic convergence rates $\rfov$ and $\rpsa(\eps)$.
However, these sets could be unduly influenced by 
parts of the spectrum that could be effectively eliminated 
at an early stage of the GMRES iteration.  When a few eigenvalues are far 
from the rest of the spectrum, (EV)~can be more descriptive than 
(PSA) and (FOV), because the approximation problem in (EV) is posed 
on a discrete point set, and isolated outliers do not influence $\rev$.
Define
\[ \BA = \delta \oplus \widehat{\BLambda}, \rlap{$\qquad \delta\in\R$,} \]
where $\widehat{\BLambda}$ is a diagonal matrix with
entries uniformly distributed in the real, positive interval $[a,b]$,
and $0<|\delta|\ll a < b$.

Since $\BA$ is normal, convergence is determined by the spectrum.
The bound (EV), with $\cev=1$, is exact.  This convergence can be
bounded using the polynomial $p_k(z)=(1-z/\delta) q_{k-1}(z)$, where
$q_{k-1}$ is the optimal degree-$k-1$ residual polynomial for the 
interval $[a,b]$.  The eigenvalue $\delta$ near the origin
causes an initial stagnation~\cite{dtt98}; the polynomial
$p_k$ suggests that this plateau will last no longer than the
number of iterations it takes for $q_{k-1}(z)$ to overcome $(1-z/\delta)$ 
for $z\in[a,b]$, i.e.,
\[ \Bigg\lceil 1 + \frac{\log|\delta| - \log 2(b-\delta)}{\log(\rev)}\Bigg\rceil \,\mbox{ iterations},\]
where 
\[ \rev = {\sqrt{b/a}-1 \over \sqrt{b/a}+1}\]
is the asymptotic convergence rate associated with the interval $[a,b]$.
Figure~\ref{fig:Cbounds} shows the stagnation caused by the eigenvalue near the origin.

%%%%%%%%%%%%%%%%%%%%%%%%%%%%%%%%%%%%%%%%%%%%%%%%%%%%%%%%%%%%%%%%%%%%%%%%%%%%%%%%
% FIGURE AND TABLE FOR EXAMPLE C
%%%%%%%%%%%%%%%%%%%%%%%%%%%%%%%%%%%%%%%%%%%%%%%%%%%%%%%%%%%%%%%%%%%%%%%%%%%%%%%%

\begin{figure}[b!]
\begin{center}
\vspace*{2em}
\includegraphics[scale=0.57]{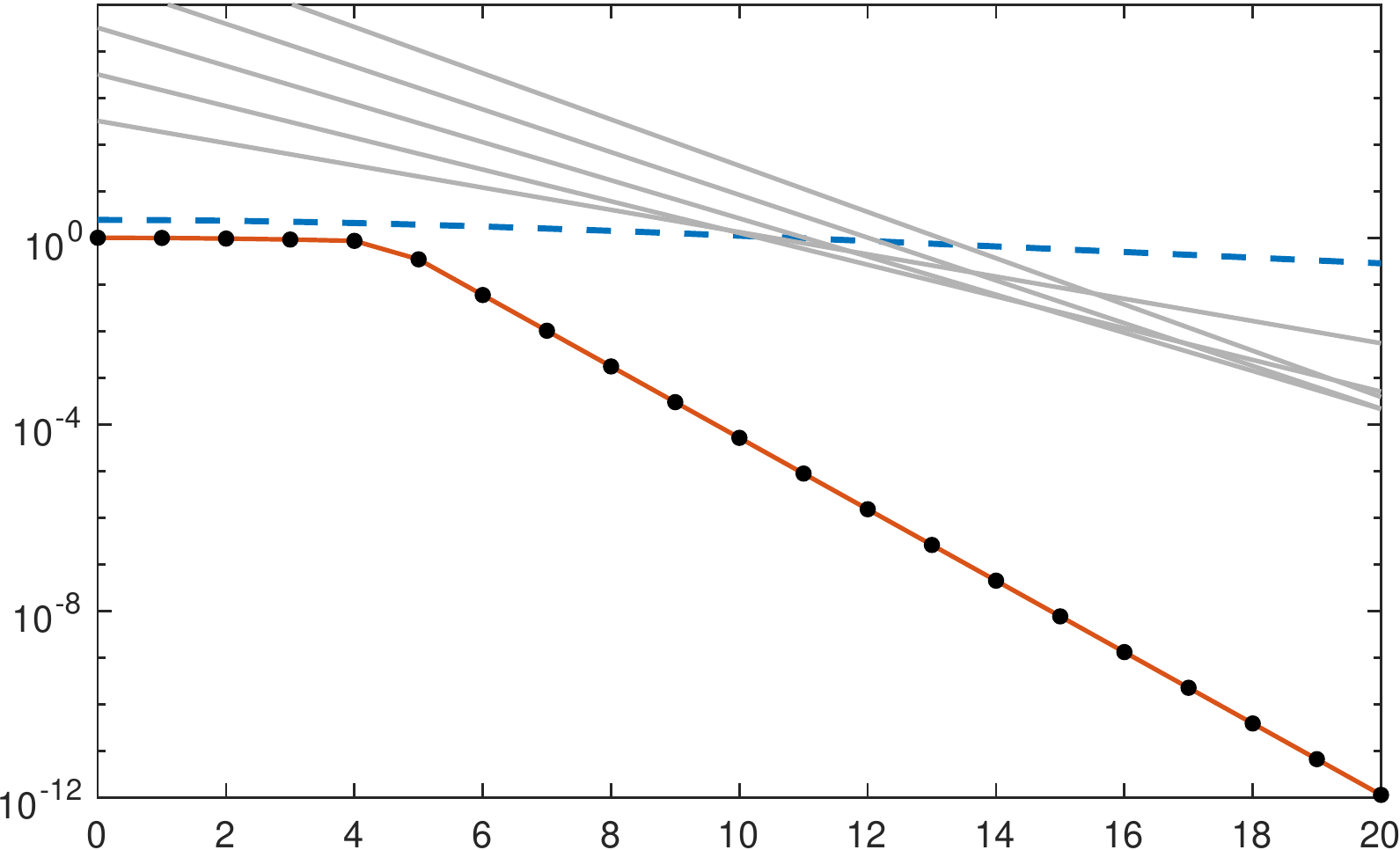}

\begin{picture}(0,0)
 \put(-180,110){\small $\displaystyle{\mingmres\! \norm{p(\BA)}}$}
 \put(-10,0){\small iteration, $k$}
 \put(-106,133){\small (FOV)}
% \put(130,125){\small $\eps=\infty$}
 \put(-151,141){\small $\eps=10^{-3}$}
 \put(-151,152){\small $\eps=10^{-4}$}
 \put(-151,164){\small $\eps=10^{-5}$}
 \put(-104,172){\rotatebox{30}{\small $\eps=10^{-6}$}}
 \put(-81,172){\rotatebox{30}{\small $\eps=10^{-7}$}}
 \put(6,95){\rotatebox{-30}{\small (EV), exact}}
\end{picture}
\end{center}

\vspace*{-5pt}
\caption{\label{fig:Cbounds}
Convergence bounds for Example~C with $\delta=0.01$ and $n=\infty$, 
using underestimates of $\PSA(\BA)$ for the approximation problem in (PSA).}
\end{figure}
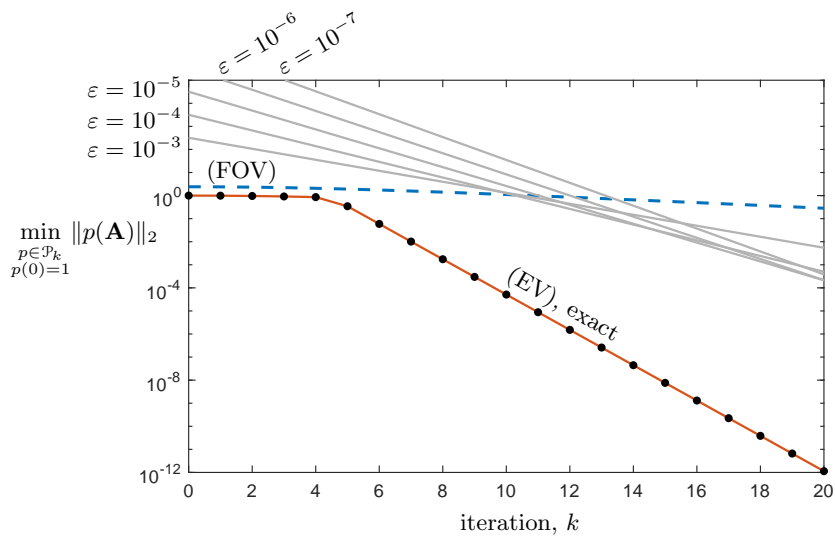

%%%%%%%%%%%%%%%%%%%%%%%%%%%%%%%%%%%%%%%%%%%%%%%%%%%%%%%%%%%%%%%%%%%%%%%%%%%%%%%%

If $\delta < 0 < a < b$, then $\FOV(\BA) = [\delta,b]$ contains the origin,
and hence $\rfov=1$, and so (FOV) does not predict any convergence.
If $0 < \delta < a < b$, then $\FOV(\BA) = [\delta,b]$ does not contain
the origin, and the asymptotic convergence rate 
\[ \rfov = {\sqrt{b/\delta}-1 \over \sqrt{b/\delta}+1}\]
will be close to one when $0<\delta \ll a < b$.
Provided $b$ is not too large, $\rev \ll \rfov$: 
(FOV) predicts slow convergence,
accurately describing the initial period of stagnation but missing the
transition to more rapid asymptotic convergence.

The pseudospectral bound suffers from the fact that it cannot 
treat $\delta$ as a single simple eigenvalue, eliminated at an
early stage of convergence.  Moreover, to give convergence 
we need $0\not \in \PSA(\BA)$, requiring $0 < \eps < |\delta|$. 
For any $\eps\in(0,(a-\delta)/2)$, 
the component of $\PSA(\BA)$ about $\lambda=\delta$ is a disk of radius $\eps$,
and this disk will always influence the asymptotic convergence rate 
associated with $\PSA(\BA)$.
This effect diminishes as $\eps$ decreases, but such small
values of $\eps$ will give large constants
($\cpsa(\eps) = 2+(b-a)/(\pi\eps)$ for $n=\infty$) due to the 
interval $[a,b]$. 

This case is the least analytically compelling of our six examples.  
In particular, it is difficult to cleanly describe 
the convergence rates associated with (PSA).\ \ 
In Figure~\ref{fig:Cbounds}, we illustrate (PSA) for 
$\delta=0.01$, $[a,b] = [1,2]$, and $n=\infty$.  
Since we argue that (PSA) gives poor bounds for this scenario, 
we do not want our portrayal of the bounds to suffer 
from an overestimate $\PSA(\BA)$ that would yield 
an easily computable asymptotic convergence rate 
that is not true to the bound.
Instead, Figure~\ref{fig:Cbounds} shows an \emph{underestimate} of (PSA)
using the slightly faster convergence rate 
for the union of two intervals, $[\delta-\eps, \delta+\eps] \cup [a-\eps, b+\eps] \subset \PSA(\BA)$.\ \ 
The asymptotic convergence rate $\rho(\delta,\eps,a,b)$ 
for the union of these two real intervals can be expressed 
in terms of elliptic integrals, as described and implemented in MATLAB by
Fischer~\cite{fischer96}.  The pseudospectral constants
are $\cpsa(\eps) = 2 + (b-a)/(\pi\eps)$.
In place of (PSA), Figure~\ref{fig:Cbounds} 
shows the underestimates $\cpsa(\eps) \rho(\delta,\eps,a,b)^k$.

Of course, one could adapt (FOV) and (PSA) to handle a single outlying 
eigenvalue by splitting up the spectrum, as in Theorem~\ref{gensp};
however, many problems present a range of outlying eigenvalues at different
scales, which are more difficult to identify and handle.

%%%%%%%%%%%%%%%%%%%%%%%%%%%%%%%%%%%%%%%%%%%%%%%%%%%%%%%%%%%%%%%%%%%%%%%%%%%%%%%%
\exhead{Example D: Only (EV) not descriptive}  
%%%%%%%%%%%%%%%%%%%%%%%%%%%%%%%%%%%%%%%%%%%%%%%%%%%%%%%%%%%%%%%%%%%%%%%%%%%%%%%%

The bound (EV) fails for any nondiagonalizable matrix,
yet defectiveness need not imply poor GMRES convergence.
Consider a small perturbation to the identity matrix,
\begin{equation} \label{eq:exD}
  \BA = \left[\begin{array}{cccc}
                1 & \delta & & \\
                  & 1 & \ddots & \\
                  &   & \ddots & \delta \\
                  &   & & 1
          \end{array}\right],\rlap{$\qquad 0<\delta\ll 1$.}
\end{equation}
This matrix is completely defective for all $\delta\ne 0$,
but small values of $\delta$ exert only the slightest impact
on convergence.  Unlike the case of $\delta=0$, 
for most initial residuals GMRES will  require $n$ iterations 
to converge \emph{exactly}; however, for small $\delta>0$ 
GMRES will make excellent progress at each step.
(For bounds describing how small changes to $\BA$ affect
the convergence of GMRES for a fixed $\Br_0$, see~\cite{SEM13}.)
If $\Br_0$ is the $n$th column of the identity matrix, 
Ipsen~\cite{Ips00} derived the exact formula
\begin{equation} \label{Dlower}
 \frac{\norm{\Br_k}}{\norm{\Br_0}} = \delta^k 
   \sqrt{\frac{1-\delta^2}{1-\delta^{2(k+1)}}}.
\end{equation}
For small $\delta$ this convergence must be close to the worst case, since
\[ \mingmres \norm{p(\BA)} \le \|(\BI-\BA)^k\| = \delta^k.\] 

%%%%%%%%%%%%%%%%%%%%%%%%%%%%%%%%%%%%%%%%%%%%%%%%%%%%%%%%%%%%%%%%%%%%%%%%%%%%%%%%
% FIGURE FOR EXAMPLE D
%%%%%%%%%%%%%%%%%%%%%%%%%%%%%%%%%%%%%%%%%%%%%%%%%%%%%%%%%%%%%%%%%%%%%%%%%%%%%%%%

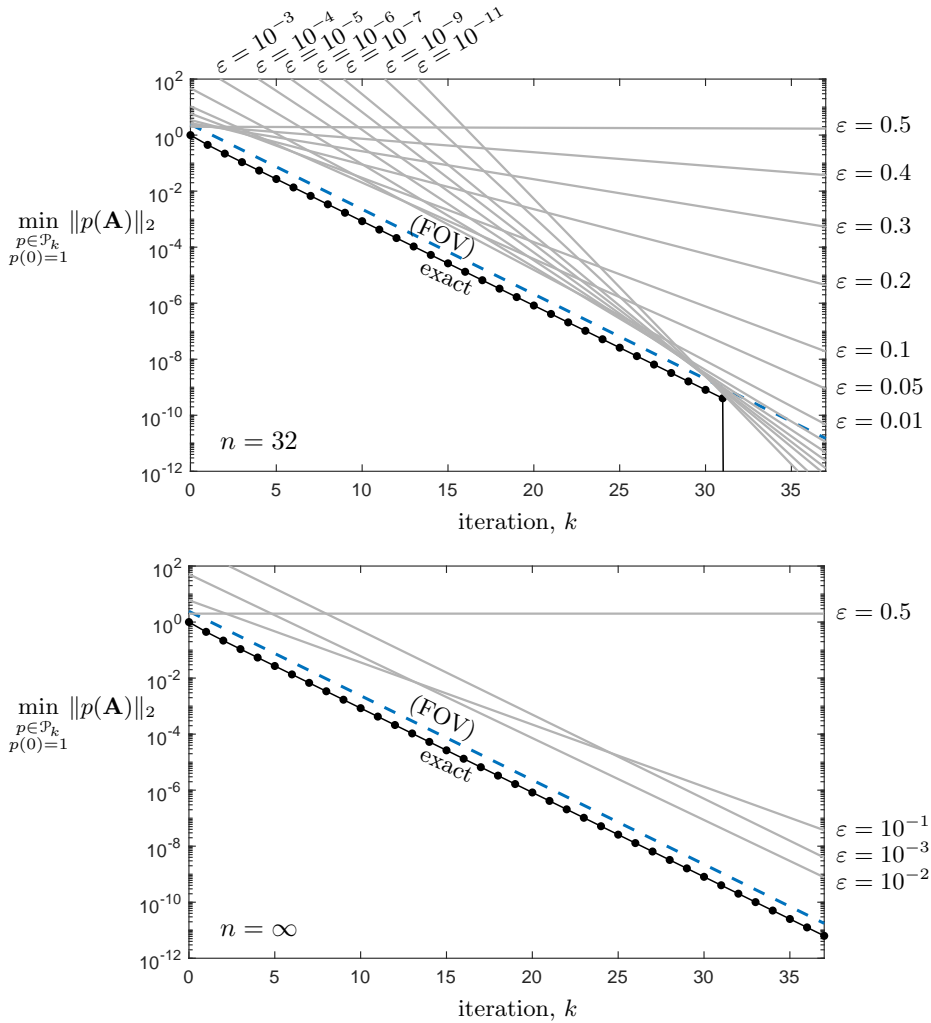
\begin{figure}[b!]  
\begin{center}
\vspace*{2em}
\includegraphics[scale=0.57]{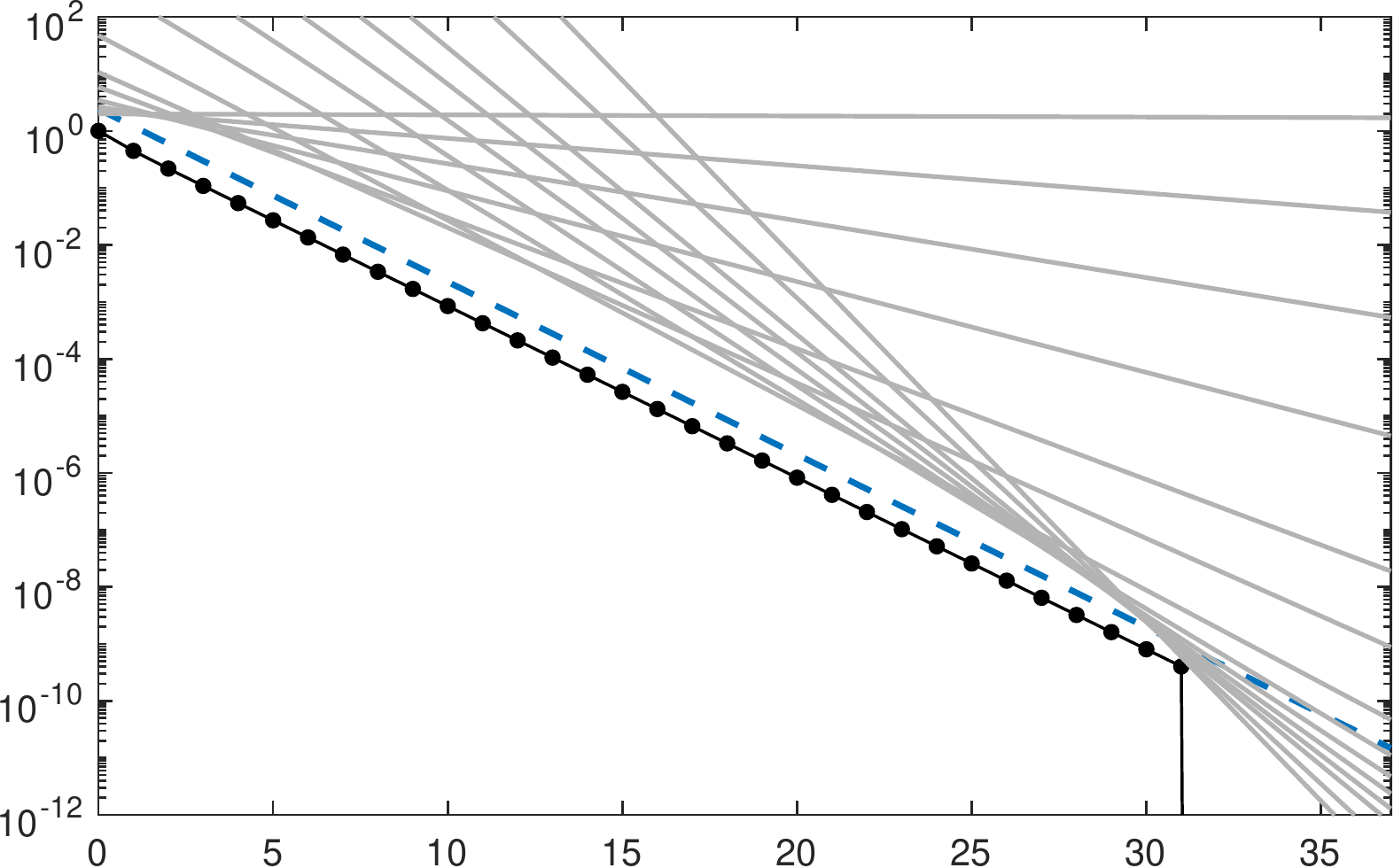}

\begin{picture}(0,0)
 \put(-180,114){\small $\displaystyle{\mingmres\! \norm{p(\BA)}}$}
 \put(-10,0){\small iteration, $k$}
 \put( 133,150){\small $\eps=0.5$}
 \put( 133,132){\small $\eps=0.4$}
 \put( 133,112){\small $\eps=0.3$}
 \put( 133, 91){\small $\eps=0.2$}
 \put( 133, 65){\small $\eps=0.1$}
 \put( 133, 51){\small $\eps=0.05$}
 \put( 133, 38){\small $\eps=0.01$}
 \put(-28,172){\rotatebox{30}{\small $\eps=10^{-11}$}}
 \put(-40,172){\rotatebox{30}{\small $\eps=10^{-9}$}}
 \put(-55,172){\rotatebox{30}{\small $\eps=10^{-7}$}}
 \put(-66,172){\rotatebox{30}{\small $\eps=10^{-6}$}}
 \put(-78,172){\rotatebox{30}{\small $\eps=10^{-5}$}}
 \put(-89,172){\rotatebox{30}{\small $\eps=10^{-4}$}}
 \put(-104,172){\rotatebox{30}{\small $\eps=10^{-3}$}}
 \put(-30,114){\rotatebox{-25}{\small (FOV)}}
 \put(-25,97){\rotatebox{-25}{\small exact}}
 \put(-100,30){$n=32$}
\end{picture}

\vspace*{10pt}
\includegraphics[scale=0.57]{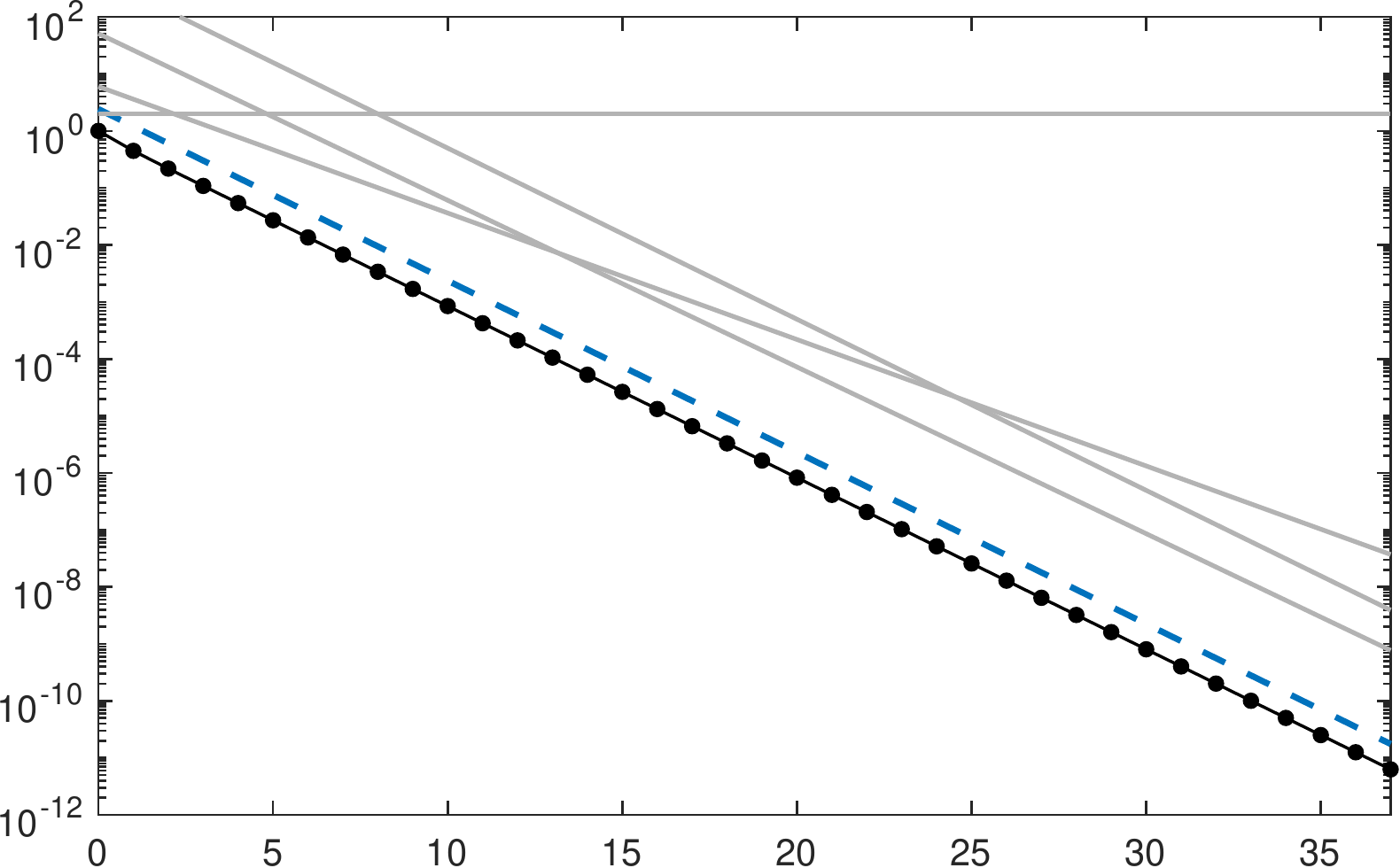}

\begin{picture}(0,0)
 \put(-180,114){\small $\displaystyle{\mingmres\! \norm{p(\BA)}}$}
 \put(-10,0){\small iteration, $k$}
 \put( 133,150){\small $\eps=0.5$}
 \put( 133, 68){\small $\eps=10^{-1}$}
 \put( 133, 58){\small $\eps=10^{-3}$}
 \put( 133, 48){\small $\eps=10^{-2}$}
 \put(-30,116.5){\rotatebox{-26}{\small (FOV)}}
 \put(-25,97){\rotatebox{-25}{\small exact}}
 \put(-100,30){$n=\infty$}
\end{picture}
\end{center}

\vspace*{-3pt}
\caption{Convergence bounds for Example~D with $\delta=1/2$ 
for $n=32$ and $n=\infty$.}
\label{Dbounds}
\end{figure}
%%%%%%%%%%%%%%%%%%%%%%%%%%%%%%%%%%%%%%%%%%%%%%%%%%%%%%%%%%%%%%%%%%%%%%%%%%%%%%%%

%%%%%%%%%%%%%%%%%%%%%%%%%%%%%%%%%%%%%%%%%%%%%%%%%%%%%%%%%%%%%%%%%%%%%%%%%%%%%%%%
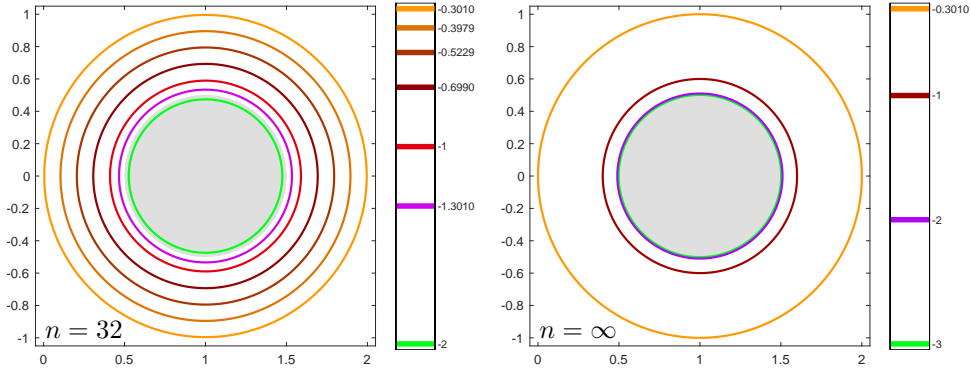
\begin{figure}[t!]
\begin{center}
\includegraphics[height=1.9in]{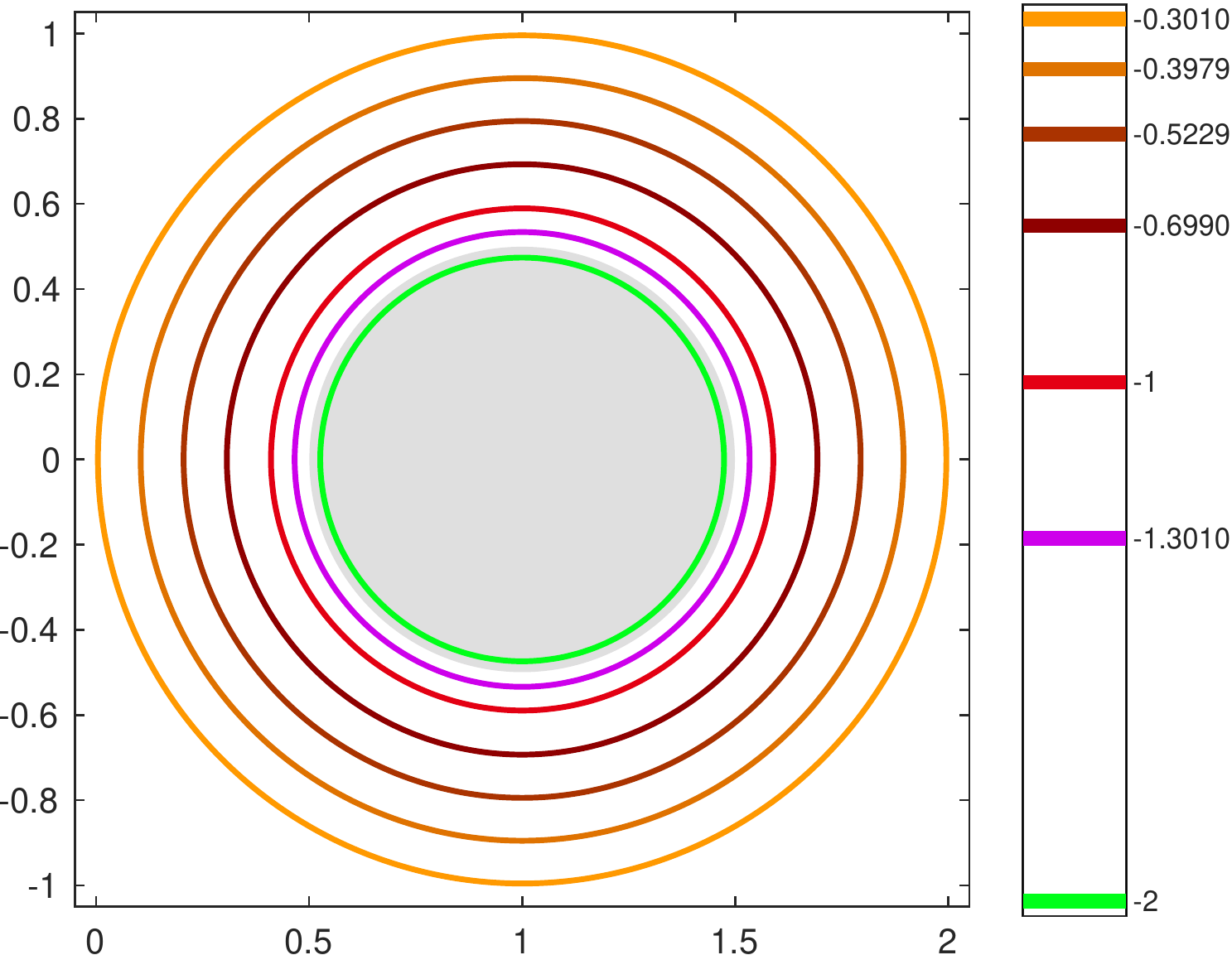}\quad
\includegraphics[height=1.9in]{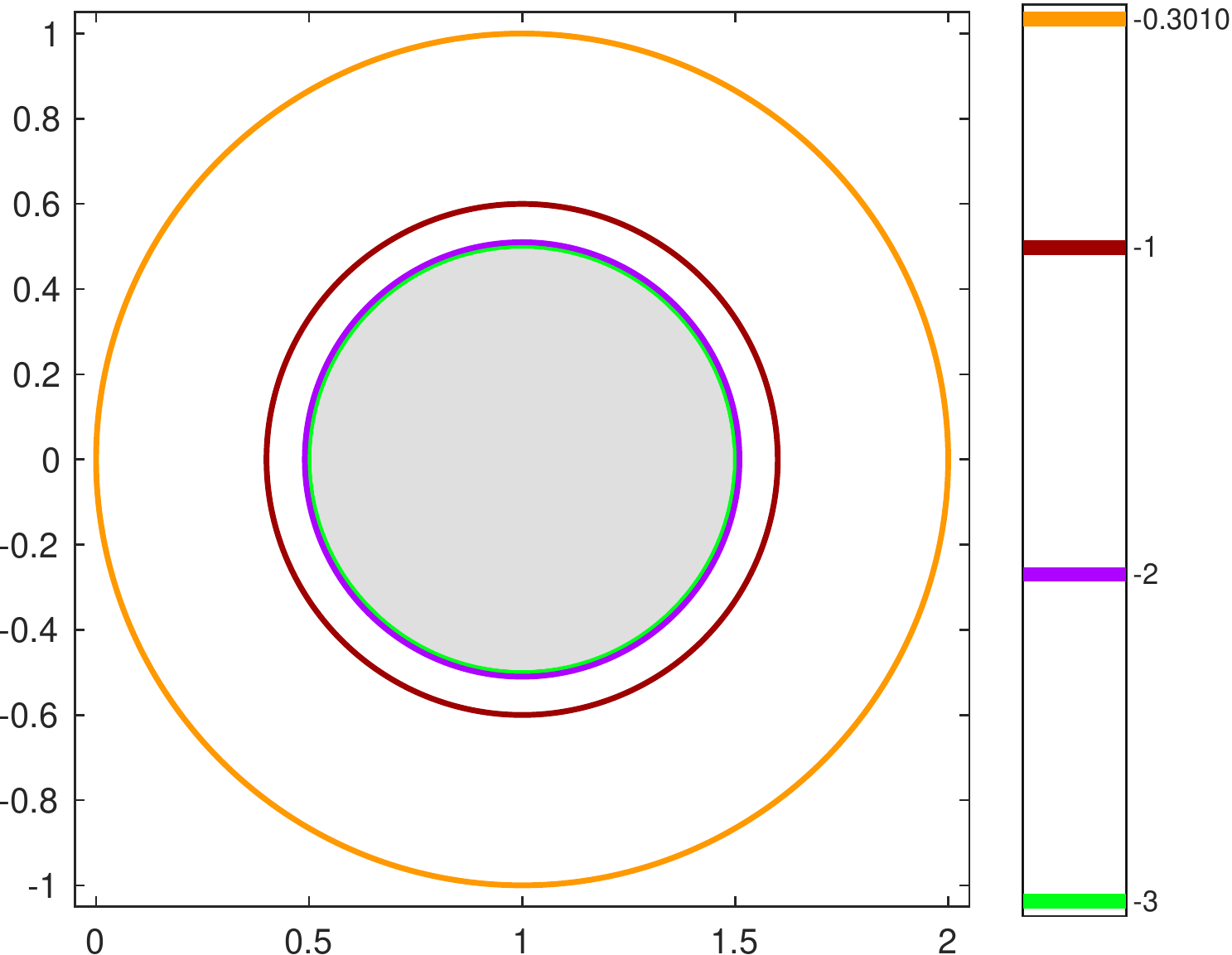}
\begin{picture}(0,0)
\put(-353,10){$n=32$}
\put(-166,10){$n=\infty$}
\end{picture}
\end{center}

\vspace*{-3pt}
\caption{Field of values $\FOV(\BA)$ and $\eps$-pseudospectra $\PSA(\BA)$
for Example~D with $\delta=1/2$.  
On the left, $n=32$ and $\eps=0.5$, $0.4$, $0.3$, $0.2$, $0.1$, $0.05$, and $0.01$;
on the right, $n=\infty$ and $\eps=0.5$, $10^{-1}$, $10^{-2}$, and $10^{-3}$ (the last of these is barely visible).}
\label{WpsaD}
\end{figure}
%%%%%%%%%%%%%%%%%%%%%%%%%%%%%%%%%%%%%%%%%%%%%%%%%%%%%%%%%%%%%%%%%%%%%%%%%%%%%%%%

Since $\BA$ is nondiagonalizable, (EV) fails to predict convergence.
The field of values is known explicitly for this example, 
$\FOV(\BA)=1+\DISK{\delta \cos(\pi/(n+1))}$~\cite[\S1.3]{gustrao},
leading to the exact formulation of (FOV): 
\[\mingmres \norm{p(\BA)} 
  \le \big(1+\sqrt{2}\big) \bigg(\delta \cos\Big({\pi\over n+1}\Big)\bigg)^k.\]
In the case of $n=\infty$, the infinite dimensional Toeplitz matrix
has spectrum $\SPEC(\BA) = 1+\DISK{\delta}$, and 
the closure of the field of values is $\overline{\FOV(\BA)} = 1 + \DISK{\delta}$; 
see~\cite[sect.~7.5]{BG05}.

The pseudospectra of $\BA$ are also disks~\cite{TE05},
but the radii of these disks are not known in closed form for general $n$.
In the limit $n\to\infty$, a theorem of Reichel and Trefethen~\cite{rt92}
shows that $\PSA(\BA) = 1+\DISK{\delta+\eps}$;
in this case, the bounds give:
\[ {\rm (FOV):}\ \mingmres \norm{p(\BA)} \le \big(1+\sqrt{2}\big)\,\delta^k \qquad
   {\rm (PSA):}\ \mingmres \norm{p(\BA)} \le 
                     (1+\delta/\eps) (\delta+\eps)^k.\]
In the limit as $\delta\to0$, both these bounds predict convergence
to arbitrary desired tolerance in a single iteration, though
the asymptotic rate $\rfov = \delta^k$ is slightly sharper
than the pseudospectral rate $\rpsa = (\delta+\eps)^k$.
One must balance the size of $\cpsa$ against
the accompanying convergence rate, just as for the normal
matrix with $\SPEC(\BA)\subseteq[a,b]$ discussed in Example~A.\ \ 
Figure~\ref{Dbounds} illustrates this example 
for $\delta=1/2$ with $n=32$ and $n=\infty$.
For the ``exact'' curve, we plot the lower bound given by
equation~(\ref{Dlower}).
The bound (PSA) is particularly interesting in the 
finite-dimensional case: For very small values
of $\eps$, (PSA) predicts convergence rates that are
too quick, associated with large constants that ensure
the bound does not intersect the convergence curve for $k<n$.
Figure~\ref{WpsaD} shows $\FOV(\BA)$ and $\PSA(\BA)$ for
$n=32$ and $n=\infty$.

Taking $\BA$ to be nondiagonalizable makes for a
clean example, but it is not necessary.
Perturbing the diagonal entries of $\BA$ from~$\lambda$ to
distinct nearby values, one can obtain finite but arbitrarily 
large values of $\cev=\COND(\BV)$, which will depend on the 
off-diagonal $\delta$ value; at the same time, taking
all the eigenvalues close to~1 makes $\rev$ arbitrary close to~0,
independent of $\delta$.  Yet $\delta$ can be seen to have a
crucial role in determining the asymptotic rate of convergence.

%%%%%%%%%%%%%%%%%%%%%%%%%%%%%%%%%%%%%%%%%%%%%%%%%%%%%%%%%%%%%%%%%%%%%%%%%%%%%%%%
\exhead{No Example: Only (FOV) descriptive} 
%%%%%%%%%%%%%%%%%%%%%%%%%%%%%%%%%%%%%%%%%%%%%%%%%%%%%%%%%%%%%%%%%%%%%%%%%%%%%%%%
Theorem~\ref{psafov} indicates that examples where (FOV) 
significantly outperforms (PSA) will be difficult to find.  
Since $\PSA(\BA)\subset\FOV(\BA)+\DISK{\eps}$, the rate
$\rpsa(\eps)$ can only be significantly slower 
than $\rfov$ when this containment is sharp, $\eps$ is 
relatively large, and the sets $\FOV(\BA)$ and $\PSA(\BA)$
are near the origin.
In such cases, values of $\eps$ that give convergence rates
$\rpsa(\eps)$ similar to $\rfov$ will be associated
with small $\eps$ and thus large constant
terms $\cpsa(\eps)$.  But proximity to the origin
implies that $\rfov$ will predict slow convergence,
and the pseudospectral bounds, while less sharp,
should still provide a decent indication of the
nature of convergence.
For an example in this vein, modify Example~D:
take a Jordan block with eigenvalue $\lambda$ near the
origin and superdiagonal having the constant $\delta$
for $0\ll \delta <  |\lambda|/\cos(\pi/(n+1))$ to ensure
the field of values extends near the origin but does
not contain it.

%%%%%%%%%%%%%%%%%%%%%%%%%%%%%%%%%%%%%%%%%%%%%%%%%%%%%%%%%%%%%%%%%%%%%%%%%%%%%%%%
\exhead{Example E: Only (FOV) not descriptive}  
%%%%%%%%%%%%%%%%%%%%%%%%%%%%%%%%%%%%%%%%%%%%%%%%%%%%%%%%%%%%%%%%%%%%%%%%%%%%%%%%
The field of values bound (FOV) is not descriptive when
there is initial stagnation followed by more rapid convergence.
The simplest example of GMRES stagnation at the first iteration
occurs for the Hermitian indefinite matrix
\[\BA = \left[\begin{array}{cc} 1 & 0\\ 0 & -1 \end{array}\right],\]
which was already used by Saad and Schultz~\cite{ss86} to illustrate
stagnation of GMRES(1).\ \ Since $\BA\in\C^{2\times2}$,
the second iteration gives exact convergence.  

Since $\BA$ is normal, (EV) is exact, correctly predicting two iterations
to solve the polynomial approximation problem on the discrete set of
two eigenvalues.  Since $\FOV(\BA)$ is the convex hull of $\SPEC(\BA)$,
we have $0\in\FOV(\BA) = [-1,1]$, and thus (FOV) predicts no convergence.
The $\eps$-pseudospectrum consists of the union of two disks of radius $\eps$,
each centered at an eigenvalue.  Thus, $\cpsa(\eps) = 2$ 
independent of $\eps\in(0,1]$. Taking $\eps\to0$, (PSA) predicts convergence
to arbitrary accuracy in two iterations.

The bound (FOV) can also fail when $\BA$ is a nonnormal matrix
for which $0\in \FOV(\BA)$ but $0 \not\in \PSA(\BA)$ for $\eps$
sufficiently small.
The following example, with
uniformly ill-conditioned eigenvectors, is the only case we 
discuss in which $\COND(\BV)$ is large and yet (EV) is still descriptive.
Let $\BLambda$ be diagonal with eigenvalues uniformly distributed
in the interval $[1,2]$ and define
\begin{equation}
  \BA = \BV\BLambda\BV^{-1}, \quad\mbox{where}\ 
    \BV = \left[\begin{array}{c c c c c}
         1  &  \sqrt{1-\delta} & \sqrt{1-\delta} & \cdots & \sqrt{1-\delta}\\
          &   \sqrt{\delta} & 0 & \cdots & 0\\
          & & \sqrt{\delta} & \ddots & \vdots \\
          & & & \ddots & 0 \\
          & & & & \sqrt{\delta}  \\
          \end{array}\right]
\label{fold}
\end{equation}
is upper triangular and $0<\delta\ll 1$ is a small positive parameter.
(The matrix $\BV$ is inspired by an example of Greenbaum and Strako\v{s}~\cite{gs94}.)

Clearly, the eigenvectors that form the column of $\BV$ 
are severely ill-conditioned;
looking at the upper left $2\times2$ block of $\BV$ alone
shows that $\COND(\BV)\ge (1+\sqrt{\delta})/\sqrt{\delta}$. 
One can show $\COND(\BV)=\ORDER(n/\sqrt{\delta})$ as
$n\to\infty$ and $\delta\to0$.  
For sufficiently small $\delta>0$,
$0\in\FOV(\BA)$ and (FOV) gives no convergence;
indeed, as seen in Figure~\ref{WpsaE},
 the field of values is quite large for small $\delta$.
The pseudospectra are also large, yet as $\eps$ is taken
small enough such that $0\not\in\PSA(\BA)$, 
the bound becomes increasingly descriptive.  
This is another case where the inclusion $\PSA(\BA)\subseteq \FOV(\BA)+\DISK{\eps}$
in Theorem~\ref{psafov} gives a poor upper bound on $\PSA(\BA)$ for
small $\eps$.

%%%%%%%%%%%%%%%%%%%%%%%%%%%%%%%%%%%%%%%%%%%%%%%%%%%%%%%%%%%%%%%%%%%%%%%%%%%%%%%%
% FIGURE AND TABLE FOR EXAMPLE E
%%%%%%%%%%%%%%%%%%%%%%%%%%%%%%%%%%%%%%%%%%%%%%%%%%%%%%%%%%%%%%%%%%%%%%%%%%%%%%%%

\begin{figure}[b!]
\begin{center}
\includegraphics[height=2.3in]{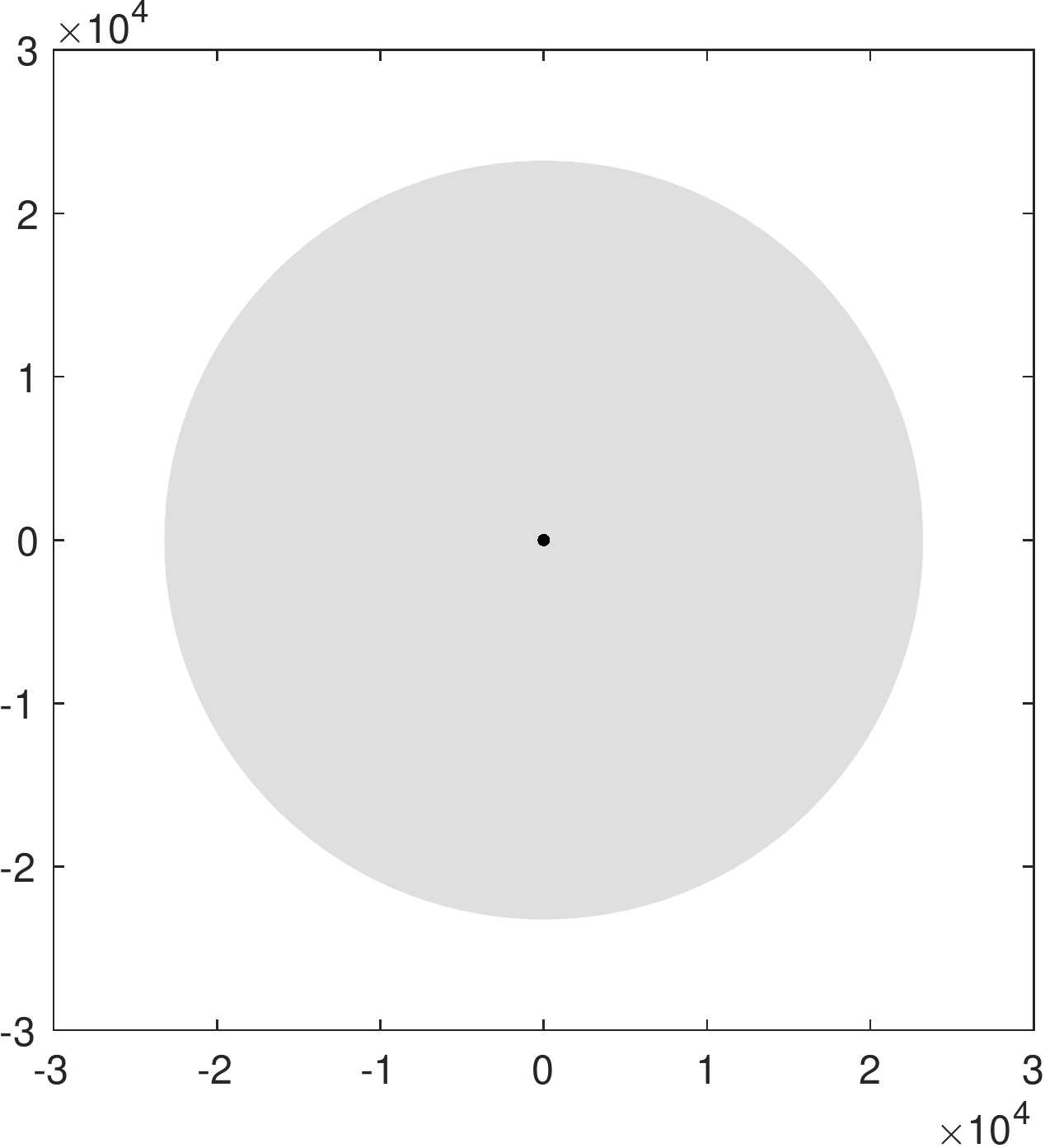}\quad
\raisebox{8.3pt}{\includegraphics[height=2.11in]{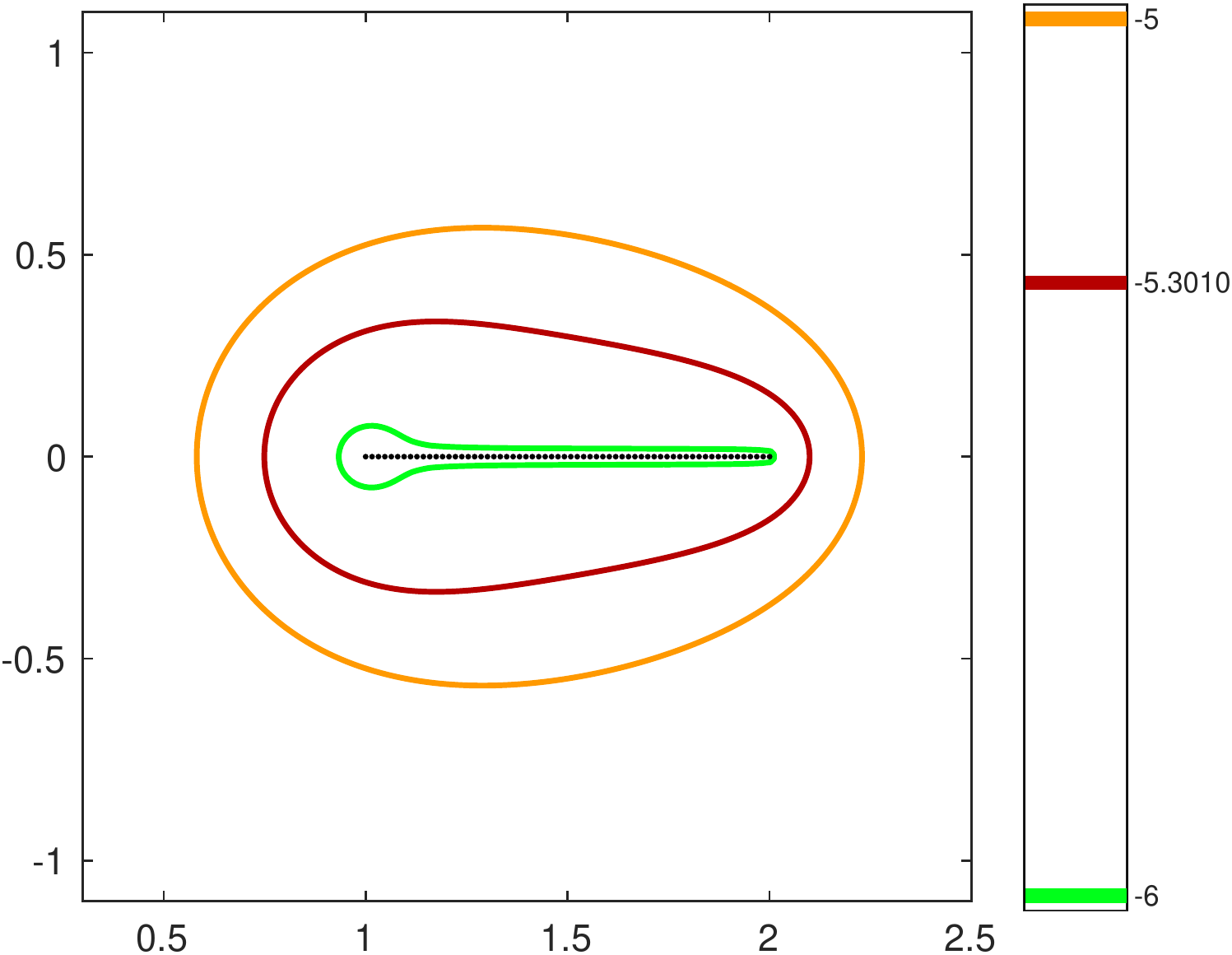}}
\end{center}

\caption{Field of values $\FOV(\BA)$ (left) and 
pseudospectra $\PSA(\BA)$ for $\eps = 10^{-5}$, $5\times 10^{-6}$, $10^{-6}$ 
(right) for the matrix~(\ref{fold}) with $\delta=10^{-8}$ and $n=64$.
Note the scale of the plot on the left.}
\label{WpsaE}
\end{figure}
%%%%%%%%%%%%%%%%%%%%%%%%%%%%%%%%%%%%%%%%%%%%%%%%%%%%%%%%%%%%%%%%%%%%%%%%%%%%%%%%

%%%%%%%%%%%%%%%%%%%%%%%%%%%%%%%%%%%%%%%%%%%%%%%%%%%%%%%%%%%%%%%%%%%%%%%%%%%%%%%%
\begin{figure}[t!]  
\begin{center}
\includegraphics[scale=0.57]{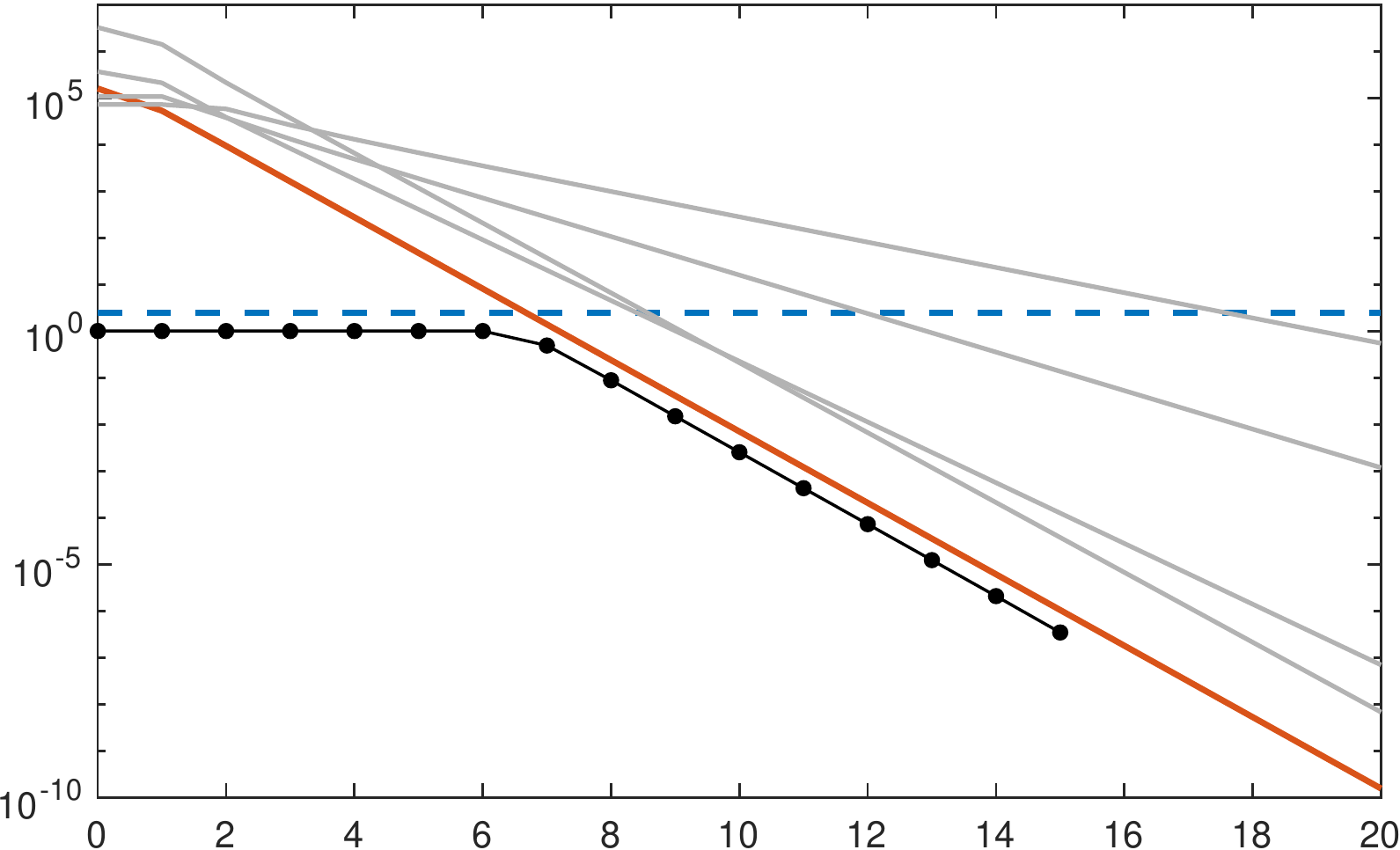}

\begin{picture}(0,0)
 \put(-180,127){\small $\displaystyle{\mingmres\! \norm{p(\BA)}}$}
 \put(-10,0){\small iteration, $k$}
 \put( 131,112){\small (FOV)}
 \put( 131,102){\small $\eps=10^{-5}$}
 \put( 131, 81){\small $\eps=5\times 10^{-6}$}
 \put( 131, 45){\small $\eps=10^{-6}$}
 \put( 131, 35){\small $\eps=10^{-7}$}
 \put(81,30){\small (EV)}
 \put(-5,84){\rotatebox{-28}{\small exact}}
\end{picture}
\end{center}

\vspace*{-5pt}
\caption{Convergence bounds for the matrix~(\ref{fold}) with $\delta=10^{-8}$ and $n=64$.}
\label{Ebounds}
\end{figure}
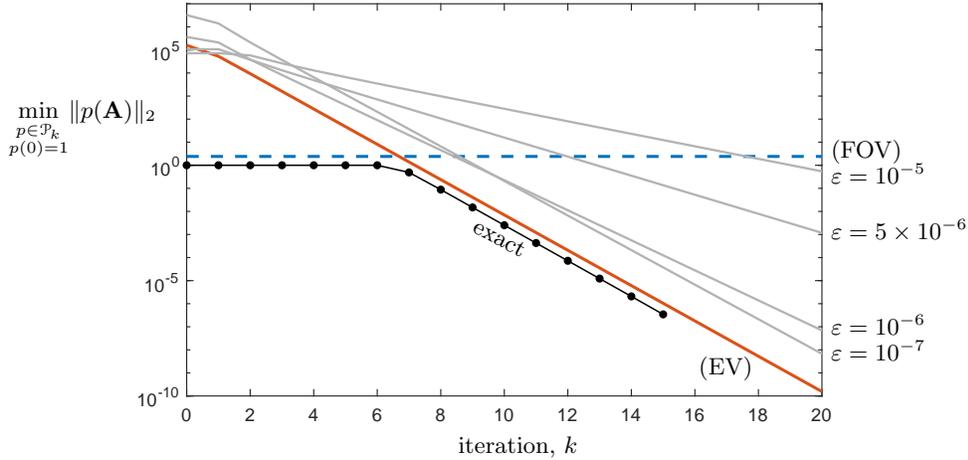
%%%%%%%%%%%%%%%%%%%%%%%%%%%%%%%%%%%%%%%%%%%%%%%%%%%%%%%%%%%%%%%%%%%%%%%%%%%%%%%%

Figure~\ref{Ebounds} shows the bounds for $\delta=10^{-8}$ and $n=64$,
the same parameters used in Figure~\ref{WpsaE}.
The constant $\cev=\COND(\BV)$ is obtained by 
computing the condition number in MATLAB, 
with the first column of $\BV$ multiplied by $\sqrt{n}$ to
improve conditioning.
The rate $\rev$ is taken from the optimal polynomial
on the interval $[1,2]$.
The pseudospectral bound was determined for each $\eps$
by bounding $\PSA(\BA)$ from the outside by an convex polygon,
obtaining the convergence rate through numerical conformal
mapping, and then applying the bound~(\ref{convex}).
The ``exact'' curve was computed using the semidefinite
programming strategy of Toh and Trefethen~\cite{tt99},
as implemented in the SDPT3 Toolbox~\cite{TTT99}.
For this example, random initial residuals typically do not
lead to the long plateau obtained in the ``exact''
(Ideal GMRES) curve shown here. 
(Whether any initial residual attains the Ideal GMRES curve
for this example is not known.)
Not only does $\FOV(\BA)$ contain the origin for these parameters; 
it contains a circle of radius $10^4$ centered at the origin.  
For the values of $\eps$ used in Figure~\ref{WpsaE},
the pseudospectra $\PSA(\BA)$ omit the origin, 
giving the convergent bounds in Figure~\ref{Ebounds}.

%%%%%%%%%%%%%%%%%%%%%%%%%%%%%%%%%%%%%%%%%%%%%%%%%%%%%%%%%%%%%%%%%%%%%%%%%%%%%%%%
\exhead{Example F: Only (PSA) descriptive}  
%%%%%%%%%%%%%%%%%%%%%%%%%%%%%%%%%%%%%%%%%%%%%%%%%%%%%%%%%%%%%%%%%%%%%%%%%%%%%%%%

For the Jordan block in Example~D, the field of values bound (FOV)
captured the single convergence rate perfectly.  For the present
example, we seek a nondiagonalizable matrix that initially stagnates
but eventually converges more rapidly, making (FOV) misleading.  
Though it might seem like a gimmick, 
we could say the most extreme example of this behavior occurs for the matrix
\[ \BA = \left[\begin{array}{cc} 1 & \alpha \\ 0 & 1\end{array}\right],
\rlap{$\qquad \alpha \ge 2$,}\]
which featured as a submatrix in Example~B.\ \  
For $\alpha\ge2$, there exist right hand sides for which GMRES 
makes no progress at the first step.  Yet at the second step,
there is exact convergence since $\BA\in\C^{2\times2}$.

Since $\BA$ is nondiagonalizable, $\cev=\infty$ and 
(EV) does not apply.  
Nor does (FOV) give convergence, since $0 \in \FOV(\BA)=1+\overline{\DISK{\alpha/2}}$,
a disk centered at~1 with radius $\alpha/2$~\cite[\S1.3]{gustrao}. 
The pseudospectral bound, however, captures the exact convergence in
two iterations.  From~(\ref{eq:psaJ2}) we have 
 $\PSA(\BA) = 1+\DISK{\sqrt{\alpha\eps + \eps^2}}$, 
giving the constant
$\cpsa(\eps) = \sqrt{\alpha/\eps+1}$.
Using the residual polynomial $p(z) = (1-z)^k$, (PSA) gives the bound
\[ \mingmres \norm{p(\BA)} 
  \le \sqrt{\alpha/\eps + 1} \ \Big(\alpha\eps+\eps^2\Big)^{k/2}.\]
At step $k=2$, the upper bound becomes 
$\alpha^{3/2} \eps^{1/2} + O(\eps^{3/2})$ as $\eps\to0$,
and hence by taking $\eps>0$ sufficiently small, (PSA) predicts
convergence to arbitrary accuracy at the second iteration.

For a more interesting example that does not rely on the 
dimension $n=2$, take a Jordan block and progressively scale down 
the off-diagonal entries:  for fixed $\beta>0$, 
\begin{equation}
  \BA = \left[\begin{array}{ccccc}
          1 & \beta & && \\
            & 1 & \textstyle{\beta\over2} & & \\
            &   & 1 & \ddots & \\
            &   &   & \ddots & \textstyle{\beta \over (n-1)} \\
            &   &   & & 1
          \end{array}\right].
\label{intmat}
\end{equation}
Reichel and Trefethen called a related example 
an ``integration matrix''~\cite{rt92} and noted that its pseudospectra are disks.
Driscoll, Toh, and Trefethen showed that Ideal GMRES exhibits an improving 
convergence rate for this matrix, and linked that behavior to the notable
shrinking of $\PSA(\BA)$ as $\eps\to0$~\cite[p.~564]{dtt98}.

Since $\BA$ is nondiagonalizable, the bound (EV) cannot be
usefully applied.  Both the field of values and the pseudospectra
of $\BA$ are circular disks, as can be seen via a diagonal unitary 
similarity transformation~(see \cite[p.~268]{TE05} for details).  
If $n\ge2$ and $\beta\ge 2$, then $0\in\FOV(\BA)$,
and so for such values of $\beta$, 
(FOV) cannot give convergence.
Figure~\ref{Fpsa} shows the field of values and
pseudospectra for $\beta=5/2$ and $N=64$.
For small $\eps$, $\FOV(\BA)$ is evidently considerably larger than $\PSA(\BA)$,
and thus (PSA) predicts much faster convergence rates than (FOV) as $\eps\to 0$.

%%%%%%%%%%%%%%%%%%%%%%%%%%%%%%%%%%%%%%%%%%%%%%%%%%%%%%%%%%%%%%%%%%%%%%%%%%%%%%%%
% FIGURE AND TABLE FOR EXAMPLE F
%%%%%%%%%%%%%%%%%%%%%%%%%%%%%%%%%%%%%%%%%%%%%%%%%%%%%%%%%%%%%%%%%%%%%%%%%%%%%%%%

\begin{figure}[b!]  
\begin{center}
\includegraphics[scale=0.57]{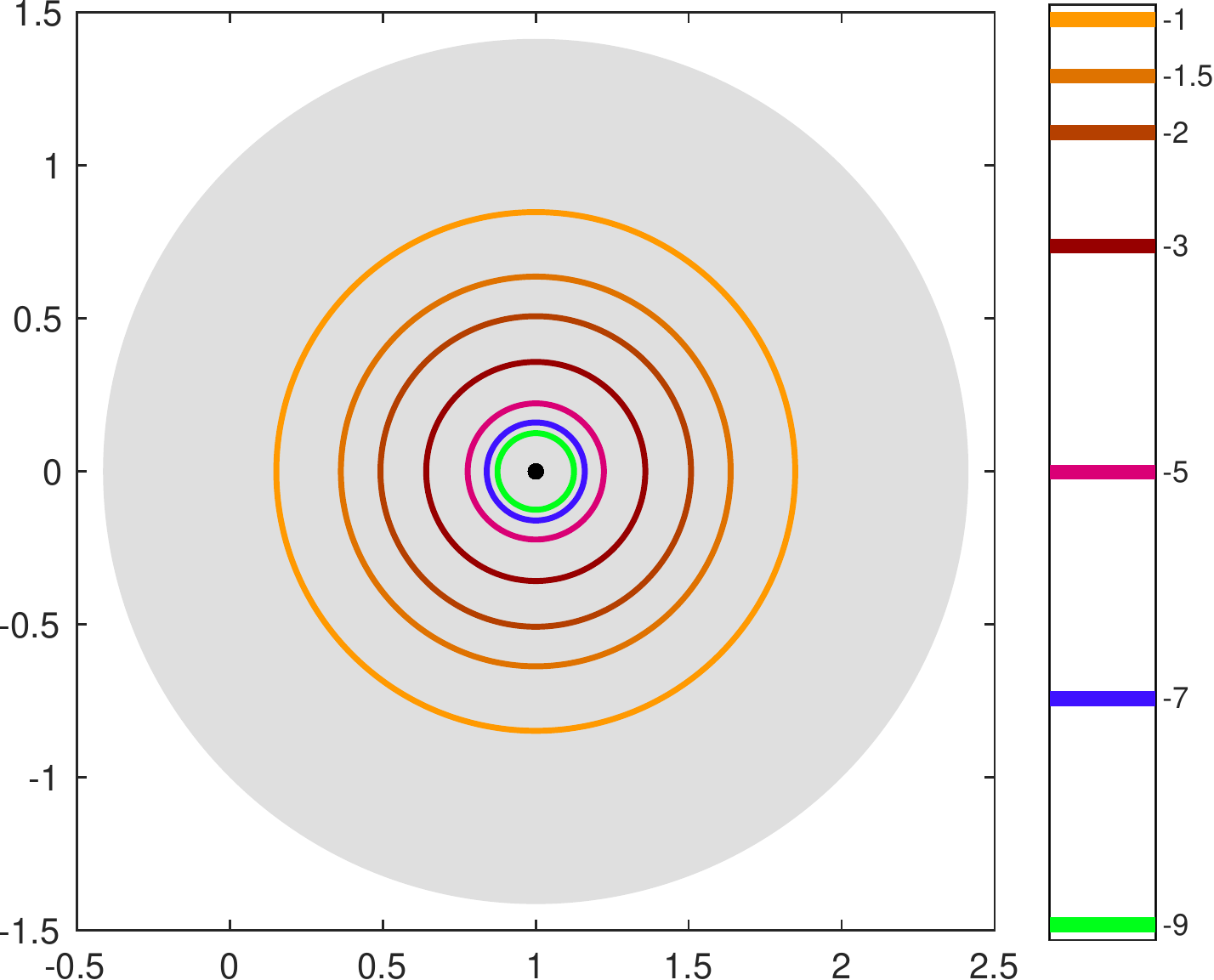}
\end{center}

\caption{The field of values $\FOV(\BA)$ (gray region) and $\eps$-pseudospectra $\PSA(\BA)$
(for $\eps=10^{-1}, 10^{-1.5}, 10^{-2}, 10^{-3}, 10^{-5}, 10^{-7}, 10^{-9}$)
for the matrix~(\ref{intmat})
with $\beta=5/2$ and $N=64$.  Notice that $0\in\FOV(\BA)$, but 
$0\not\in\PSA(\BA)$ for the pseudospectra shown here.
\label{Fpsa}}
\end{figure}
%%%%%%%%%%%%%%%%%%%%%%%%%%%%%%%%%%%%%%%%%%%%%%%%%%%%%%%%%%%%%%%%%%%%%%%%%%%%%%%%

%%%%%%%%%%%%%%%%%%%%%%%%%%%%%%%%%%%%%%%%%%%%%%%%%%%%%%%%%%%%%%%%%%%%%%%%%%%%%%%%
% FIGURE AND TABLE FOR EXAMPLE F
%%%%%%%%%%%%%%%%%%%%%%%%%%%%%%%%%%%%%%%%%%%%%%%%%%%%%%%%%%%%%%%%%%%%%%%%%%%%%%%%

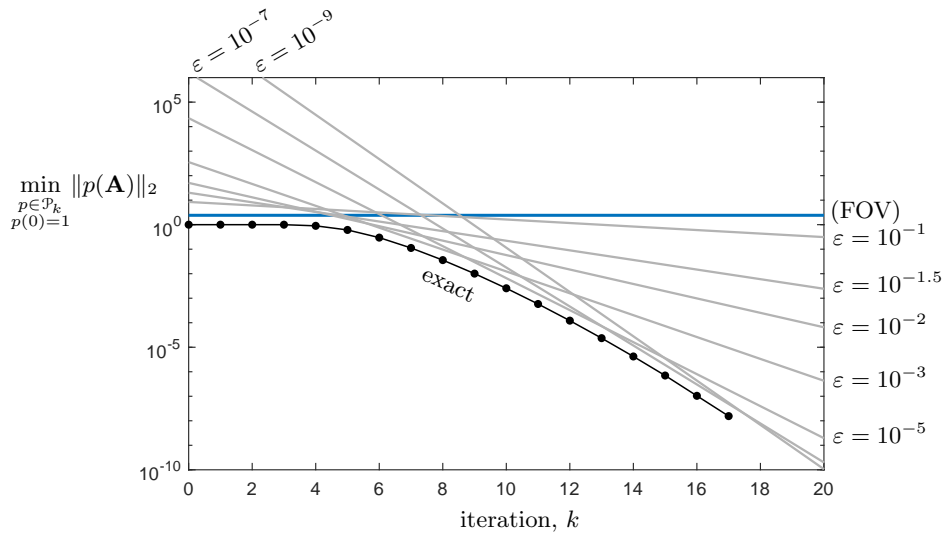
\begin{figure}[t!]  
\begin{center}
\vspace*{2em}
\includegraphics[scale=0.57]{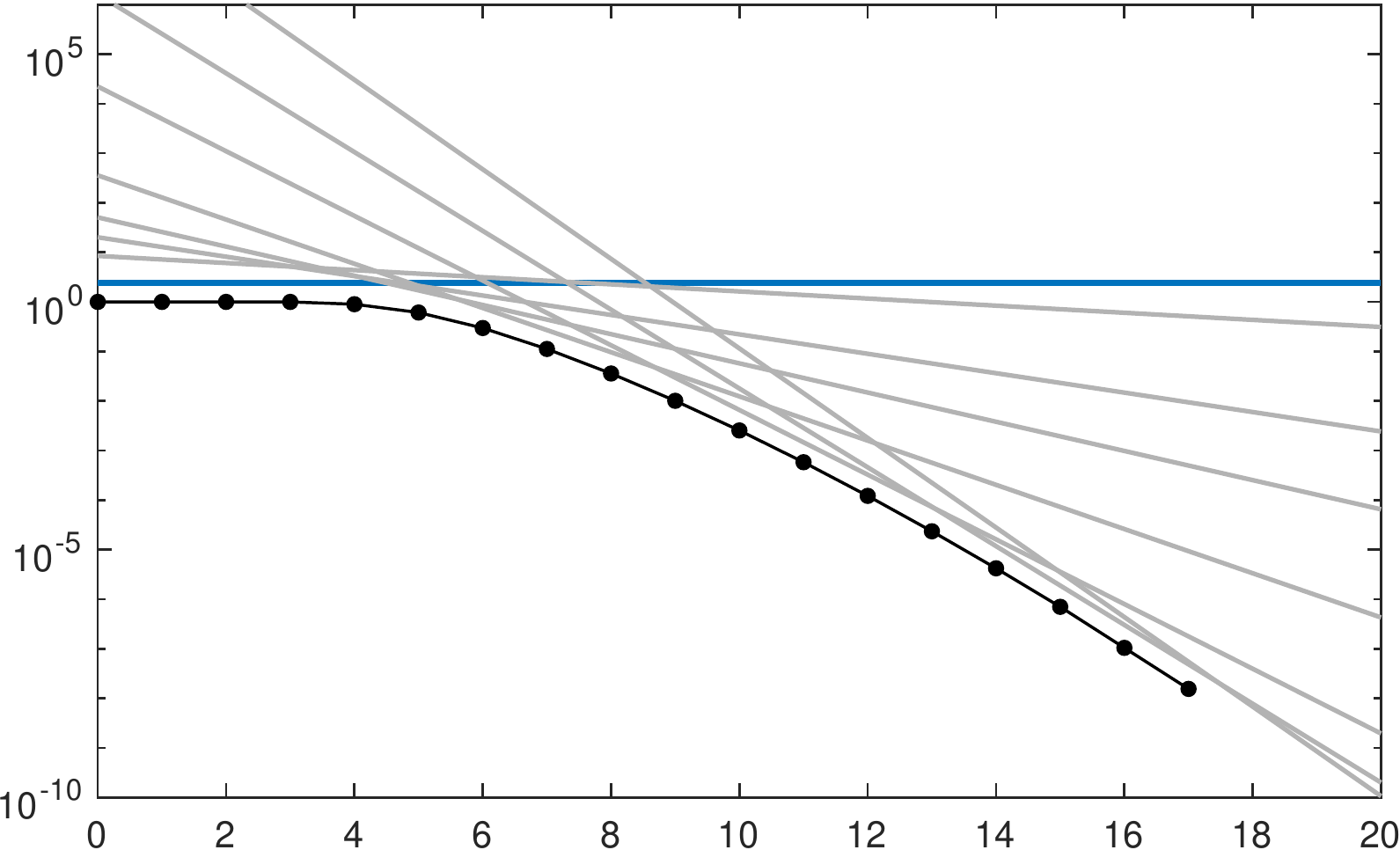}

\begin{picture}(0,0)
 \put(-180,127){\small $\displaystyle{\mingmres\! \norm{p(\BA)}}$}
 \put(-10,0){\small iteration, $k$}
 \put( 130,117){\small (FOV)}
 \put( 131,107){\small $\eps=10^{-1}$}
 \put( 131, 89){\small $\eps=10^{-1.5}$}
 \put( 131, 73){\small $\eps=10^{-2}$}
 \put( 131, 53){\small $\eps=10^{-3}$}
 \put( 131, 32){\small $\eps=10^{-5}$}
 \put(-114,172){\rotatebox{30}{\small $\eps=10^{-7}$}}
 \put(-90,172){\rotatebox{30}{\small $\eps=10^{-9}$}}
 \put(-25,94){\rotatebox{-25}{\small exact}}
\end{picture}
\end{center}

\caption{Convergence bounds for the integration matrix~(\ref{intmat}) 
         with $\alpha=5/2$ and $n=64$.  The bounds (PSA) correspond to 
         the same values of $\eps$ for which $\PSA(\BA)$ is shown in Figure~\ref{Fpsa}.}
\label{Fbounds}
\end{figure}
%%%%%%%%%%%%%%%%%%%%%%%%%%%%%%%%%%%%%%%%%%%%%%%%%%%%%%%%%%%%%%%%%%%%%%%%%%%%%%%%

Figure~\ref{Fbounds} shows the corresponding GMRES bounds.
The ``exact'' curve was again computed using
the SDPT3 Toolbox~\cite{TTT99}.  
The pseudospectral bounds were obtained by numerically computing 
the radius of each pseudospectral boundary.  Increasing the
dimension $n$ does not significantly alter the pseudospectra
for the values of $\eps$ shown here.
(For one thing, such an extension would be a norm $\beta/n$ 
perturbation to the block diagonal matrix $\BA\oplus \BI$.)

As an alternative to (FOV), one might instead consider the
Crouzeix--Greenbaum bound~(CG)~\cite{CG19}.
Figure~\ref{fig:CGex2} shows the sets $\CG$
defined in~(\ref{eq:CG}) for 
the matrix~(\ref{intmat}) of dimension $N=64$.
The plot on the left uses $\beta=5/2$ (as in Figures~\ref{Fpsa} 
and~\ref{Fbounds}), giving a set $\CG$ in~(\ref{eq:CG}) that 
does not include the origin but surrounds it, 
and so the bound (CG) cannot give convergence.
The right plot shows $\CG$ for the smaller value $\beta=2$:
although $\BA$ is a nondiagonalizable matrix with just one 
Jordan block, the set $\CG$ excludes the origin and will 
yield a convergent bound with an asymptotic rate determined by
$\CG$.

%%%%%%%%%%%%%%%%%%%%%%%%%%%%%%%%%%%%%%%%%%%%%%%%%%%%%%%%%%%%%%%%%%%%%%%%%%%%%%%%
% FIGURES FOR SECOND CROUZEIX-GREENBAUM example
%%%%%%%%%%%%%%%%%%%%%%%%%%%%%%%%%%%%%%%%%%%%%%%%%%%%%%%%%%%%%%%%%%%%%%%%%%%%%%%%

\begin{figure}[t!]
\begin{center}
\includegraphics[width=2.25in]{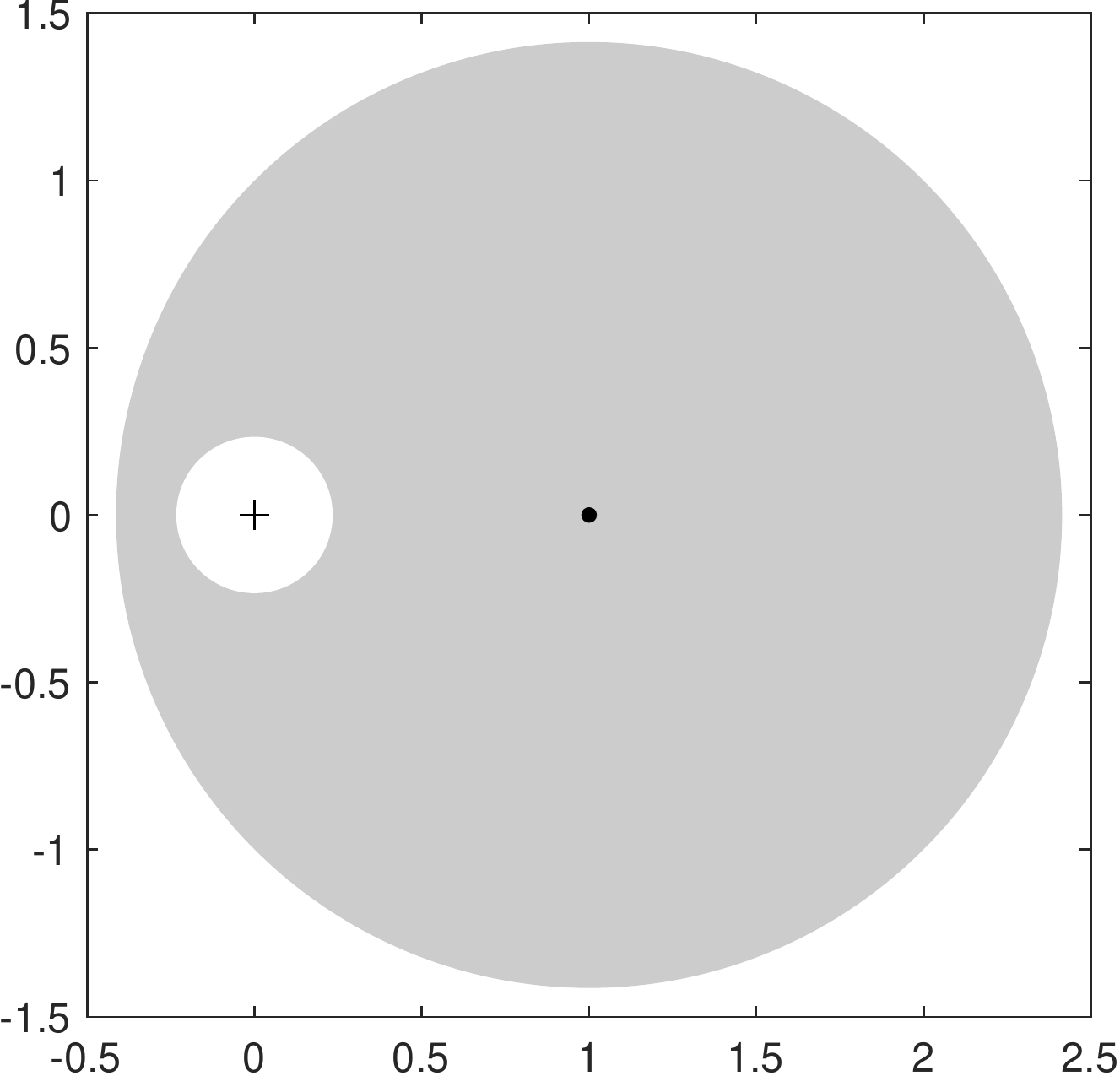}\qquad
\includegraphics[width=2.25in]{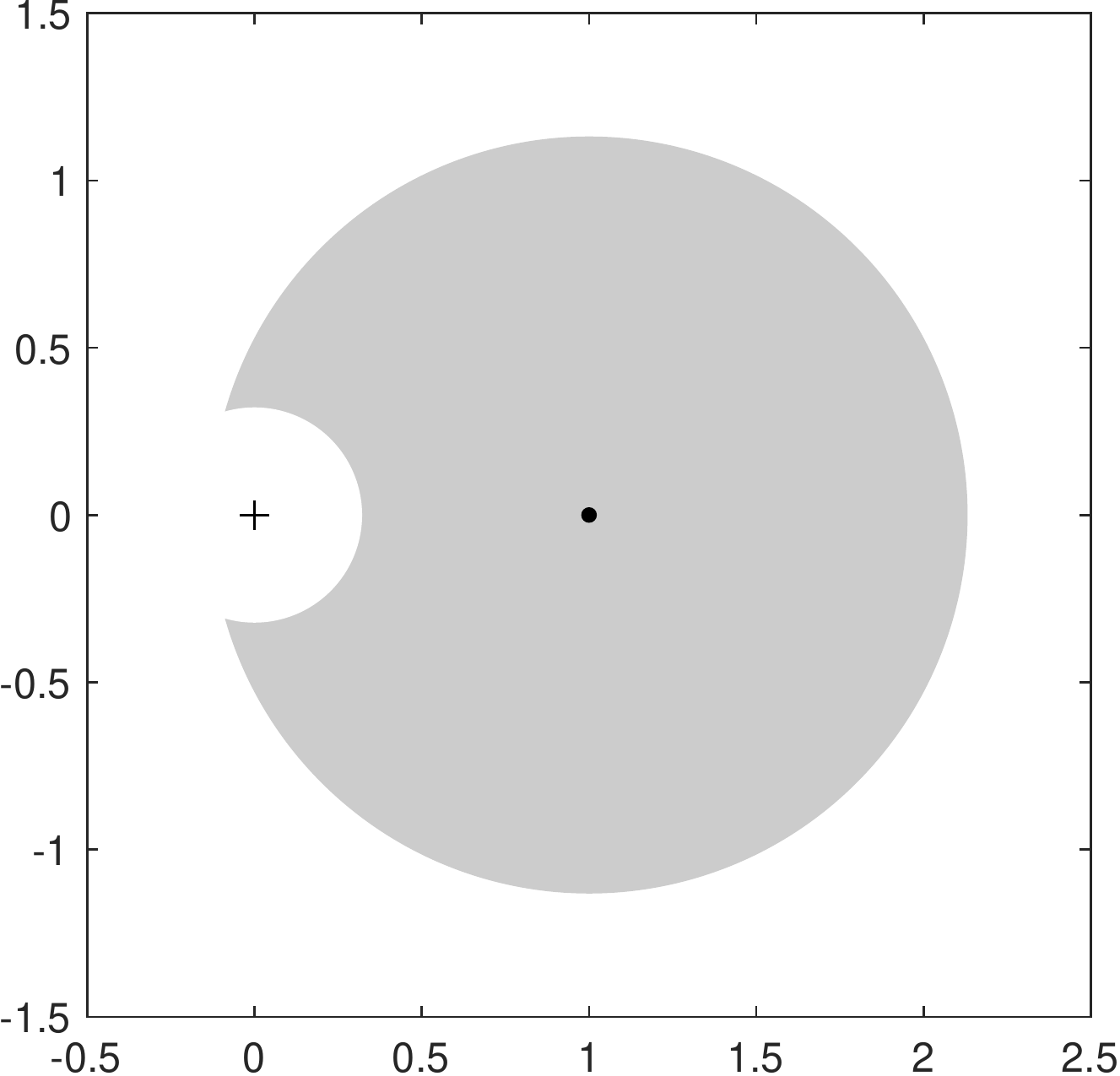}

\begin{picture}(0,0)
\put(-50,25){\small $\beta=5/2$}
\put(143,25){\small $\beta=2$}
\end{picture}
\end{center}

\vspace*{-10pt}
\caption{\label{fig:CGex2}
The Crouzeix--Greenbaum sets $\CG$ for matrix~$(\ref{intmat})$,
$N=64$.  For $\beta=5/2$ (left), $\CG$ surrounds the origin
(marked with $+$), and the bound~(CG) does not describe convergence.  
For $\beta=2$ (right), $\CG$ does not surround the origin, and so
(CG) will give a convergent bound.}
\end{figure}
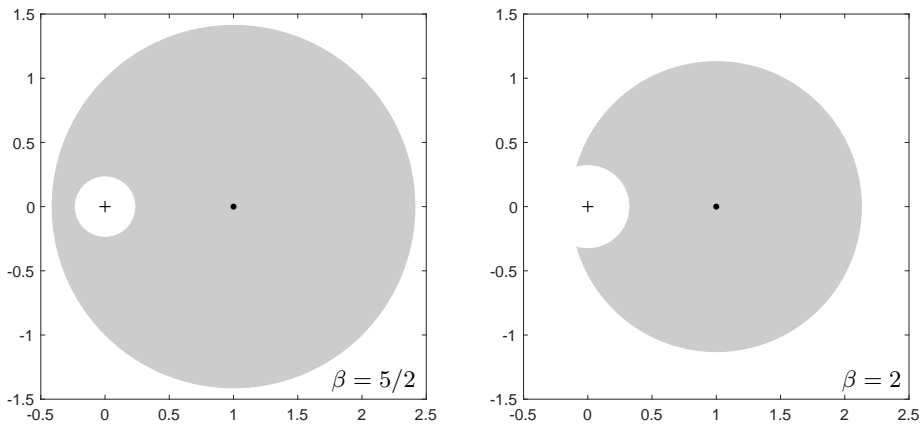
%%%%%%%%%%%%%%%%%%%%%%%%%%%%%%%%%%%%%%%%%%%%%%%%%%%%%%%%%%%%%%%%%%%%%%%%%%%%%%%%

%%%%%%%%%%%%%%%%%%%%%%%%%%%%%%%%%%%%%%%%%%%%%%%%%%%%%%%%%%%%%%%%%%%%%%%%%%%%%%%%
% FIGURE AND TABLE FOR EXAMPLE F / JCF VERSION
%%%%%%%%%%%%%%%%%%%%%%%%%%%%%%%%%%%%%%%%%%%%%%%%%%%%%%%%%%%%%%%%%%%%%%%%%%%%%%%%
\begin{figure}[b!]  
\begin{center}
\vspace*{2em}
\includegraphics[scale=0.57]{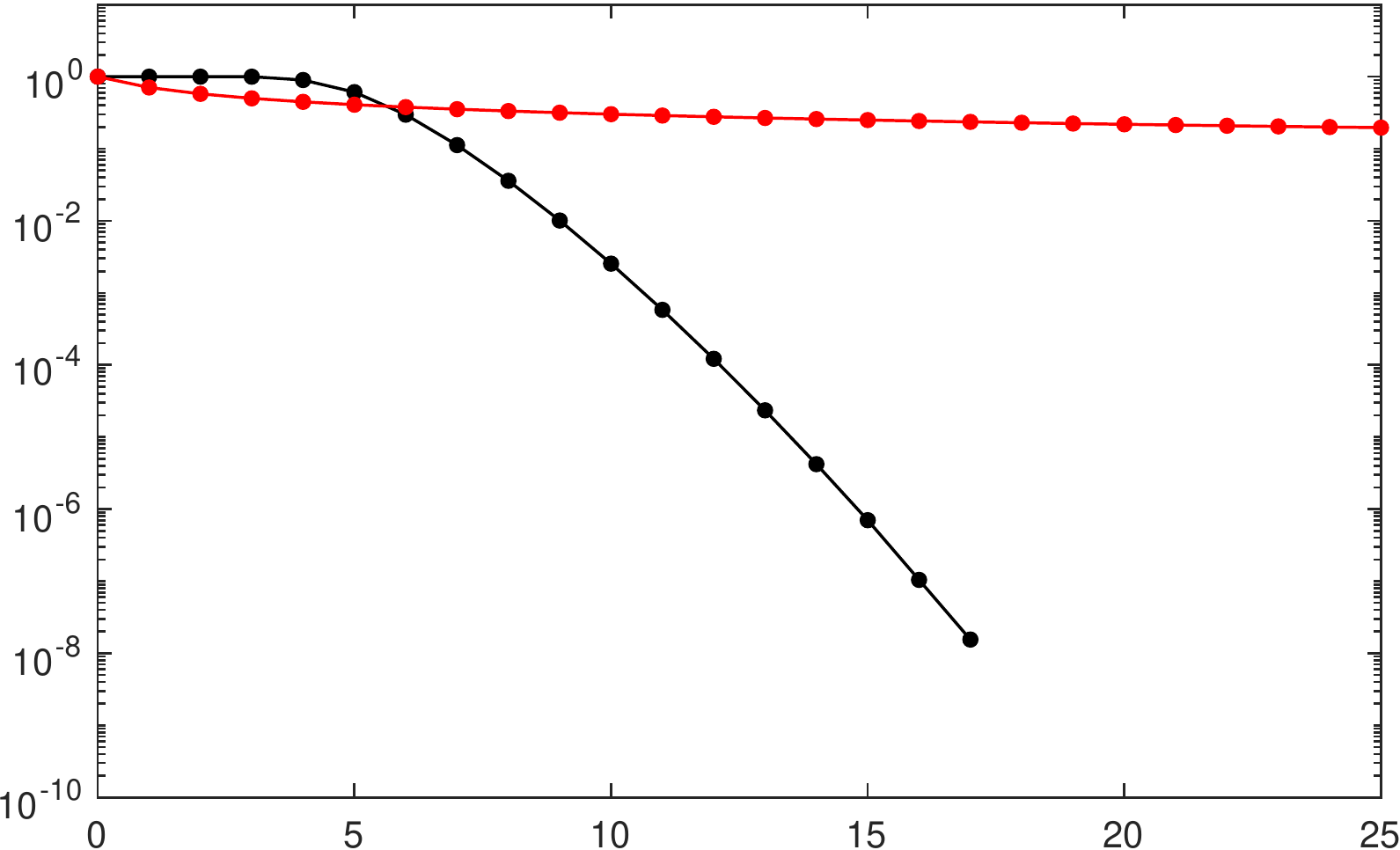}

\begin{picture}(0,0)
 \put(-10,0){\small iteration, $k$}
 \put(-23,61){\rotatebox{0}{\small $\displaystyle{\mingmres\! \norm{p(\BA)}}$}}
 \put(0,135){\rotatebox{0}{\small lower bound on $\displaystyle{\mingmres\! \norm{p(\BJ)}}$}}
\end{picture}
\end{center}

%\vspace*{-5pt}
\caption{Comparison of the Ideal GMRES problems for the matrix $\BA$ in (\ref{intmat})
         and its Jordan factor $\BJ$.  In this case transforming to the Jordan form 
         results in a more difficult GMRES problem.}
\label{FexJ}
\end{figure}
%%%%%%%%%%%%%%%%%%%%%%%%%%%%%%%%%%%%%%%%%%%%%%%%%%%%%%%%%%%%%%%%%%%%%%%%%%%%%%%%

When faced with a nondiagonalizable matrix,
one might naturally think of using the Jordan canonical form 
as the basis for GMRES analysis.
The matrix~(\ref{intmat}) provides a cautionary example. 
Suppose we take the Jordan form $\BA = \BX \BJ\BX^{-1}$,
and then follow the example of the eigenvalue--eigenvector
bound~(EV) to obtain
\begin{equation} \label{jordanbnd}
 \mingmres \norm{p(\BA)} \le \kappa(\BX) \mingmres \norm{p(\BJ)}.
\end{equation}
For the matrix~(\ref{intmat}) we compute
\begin{eqnarray*}
  \BA &=& \BX \BJ \BX^{-1} \\[5pt]
   &=& 
       \left[\begin{array}{ccccc}
          1 &  & & & \\
            & \textstyle{1\over \beta} &  & & \\
            &   & \textstyle{2\over \beta^2} &  & \\
            &   &   &  \ddots &  \\
            &   &   & & \textstyle{(n-1)! \over \beta^{n-1}}
          \end{array}\right] 
       \left[\begin{array}{ccccc}
          1 & 1 & && \\
            & 1 & 1 & & \\
            &   & 1 & \ddots & \\
            &   &   & \ddots & 1 \\
            &   &   & & 1
          \end{array}\right] 
       \left[\begin{array}{ccccc}
          1 &  & & & \\
            & \beta &  & & \\
            &   & \textstyle{\beta^2\over 2} &  & \\
            &   &   & \ddots  &  \\
            &   &   & & \textstyle{\beta^{n-1} \over (n-1)!}
          \end{array}\right]. 
\end{eqnarray*}
Note that $\kappa(\BX) \ge  (n-1)!/\beta^{n-1}$ grows \emph{factorially} with $n$.
The bound~(\ref{jordanbnd}) then requires analysis of $\norm{p(\BJ)}$,
which amounts to our Example~D in~(\ref{eq:exD}) with $\delta=1$:
we expect convergence at a slow rate.  
Indeed Ipsen's analysis~\cite{Ips00} 
(applying GMRES to the last column of the identity matrix) ensures that
\[ \mingmres \norm{p(\BJ)} \ge 1/\sqrt{k},\]
which does not capture the improving rate of convergence exhibited for this $\BA$;
see Figure~\ref{FexJ}.
While $0 \not \in \FOV(\BJ) = \{z\in\C: |1-z| \le \cos(\pi/(n+1)\}$~\cite[\S{1.3}]{gustrao},
the rate of convergence associated with $\FOV(\BJ)$ will be very slow.
For the values of $\eps$ used in Figure~\ref{Fpsa}, 
$\PSA(\BJ)$ will be much larger than $\PSA(\BA)$.
\emph{By transforming to the Jordan form, we have introduced an enormous constant
but arrived at an Ideal GMRES problem for $\BJ$ that converges more slowly than 
the Ideal GMRES problem for the original $\BA$.}

%%%%%%%%%%%%%%%%%%%%%%%%%%%%%%%%%%%%%%%%%%%%%%%%%%%%%%%%%%%%%%%%%%%%%%%%%%%%%%%%
\exhead{No Example: Only (PSA) not descriptive} 
%%%%%%%%%%%%%%%%%%%%%%%%%%%%%%%%%%%%%%%%%%%%%%%%%%%%%%%%%%%%%%%%%%%%%%%%%%%%%%%%

This scenario would require GMRES to (eventually) converge steadily, 
and for this convergence to be captured by (EV) and (FOV) but not (PSA).\ \ 
If GMRES initially stagnates, then $\rfov$ must be close to~1 
(since $\cfov = 1+\sqrt{2}$ is small), and so (FOV) could only predict slow
overall convergence.  Thus, the example we seek could not exhibit initial
stagnation.  In this case, if (EV) is to be accurate, then $\cev = \COND(\BV)$ 
must be small: implying that $\BA$ must be nearly normal.
Theorems~\ref{bauerfike}
and~\ref{psafov} then insure that $\PSA(\BA)$ cannot be much larger
than $\SPEC(\BA)$ and $\FOV(\BA)$, so it is impossible
to get an example where (PSA) gives a significantly slower asymptotic 
convergence rates than (EV) and (FOV).\ \ (One could take $\SPEC(\BA)$ 
to contain an eigenvalue very close to the origin, requiring one to take
$\eps$ very small to ensure $0\not\in\PSA(\BA)$, but this effect is limited
by the fact that $\BA$ must be close to normal.  Moreover, if $\BA$ has other
eigenvalues much farther from the origin, the small eigenvalue will cause 
GMRES to exhibit initial stagnation, and (FOV) will not capture the eventual convergence.)

%%%%%%%%%%%%%%%%%%%%%%%%%%%%%%%%%%%%%%%%%%%%%%%%%%%%%%%%%%%%%%%%%%%%%%%%%%%%%%%%
\subsection{Summary of the Examples}
%%%%%%%%%%%%%%%%%%%%%%%%%%%%%%%%%%%%%%%%%%%%%%%%%%%%%%%%%%%%%%%%%%%%%%%%%%%%%%%%

Let us collect some of the points highlighted in these examples.

\begin{enumerate}
\item[{\rm (EV)}] 
This bound works well for normal matrices and for~(\ref{fold}),
where all eigenvalues were uniformly ill-conditioned,
but fails when the matrix was nondiagonalizable;
it can also fail to be descriptive when $\kappa(\BV)$ is 
large primarily because of a small number of ill-conditioned eigenvalues.
(See the example in Section~\ref{sec:convdiff}.)

\item[{\rm (FOV)}] 
This bound performs well when only one convergence stage was
observed, as in Examples~A and~D.\ \  Its primary advantages 
over (PSA) for these examples was sharpness 
(Figures~\ref{fig:Abounds} and~\ref{Dbounds}) and ease of computability.
When GMRES exhibits an initial period of transient stagnation,
(FOV) fails to capture the eventual convergence, as in Examples~B, C, E, and~F.
\item[{\rm (PSA)}]
This bound inherits properties of both (EV) and (FOV), but can also
capture interesting information between these extremes,
as seen in Example~F.\ \   The primary flaw in (PSA), 
exploited in Examples~B and~C, is its
inability to recognize that the spectrum is a discrete point set, 
and thus it tends to overestimate the influence of outlying eigenvalues
that GMRES can effectively eliminate at an early stage of convergence.  
The bounds (PSA$'$) and (PSA$''$) suggest a way to address this shortcoming.
(That said, pseudospectral techniques are not a panacea for bounding 
the 2-norm of matrix polynomials.  For some extreme examples of their
limitations, see~\cite{GT93,RR11}.)
\end{enumerate}

\medskip
In the next section, we illustrate
how pseudospectra can yield convergence estimates during an iteration,
and apply the bounds surveyed here to a matrix
derived from a convection-diffusion problem.

%%%%%%%%%%%%%%%%%%%%%%%%%%%%%%%%%%%%%%%%%%%%%%%%%%%%%%%%%%%%%%%%%%%%%%%%%%%%%%%%
\section{Adaptive Pseudospectral Bounds} \label{sec:adaptive}
%%%%%%%%%%%%%%%%%%%%%%%%%%%%%%%%%%%%%%%%%%%%%%%%%%%%%%%%%%%%%%%%%%%%%%%%%%%%%%%%

The pseudospectral bound (PSA) often provides a good indication 
of GMRES convergence, especially when considering a collection
of bounds based on a wide range of $\eps$ values.
Pseudospectra can be expensive to compute for large $\BA$;
however, one can use elements from early GMRES iterations to
\emph{approximate} the pseudospectra of $\BA$, and hence obtain
an \emph{estimate} for how the convergence will proceed.
Such insight could, for example, give some indication of 
when to restart the GMRES algorithm.
(An entirely different method for predicting future convergence
based on early iterations has been suggested by Liesen~\cite{Lie00}.)

Suppose we have taken $k$ steps of GMRES.\ \ 
The implementation of Saad and Schultz~\cite{ss86} uses 
the Arnoldi process to build an orthonormal basis
$\{\Bv_1, \ldots, \Bv_{k+1}\}$ 
for the Krylov subspace 
$\KRY_{k+1}(\BA, \Br_0) := \SPAN \{\Br_0, \BA\Br_0, \ldots, \BAs^k\Br_0\}$.
Organize the basis vectors into 
$\BV_k = [\Bv_1\  \cdots\ \Bv_k]\in\C^{n\times k}$
and $\BV_{k+1} = [\BV_k\ \Bv_{k+1}] \in\C^{n\times (k+1)}$.
The Arnoldi process gives a partial upper Hessenberg decomposition
of $\BA$, 
\begin{equation} \label{arnoldi}
 \BA\BV_k = \BV_{k+1} \BHtil_k \qquad \mbox{and}\qquad
   \BV_k^*\BA\BV_k = \BH_k,
\end{equation}
where $\BHtil_k\in\C^{(k+1)\times k}$ is upper Hessenberg and 
$\BH_k\in\C^{k\times k}$ consists of the first $k$ rows
of $\BHtil_k$; see, e.g.,~\cite[\S6.3]{Saa03}.  We can take
the subdiagonal entries of $\BHtil_k$ to be nonnegative.

Toh and Trefethen~\cite{tt96} show that the pseudospectra of 
$\BH_k$ and $\BHtil_k$ can potentially yield good approximations 
to those of $\PSA(\BA)$ even when $k\ll n$.  
A point $z\in\C$ is in the $\eps$-pseudospectrum of 
the $(k+1)\times k$ rectangular matrix $\BHtil_k$, 
$z\in\PSA(\BHtil_k)$, 
provided $s_{\rm min}(z\BItil_k-\BHtil_k) < \eps$, where $\BItil_k$ 
is the $k\times k$ identity matrix augmented by a row of zeros and
$s_{\rm min}(\cdot)$ denotes the smallest singular value; 
for details, see~\cite{tt96,WT02}.
With this definition, $\PSA(\BH_k)$,
$\PSA(\BHtil_k)$, and $\PSA(\BA)$ are related as follows.

\medskip
\begin{theorem} \label{ritzpsa}
Suppose $\BV_k\adj\BA\BV_k = \BH_k$ and $\BA\BV_k = \BV_{k+1}\BHtil_{k}$.
Then
\begin{eqnarray*}
&(i)&\qquad \PSA(\BHtil_1)\subseteq \PSA(\BHtil_2) \subseteq\cdots\subseteq
      \PSA(\BHtil_{n-1}) \subseteq \PSA(\BH_n) = \PSA(\BA). \\
&(ii)&\qquad \PSA(\BH_k) \subseteq \sigma_\sepshat(\BHtil_k)
                        \subseteq \sigma_\sepshat(\BA),
      \mbox{where $\widehat{\eps} := \eps+h_{k+1,k}$}.
\end{eqnarray*}
\end{theorem}
\vspace*{-1.25em}
\begin{proof}
Toh and Trefethen proved  part~(i), which
follows immediately from noting that 
$s_{\rm min}(z\BItil_k-\BHtil_k) \ge s_{\rm min}(z\BItil_{k+1}-\BHtil_{k+1})$.\ \ 
For part~(ii), suppose that $z\in\PSA(\BH_k)$.  Observe that
$s_{\rm min}(z\BItil_k-\BHtil_k) \le s_{\rm min}(z\BI-\BH_k)  +  h_{k+1,k} 
                            <  \eps + h_{k+1,k},$
where the first inequality follows from~\cite[Thm.~3.3.16]{hojo2}.
Applying part~$(i)$ to this bound completes the proof.\qquad
\end{proof}

Since the pseudospectra of $\BH_k$ and $\BHtil_k$ approximate those of 
$\BA$, it is natural to approximate the bound (PSA) by replacing
$\PSA(\BA)$ by $\PSA(\BH_k)$ or $\PSA(\BHtil_k)$.  
The resulting expressions are no longer convergence bounds, but only estimates
(as we indicate with the ``$\lesssim$'' symbol):
\begin{eqnarray}
 \frac{\norm{\Br_k}}{\norm{\Br_0}} 
    &\lesssim&
       \frac{{\cal L}(\Gamma^{(k)}_\eps)}{2\pi\eps} 
          \mingmres \max_{z\in\PSA(\BH_k)} |p(z)|,  \label{eq:approxpsa1}\\
 \frac{\norm{\Br_k}}{\norm{\Br_0}} 
    &\lesssim&
       \frac{{\cal L}(\widetilde{\Gamma}^{(k)}_\eps)}{2\pi\eps} 
         \mingmres \max_{z\in\PSA(\BHtil_k)} |p(z)|, \label{eq:approxpsa2}
\end{eqnarray}
where $\Gamma^{(k)}_\eps$ and $\widetilde{\Gamma}^{(k)}_\eps$ 
denote Jordan curves enclosing $\PSA(\BH_k)$ and $\PSA(\BHtil_k)$.

What value of $\eps$ is relevant at a specific iteration?
The following bounds, while not necessarily sharp,
suggest one way to approach this question.
This proposition gives pseudospectral interpretations
(cf.~\cite[Lemma~2.1]{SG98}) of results about Ritz and
Harmonic values.  For the result about Ritz values, see, 
e.g., \cite[\S4.6]{arpack}).
The result about harmonic Ritz values follows from
Simonicini and Gallopoulos~\cite{SG96}; 
see also Goosens and Roose~\cite{gr99}.

\begin{proposition}
For $k<n$, the eigenvalues of 
$\BH_k$ (Ritz values) are contained in the $\eps$-pseudospectrum
of $\BA$ for $\eps=h_{k+1,k}$.
If $\BH_k$ is nonsingular, the roots of the GMRES residual 
polynomial (harmonic Ritz values~\cite{freund92,gr99}) are 
contained in the $\eps$-pseudospectrum of $\BA$ 
for $\eps=h_{k+1,k} + h_{k+1,k}^2/s_{\rm min}(\BH_k)$.
\end{proposition} 
\begin{proof}
To prove the first part, form the perturbation
$\BE := -h_{k+1,k}\Bv_{k+1}\Be_k^*\BV_k^*$.  
Then, using~(\ref{arnoldi}) and the fact that $\BV_k^*\BV_k^{} = \BI$,
\begin{eqnarray*}
(\BA+\BE)\BV_k &=& \BA\BV_k -  h_{k+1,k}\Bv_{k+1}\Be_k^*\BV_k^*\BV_k^{} \\
               &=& (\BV_k \BH_k + h_{k+1,k}\Bv_{k+1}\Be_k^*) - h_{k+1,k}\Bv_{k+1}\Be_k^* 
               = \BV_k\BH_k,
\end{eqnarray*}
where $\Be_k\in\C^k$ is the $k$th column of the $k\times k$ identity.
Since $(\BA+\BE)\BV_k = \BV_k\BH_k$, 
$\RANGE(\BV_k)$ is an invariant subspace of $\BA+\BE$, and
thus $\SPEC(\BH_k)\subseteq \SPEC(\BA+\BE) \subset \PSA(\BA)$,
where $\eps := h_{k+1,k} = \norm{\BE}$.

The harmonic Ritz values are the eigenvalues of $(\BH_k + h_{k+1,k}^2\Bf_k^{}\Be_k^*)$, 
where $\Bf_k := \BH_k^{-*}\Be_k$~\cite{gr99,SG96}.
Defining the perturbation 
$\BE:=(h_{k+1,k}^2\BV_{\kern-1.5pt k}^{} \Bf_k^{}\Be_k^* - h_{k+1,k}^{}\Bv_{k+1}^{}\Be_k^*)\BV_k^*$,
from~(\ref{arnoldi}) it follows that
\[ (\BA+\BE)\BV_k = \BV_k(\BH_k+h_{k+1,k}^2\Bf_k^{}\Be_k^*),\]
and thus $\SPEC(\BH_k+h_{k+1,k}^2\Bf_k^{}\Be_k^*) \subseteq
         \SPEC(\BA+\BE)\subset\PSA(\BA)$,
where $\norm{\BE} \le \eps 
       := h_{k+1,k} + h_{k+1,k}^2/s_{\rm min}(\BH_k)$.\qquad
\end{proof}

\medskip

GMRES convergence \emph{estimates} based on $\PSA(\BH_k)$ 
have several applications.  
If $k$ is small, the estimate might hint at GMRES behavior
at future iterations.  If $k$ is larger (e.g., the $k$
that satisfies the GMRES convergence criterion) and
$\PSA(\BH_k)\approx \PSA(\BA)$ over a range of $\eps$ values,
one could estimate the upper bound on GMRES convergence
that might inform future runs of GMRES with the same $\BA$.\ \ 
Note that $\PSA(\BH_k)$ depends on the initial residual $\Br_0$.
If $\Br_0$ is deficient in all eigenvector directions associated 
with a particular eigenvalue, that eigenvalue cannot influence 
$\BH_k$ nor the GMRES estimates derived from it.
If $\Br_0$ only has a small component in a certain eigenvector
direction, that component may not exert much influence on early 
iterations (and $\BH_k$ for small $k$), but become significant at later
iterations.

Toh and Trefethen
observe qualitative links between the pseudospectra and the
GMRES iteration polynomial~\cite{toh96,tt99}.
A deeper quantitative understanding
of this relationship could give insight about
the ability of pseudospectral bounds to describe GMRES
convergence, the value of $\eps$ for which $\PSA(\BA)$ 
gives the best bound at the $k$th iteration,
and the merits of adaptive strategies, like the one described 
here, to capture significant features of convergence behavior.

Figure~\ref{adapt:int:bounds} shows adaptive convergence
estimates drawn from the ``integration matrix'' (\ref{intmat}),
again with $n=64$ and $\beta=5/2$, based on the pseudospectra of $\BH_k$ 
for $k=4$ and $k=12$ for a specific choice of initial residual, 
shown in Figure~\ref{fig:psaint}.
In Figure~\ref{Fbounds}, we saw that the bound 
(PSA) was descriptive for this example, and thus hope these 
Arnoldi estimates would perform similarly well.  
The estimates shown here are based on 
a random initial residual with entries drawn from the standard 
normal distribution.  Taking estimates at iteration $k=4$
gives some hint of the quick convergence that follows; when $k=12$,
the pseudospectra of $\BH_k$ match those of $\BA$ for relevant
values of $\eps$ and characterize the worst case convergence curve.
The curve labeled $\norm{\Br_k}/\norm{\Br_0}$ is the actual GMRES
convergence obtained for this particular initial residual, with an 
asterisk marking the iteration from which the convergence estimate 
was drawn.

%%%%%%%%%%%%%%%%%%%%%%%%%%%%%%%%%%%%%%%%%%%%%%%%%%%%%%%%%%%%%%%%%%%%%%%%%%%%%%%%
% FIGURE AND TABLE FOR ADAPTIVE EXAMPLES USING THE INTEGRATION MATRIX
%%%%%%%%%%%%%%%%%%%%%%%%%%%%%%%%%%%%%%%%%%%%%%%%%%%%%%%%%%%%%%%%%%%%%%%%%%%%%%%%
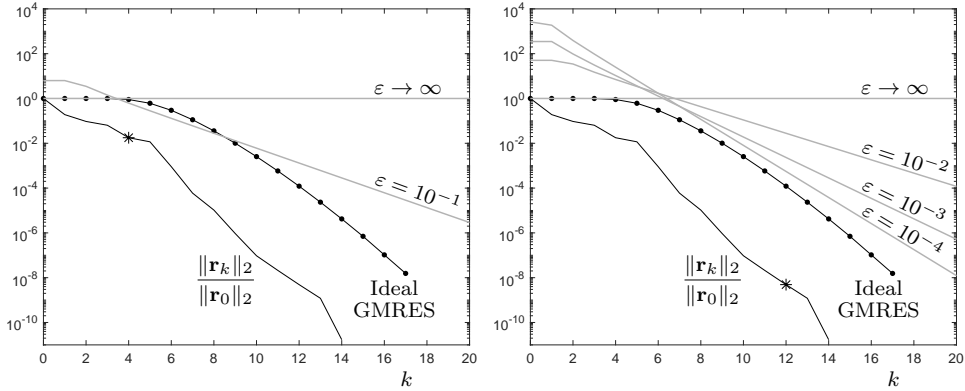
\begin{figure}[t!]
\begin{center}
\includegraphics[scale=0.37]{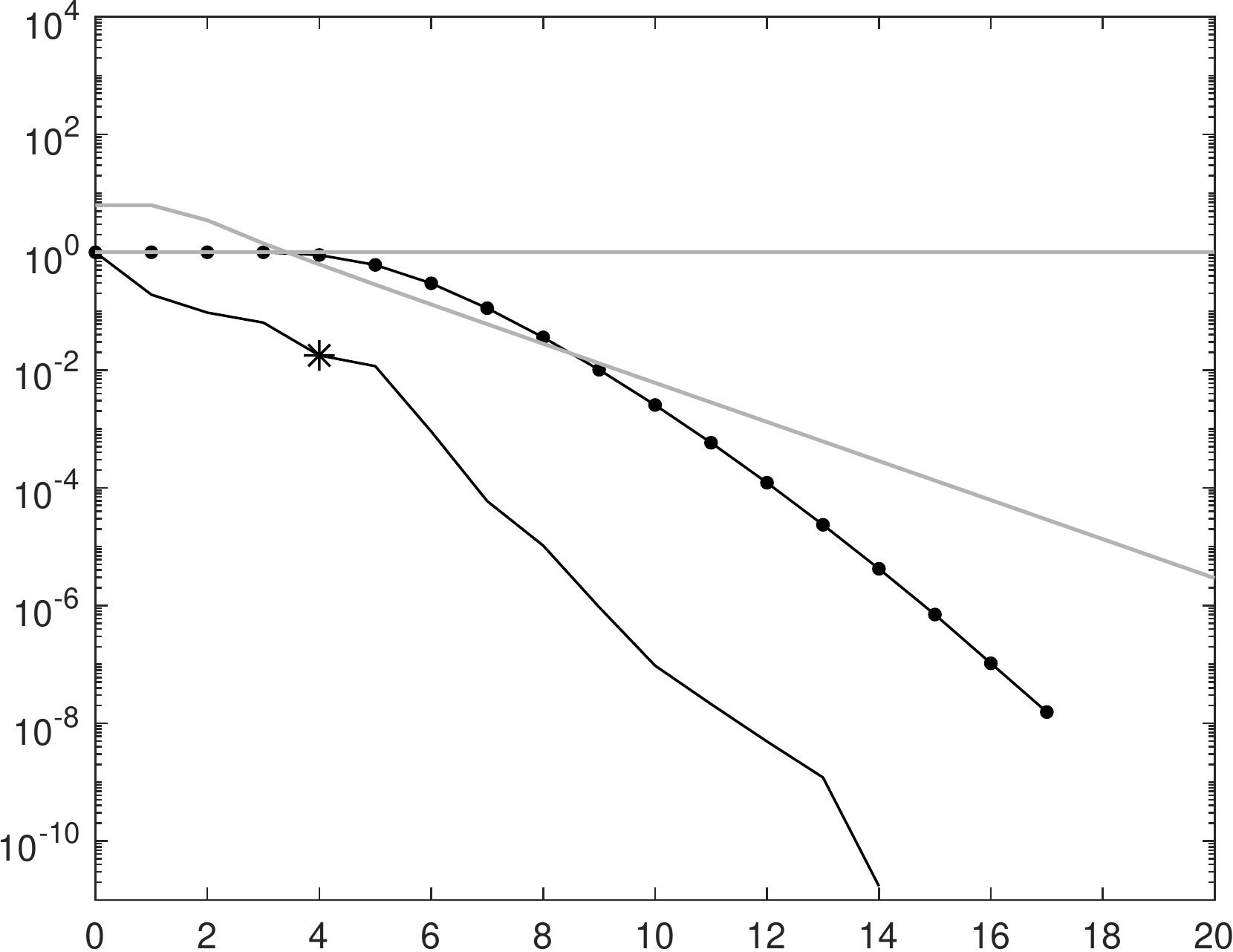}
\begin{picture}(0,0)
\put(-110,30){\footnotesize $\displaystyle{\|\Br_k\|_2\over \|\Br_0\|_2}$}
\put(-50,18){\footnotesize \shortstack{Ideal\\ GMRES}}
\put(-42,102.5){\footnotesize $\eps\to\infty$}
\put(-42,68.5){\footnotesize \rotatebox{-20}{$\eps=10^{-1}$}}
\put(-32,-8){\footnotesize $k$}
\end{picture}
\includegraphics[scale=0.37]{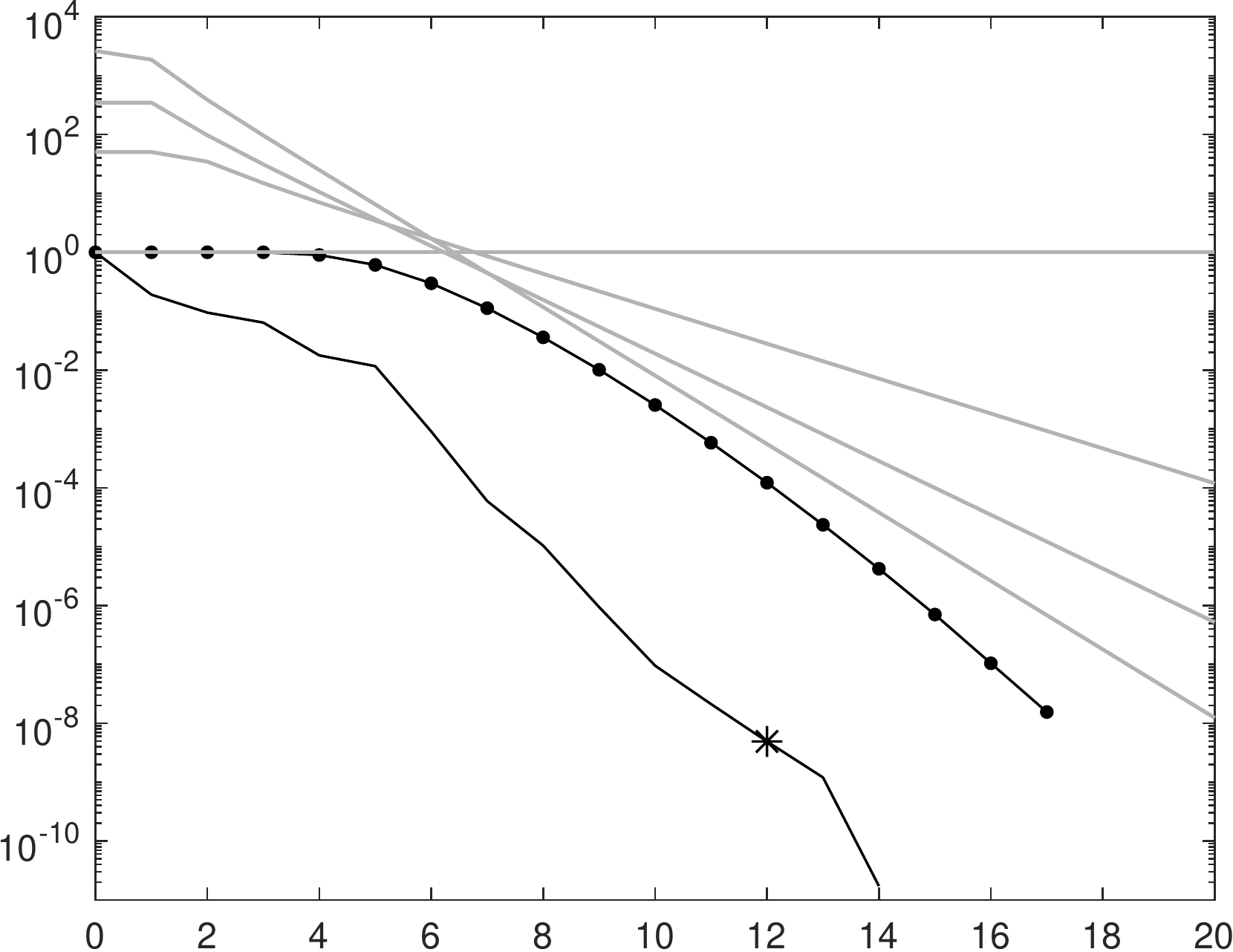}
\begin{picture}(0,0)
\put(-110,30){\footnotesize $\displaystyle{\|\Br_k\|_2\over \|\Br_0\|_2}$}
\put(-51,18){\footnotesize \shortstack{Ideal\\ GMRES}}
\put(-42,102.5){\footnotesize $\eps\to\infty$}
\put(-42,80.5){\footnotesize \rotatebox{-18}{$\eps=10^{-2}$}}
\put(-42,66.5){\footnotesize \rotatebox{-26}{$\eps=10^{-3}$}}
\put(-42,57){\footnotesize \rotatebox{-31}{$\eps=10^{-4}$}}
\put(-32,-8){\footnotesize $k$}
\end{picture}
\end{center}

\vspace*{3pt}
\caption{\label{adapt:int:bounds}
Adaptive convergence estimates for the integration 
matrix~(\ref{intmat}) with $\beta=5/2$ and $n=64$, 
generated at iteration $k=4$ on the left and $k=12$ on the right,
based on pseudospectra of $\BH_k$ shown in Figure~\ref{fig:psaint}.}
\end{figure}
%%%%%%%%%%%%%%%%%%%%%%%%%%%%%%%%%%%%%%%%%%%%%%%%%%%%%%%%%%%%%%%%%%%%%%%%%%%%%%%%

%%%%%%%%%%%%%%%%%%%%%%%%%%%%%%%%%%%%%%%%%%%%%%%%%%%%%%%%%%%%%%%%%%%%%%%%%%%%%%%%
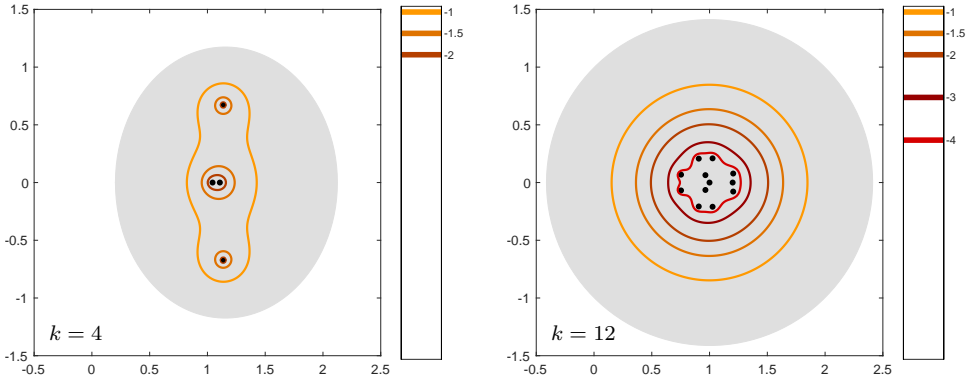
\begin{figure}[h!]
\begin{center}
\includegraphics[scale=0.42]{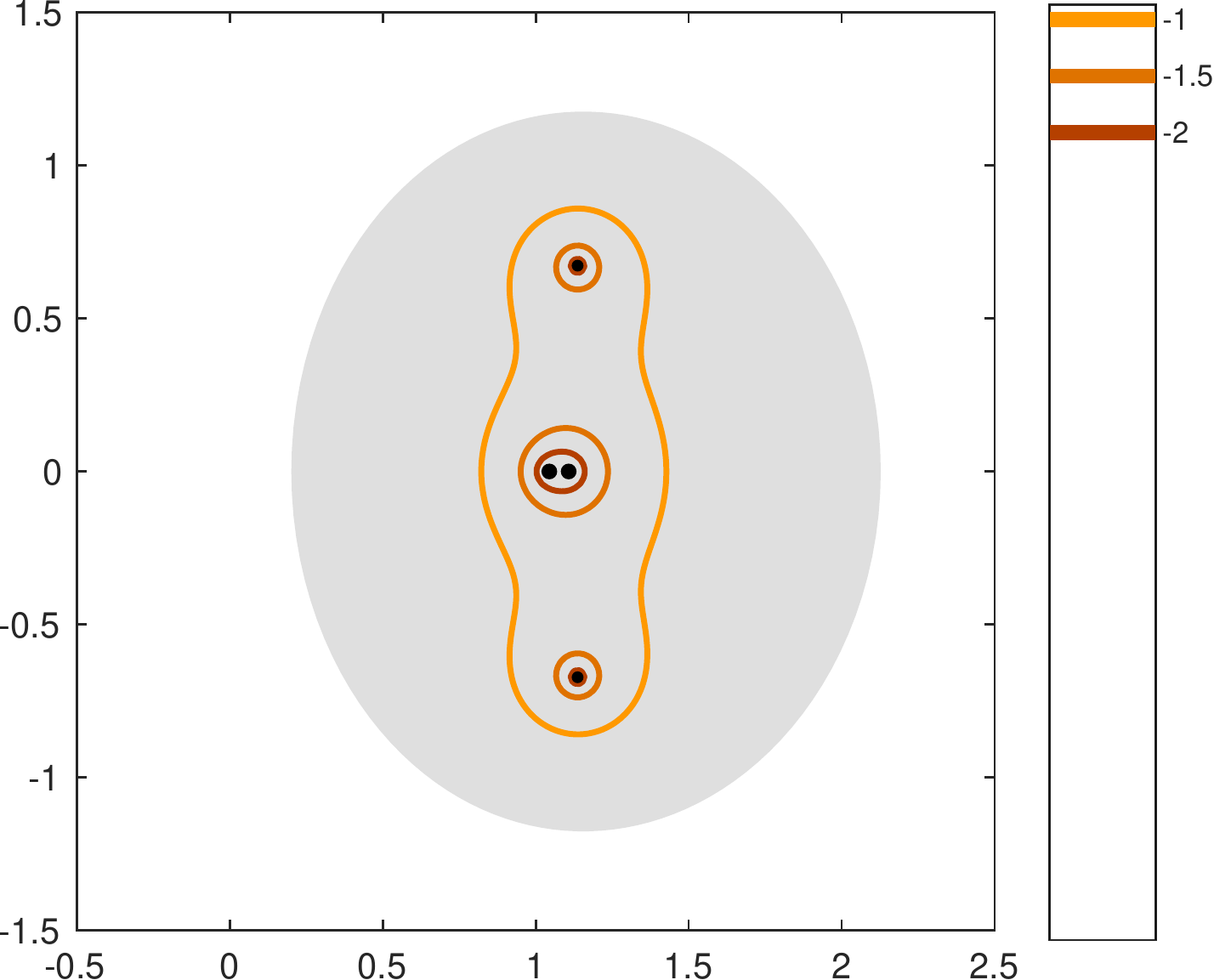}
\begin{picture}(0,0)
\put(-160,13){\footnotesize $k=4$}
\end{picture}
\quad
\includegraphics[scale=0.42]{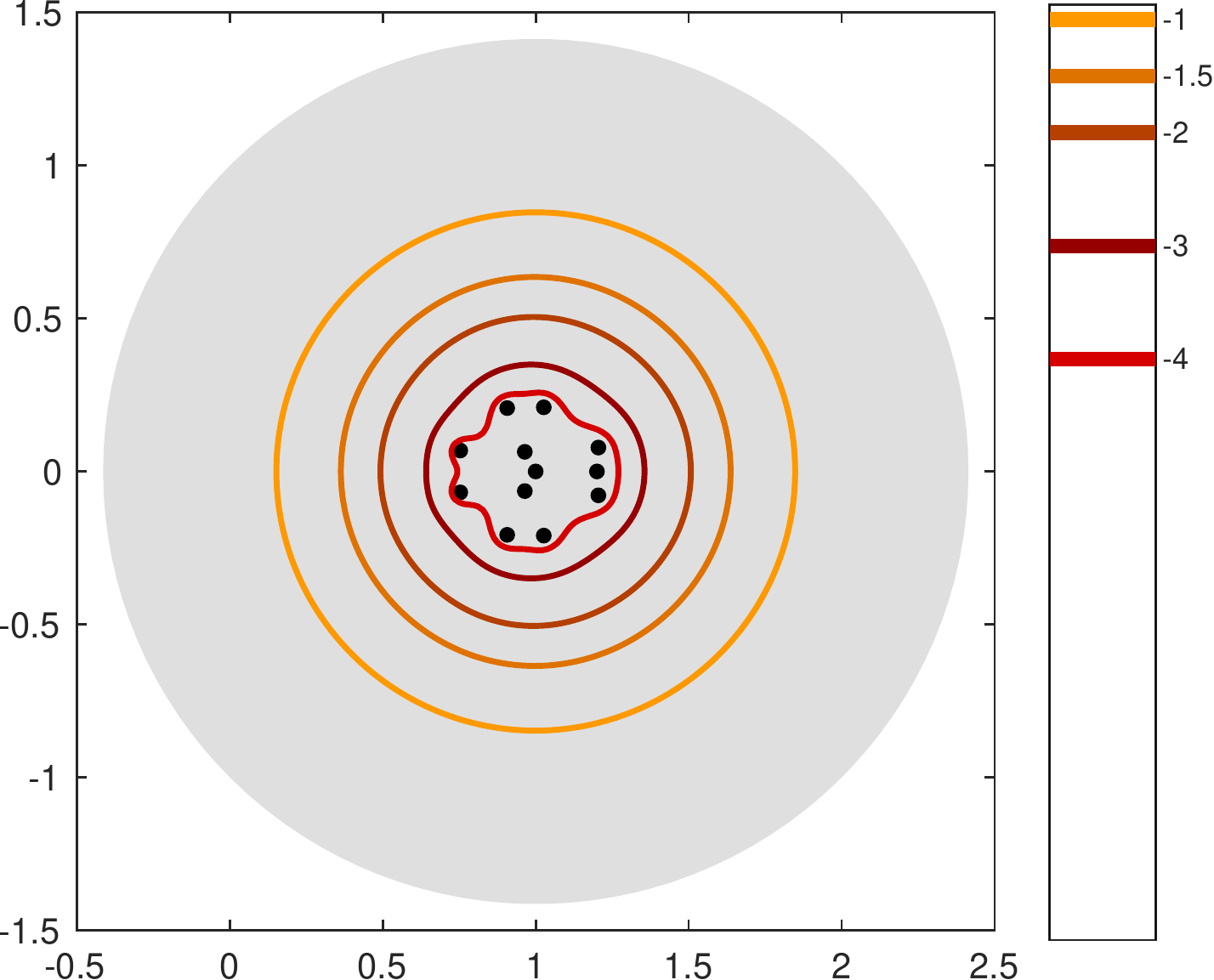}
\begin{picture}(0,0)
\put(-160,13){\footnotesize $k=12$}
\end{picture}
\end{center}

\vspace*{3pt}
\caption{\label{fig:psaint}
At iterations $k=4$ and $k=12$, 
the field of values $\FOV(\BH_k)$ (gray region) and $\eps$-pseudospectra $\PSA(\BH_k)$
(for $\eps=10^{-1}, 10^{-1.5}, 10^{-2}, 10^{-3}, 10^{-4}$)
for the matrix~(\ref{intmat}) with $\beta=5/2$ and $n=64$, 
generated using the same $\Br_0$ whose GMRES convergence
is illustrated in Figure~\ref{adapt:int:bounds}.  
The color levels use the same scale as those in Figure~\ref{Fpsa} 
for the full matrix $\BA$, to facilitate comparison.}
\end{figure}
%%%%%%%%%%%%%%%%%%%%%%%%%%%%%%%%%%%%%%%%%%%%%%%%%%%%%%%%%%%%%%%%%%%%%%%%%%%%%%%%

%%%%%%%%%%%%%%%%%%%%%%%%%%%%%%%%%%%%%%%%%%%%%%%%%%%%%%%%%%%%%%%%%%%%%%%%%%%%%%%%
\subsection{A Practical Example} \label{sec:convdiff}
%%%%%%%%%%%%%%%%%%%%%%%%%%%%%%%%%%%%%%%%%%%%%%%%%%%%%%%%%%%%%%%%%%%%%%%%%%%%%%%%

We illustrate the use of this estimation technique, 
along with the bounds (EV), (EV$'$), (FOV), and (PSA),
for a model problem from fluid dynamics. 
Let $\BA$ be the matrix generated by a 
streamline upwinded Petrov--Galerkin finite element discretization 
of the two-dimensional convection-diffusion
equation,
\[-\nu \Delta u + {\bf w}\cdot\nabla u = f 
   \quad {\rm on}\quad \Omega=[0,1]\times[0,1],\]
with diffusion coefficient $\nu=0.01$, 
constant advection in the vertical direction, ${\bf w} = [0,1]^T$, 
and Dirichlet boundary conditions that induce an interior layer and a
boundary layer.  The solution is approximated using bilinear finite
elements on a regular square grid with $N$ unknowns in each coordinate
direction, yielding a matrix $\BA$ of dimension $n=N^2$.\ \ 
This problem is discussed by
Fischer et al.~\cite{FRSW99}; we apply the upwinding parameter
they suggest, and focus on the case of $N=13$ ($n=169$).  
Though this upwinding parameter can give good approximate
solutions to the partial differential equation, the corresponding
matrix $\BA$ is highly nonnormal.  Its eigenvalues, though very sensitive
to perturbations, are known explicitly for this special wind
direction~\cite{FRSW99}.  These eigenvalues
fall on $N=13$ lines in the complex plane with constant real part,
with $N=13$ eigenvalues per line; the eigenvalues with largest
real part are the most ill-conditioned.
Liesen and Strakos have investigated the influence of the
spectral properties of this class of discretizations 
on GMRES convergence~\cite{LS05a}.

Figure~\ref{supg:bounds} shows GMRES convergence, along with
the bounds (EV), (FOV), and (PSA) for this model problem.
The Ideal GMRES curve was computed using the SDPT3 
Toolbox~\cite{TTT99}, and we compare it to GMRES
convergence for an initial residual derived from
boundary conditions that induce an internal layer and
a boundary layer.

%%%%%%%%%%%%%%%%%%%%%%%%%%%%%%%%%%%%%%%%%%%%%%%%%%%%%%%%%%%%%%%%%%%%%%%%%%%%%%%%
% SUPG CONVERGENCE
%%%%%%%%%%%%%%%%%%%%%%%%%%%%%%%%%%%%%%%%%%%%%%%%%%%%%%%%%%%%%%%%%%%%%%%%%%%%%%%%

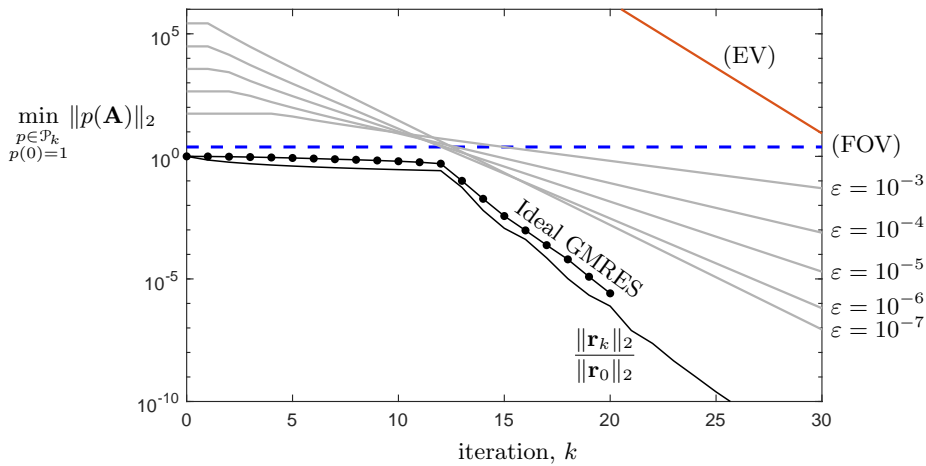
\begin{figure}[b!]
\begin{center}
\includegraphics[scale=0.57]{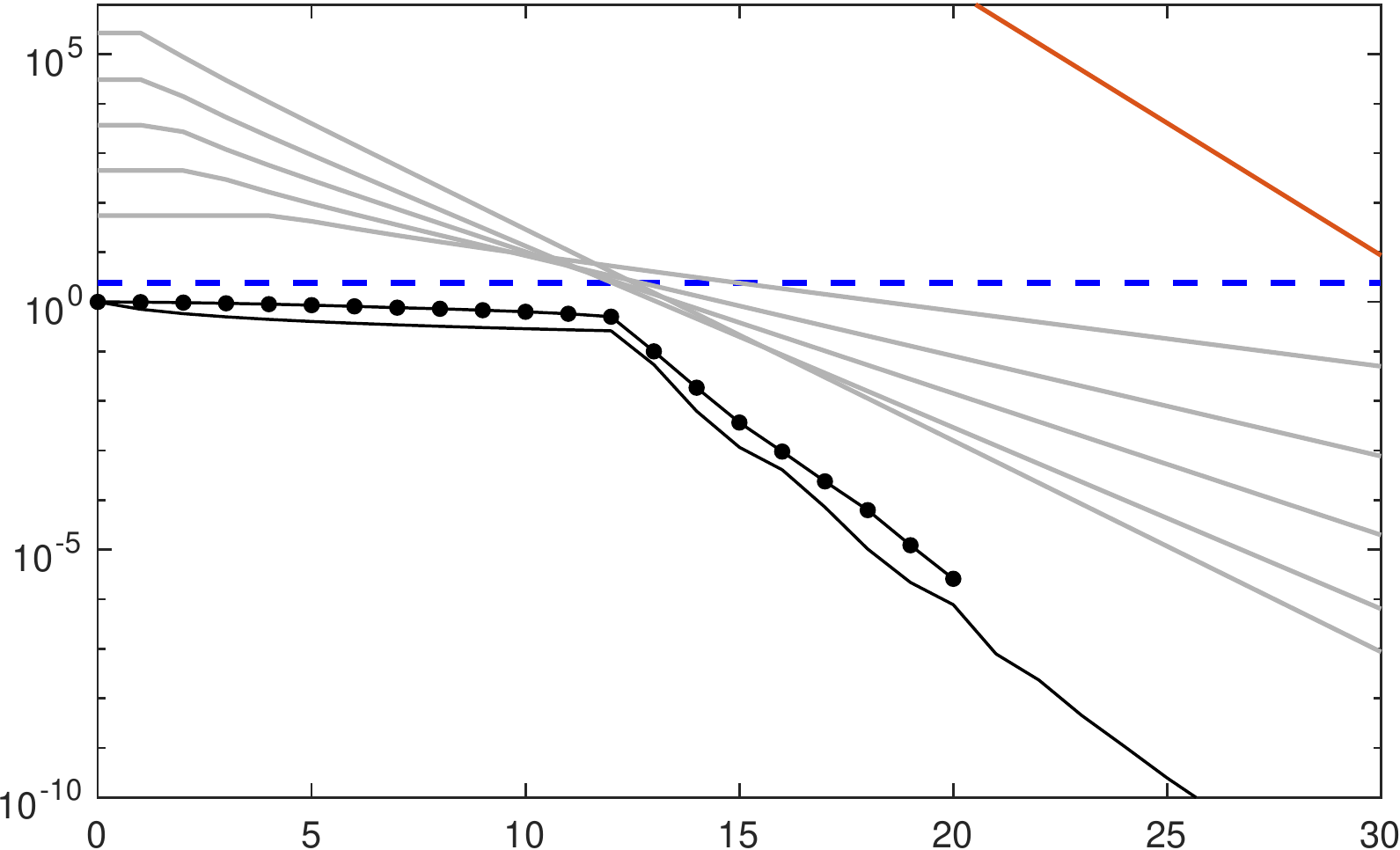}

\begin{picture}(0,0)
 \put(-180,127){\small $\displaystyle{\mingmres\! \norm{p(\BA)}}$}
 \put(-10,0){\small iteration, $k$}
 \put( 131,116){\small (FOV)}
 \put( 90,150){\small (EV)}
 \put( 131,100){\small $\eps=10^{-3}$}
 \put( 131, 84){\small $\eps=10^{-4}$}
 \put( 131, 68){\small $\eps=10^{-5}$}
 \put( 131, 55){\small $\eps=10^{-6}$}
 \put( 131, 46){\small $\eps=10^{-7}$}
 \put(33,37){\small $\displaystyle{\|\Br_k\|_2\over \|\Br_0\|_2}$}
 \put(11,94){\rotatebox{-36}{\small Ideal GMRES}}
\end{picture}
\end{center}

\caption{Convergence bounds for the convection-diffusion problem with $N=13$.}
\label{supg:bounds}
\end{figure}
%%%%%%%%%%%%%%%%%%%%%%%%%%%%%%%%%%%%%%%%%%%%%%%%%%%%%%%%%%%%%%%%%%%%%%%%%%%%%%%%

The matrix $\BA$ has a significant departure from normality,
as illustrated in Figure~\ref{supg:spec}.  
GMRES exhibits a period of slow convergence for
typical initial residuals, followed by a more rapid phase of convergence.
Ernst investigated the field of values bound (FOV) for this 
general problem~\cite{Ern00}.  
Since $\BA$ results from a coercive finite element discretization,
$\FOV(\BA)$ is contained in the open right-half plane and (FOV) 
guarantees convergence.  
Indeed, in this case the bound (FOV) gives $\rfov\approx 0.968$.  
As expected given the initial period of slow convergence,
this bound is descriptive at early iterations but fails to capture the
faster second phase of convergence.
As in Examples~E and~F, the bound~(PSA) does better.  
Though this bound somewhat underestimates the convergence rate attained
during the second phase of convergence
(at least for the values of $\eps$ shown in Figure~\ref{supg:bounds}),
it accurately captures the end of the slow first phase,
a feature that eludes the bounds~(EV) and~(FOV).

%%%%%%%%%%%%%%%%%%%%%%%%%%%%%%%%%%%%%%%%%%%%%%%%%%%%%%%%%%%%%%%%%%%%%%%%%%%%%%%%
% SUPG spectral quantities
%%%%%%%%%%%%%%%%%%%%%%%%%%%%%%%%%%%%%%%%%%%%%%%%%%%%%%%%%%%%%%%%%%%%%%%%%%%%%%%%

\begin{figure}
\begin{center}
\includegraphics[height=1.85in]{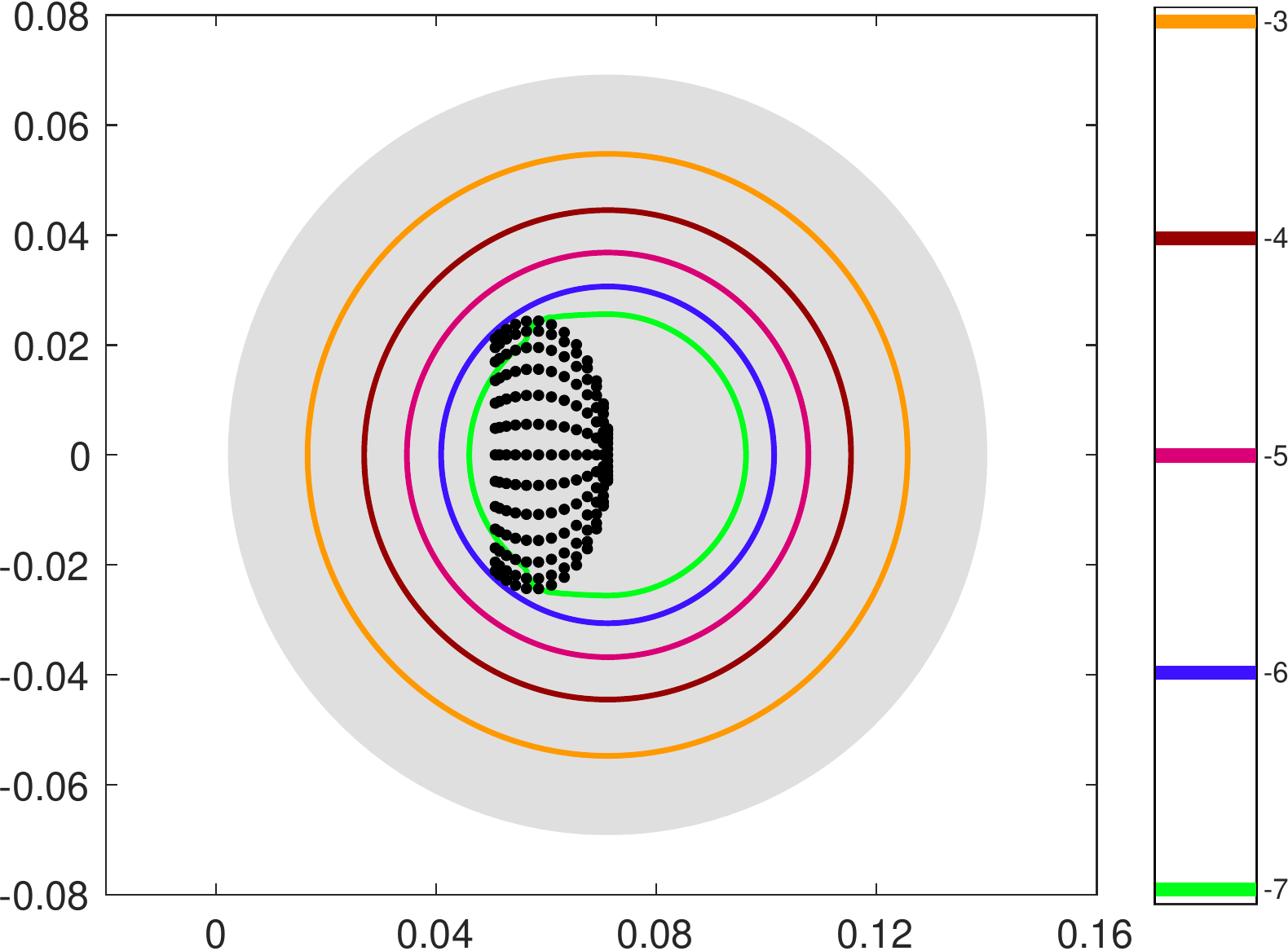}
\quad
\includegraphics[height=1.85in]{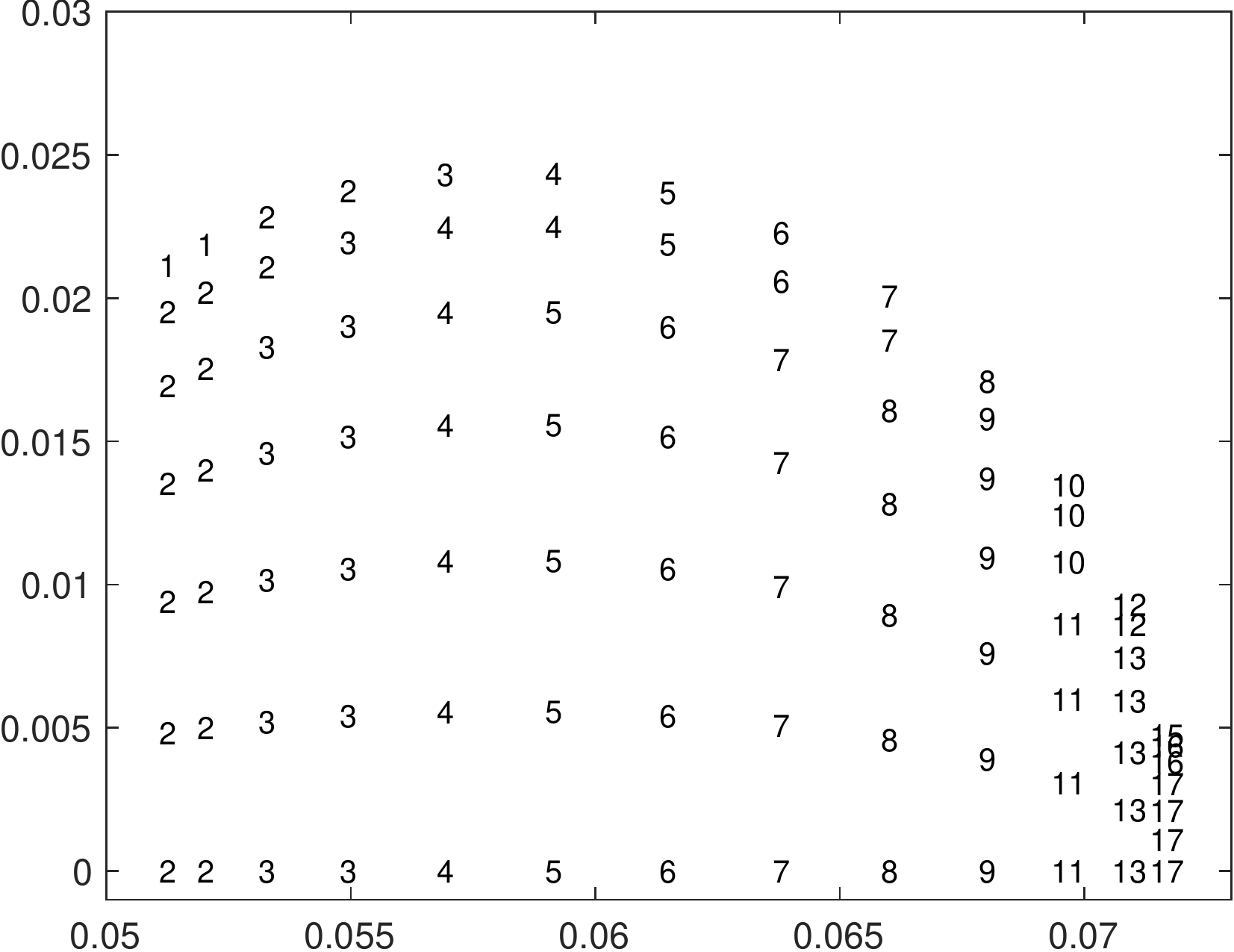}
\end{center}
\caption{ \label{supg:spec}
On the left, the eigenvalues (black dots), field of values (gray region), 
and $\eps$-pseudospectra ($\eps=10^{-3}$, $10^{-4}, \ldots, 10^{-7}$)
for the convection-diffusion problem with $N=13$.
On the right, the condition number of the eigenvalues used in {\rm (EV$'$)}
, displayed as $\log_{10}(\kappa(\lambda_j))$ rounded
to the nearest integer and located at $\lambda_j\in\C$, 
for the eigenvalues on and above the real axis.}
\end{figure}
%%%%%%%%%%%%%%%%%%%%%%%%%%%%%%%%%%%%%%%%%%%%%%%%%%%%%%%%%%%%%%%%%%%%%%%%%%%%%%%%

Explicit formulas are available for the eigenvalues and eigenvectors 
of this matrix~\cite{FRSW99}.
For our chosen parameters all the eigenvalues are distinct, and hence the matrix
is diagonalizable; however, scaling each column of the eigenvector
matrix to have unit 2-norm, we compute $\COND(\BV) \approx 4.6\times 10^{16}$:
the matrix is close to being non-diagonalizable.
To get an upper bound on $\rev$, we use  the convergence rate 
associated with the convex hull of the true eigenvalues of $\BA$; 
Figure~\ref{supg:bounds} shows that this \emph{rate} seems to agree
with the second phase of convergence (as observed in~\cite[p.~191]{FRSW99}),
but $\cev$ is much too large to make the bound descriptive.  
The pseudospectral bounds are better; 
we estimate the rates $\rpsa(\eps)$ by calculating the
convergence rate of an approximate convex hull of $\PSA(\BA)$
and applying the convergence bound for convex sets~(\ref{convex}).

%%%%%%%%%%%%%%%%%%%%%%%%%%%%%%%%%%%%%%%%%%%%%%%%%%%%%%%%%%%%%%%%%%%%%%%%%%%%%%%%
% SUPG CONVERGENCE - (EV')
%%%%%%%%%%%%%%%%%%%%%%%%%%%%%%%%%%%%%%%%%%%%%%%%%%%%%%%%%%%%%%%%%%%%%%%%%%%%%%%%
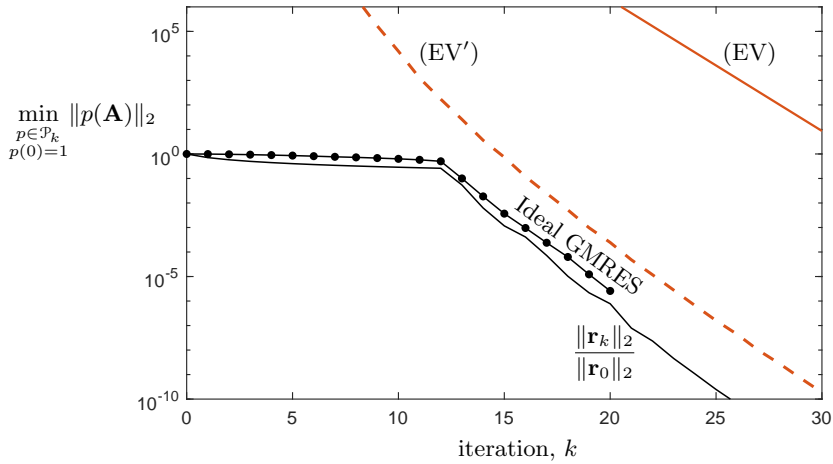
\begin{figure}  
\begin{center}
\includegraphics[scale=0.57]{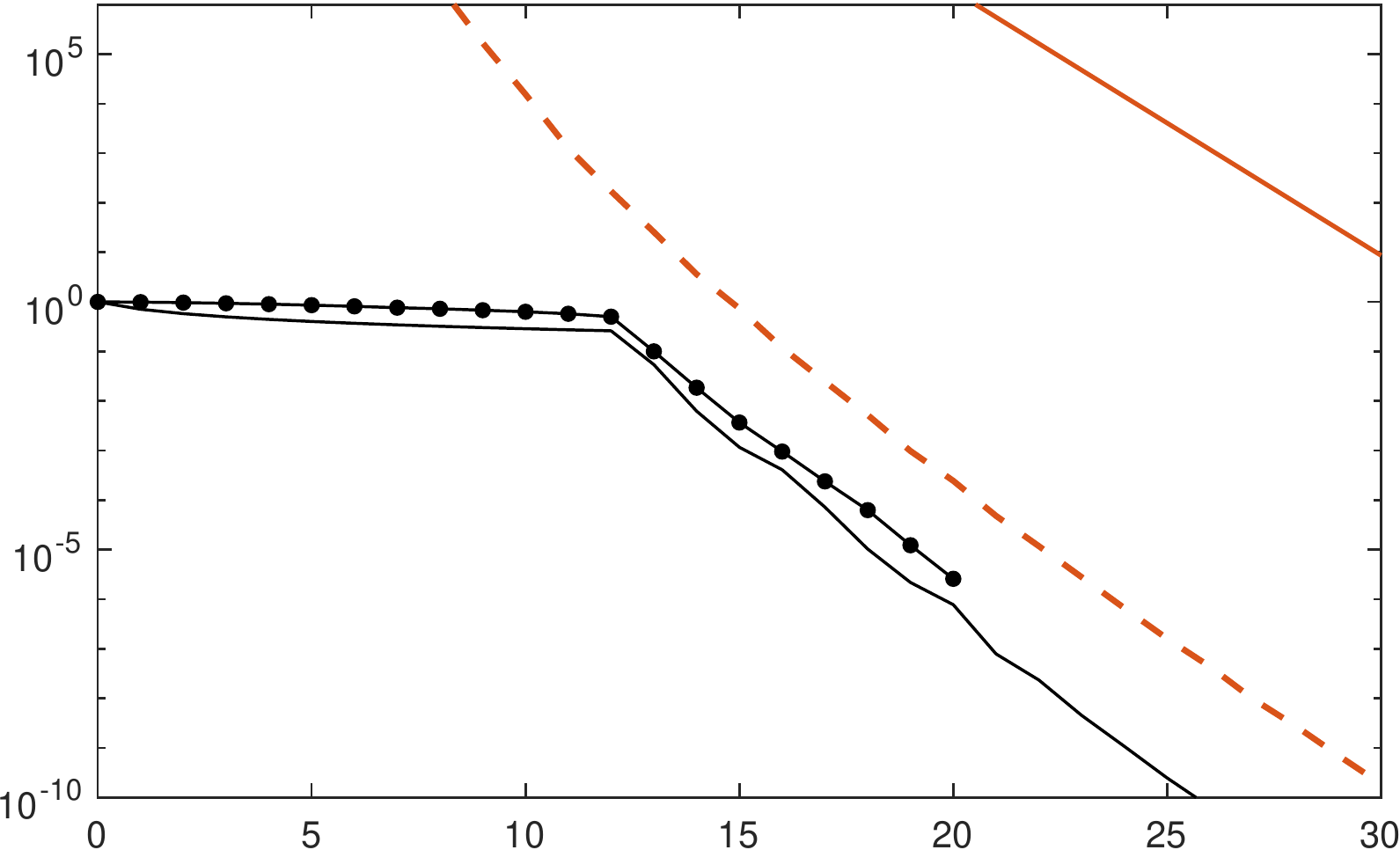}

\begin{picture}(0,0)
 \put(-180,127){\small $\displaystyle{\mingmres\! \norm{p(\BA)}}$}
 \put(-10,0){\small iteration, $k$}
 \put( 90,150){\small (EV)}
 \put(-25,150){\small (EV$'$)}
 \put(33,37){\small $\displaystyle{\|\Br_k\|_2\over \|\Br_0\|_2}$}
 \put(11,94){\rotatebox{-36}{\small Ideal GMRES}}
\end{picture}
\end{center}
\caption{Convergence bound {\rm (EV$'$)} for the 
          convection-diffusion problem with $N=13$.}
\label{supg:evprime}
\end{figure}
%%%%%%%%%%%%%%%%%%%%%%%%%%%%%%%%%%%%%%%%%%%%%%%%%%%%%%%%%%%%%%%%%%%%%%%%%%%%%%%%

The left plot in Figure~\ref{supg:spec} shows the eigenvalues, 
field of values, and pseudospectra for this example.  
While the eigenvalues of $\BA$ are ill-conditioned, the degree of 
ill-conditioning is not uniform across the spectrum: the eigenvalues 
closest to the origin are less sensitive than the rightmost eigenvalues.
The right plot in Figure~\ref{supg:spec} plots
$\log_{10}(\kappa(\lambda_j))$ (rounded to the nearest integer)
at the location of $\lambda_j$, for those eigenvalues on or above
the real axis.
Given the extreme range of $\kappa(\lambda_j)$ values, 
how does the bound (EV$'$) perform in this situation?  
Figure~\ref{supg:evprime} illustrates that (EV$'$) can handle such disparate 
ill-conditioning quite well.  In the second phase of convergence, 
(EV$'$) is more accurate than the three standard bounds.  For this example, 
the eigenvalue condition numbers were computed from explicit
formulas for left and right eigenvectors, and (EV$'$)
was calculated in quadruple precision arithmetic.

%%%%%%%%%%%%%%%%%%%%%%%%%%%%%%%%%%%%%%%%%%%%%%%%%%%%%%%%%%%%%%%%%%%%%%%%%%%%%%%%
% FIGURES FOR ADAPTIVE EXAMPLES USING THE SUPG MATRIX FOR N=13
%%%%%%%%%%%%%%%%%%%%%%%%%%%%%%%%%%%%%%%%%%%%%%%%%%%%%%%%%%%%%%%%%%%%%%%%%%%%%%%%
\begin{figure}[t!]  
\begin{center}
\includegraphics[scale=0.35]{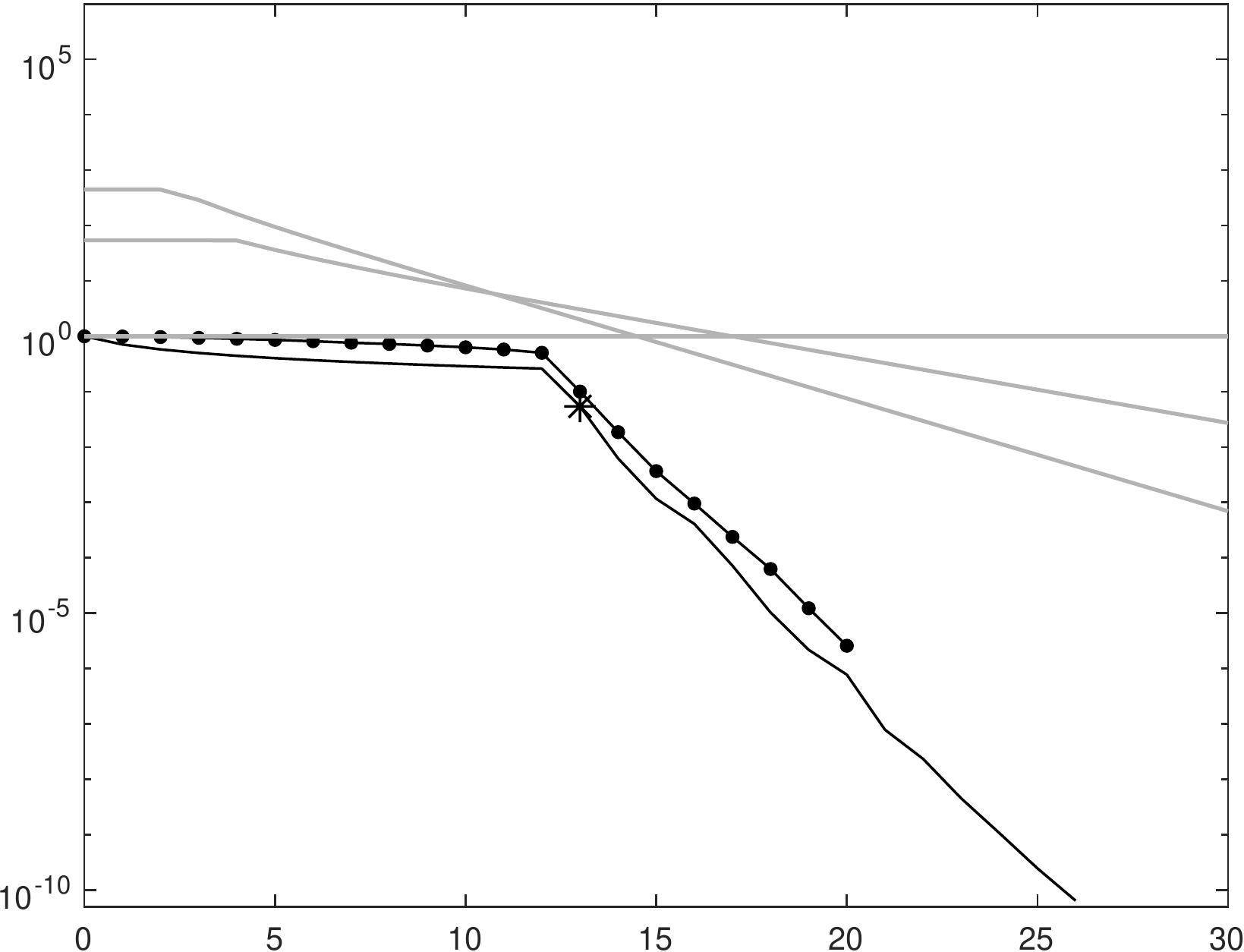}
\begin{picture}(0,0)
\put(-71,18){\footnotesize $\displaystyle{\|\Br_k\|_2\over \|\Br_0\|_2}$}
\put(-87,72){\footnotesize \rotatebox{-42}{Ideal GMRES}}
\put(-42,83.5){\footnotesize $\eps\to\infty$}
\put(-42,71.5){\footnotesize \rotatebox{-14}{$\eps=10^{-3}$}}
\put(-42,62.7){\footnotesize \rotatebox{-18}{$\eps=10^{-4}$}}
\put(-46,-8){\footnotesize $k$}
\end{picture}
\quad
\includegraphics[scale=0.35]{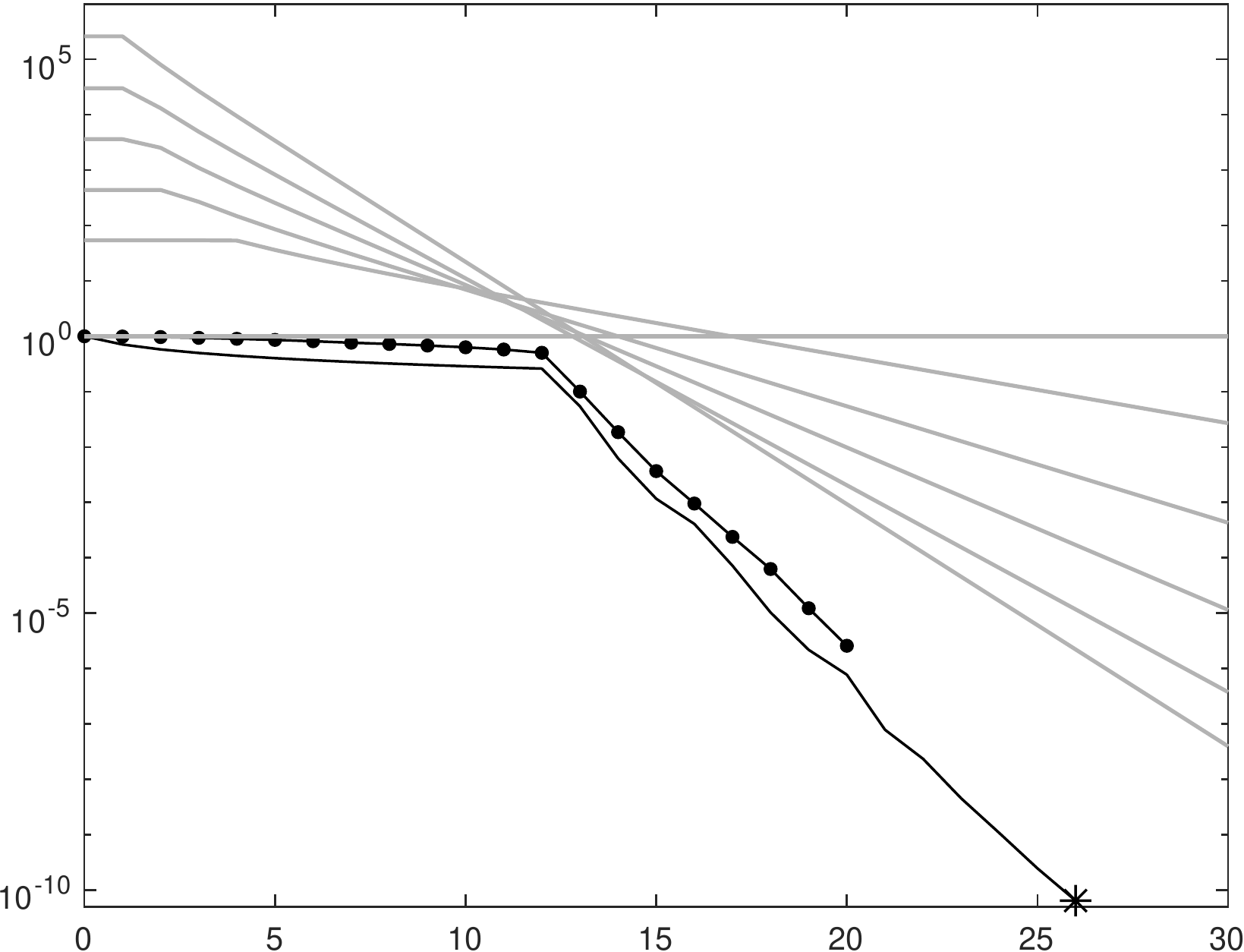}
\begin{picture}(0,0)
\put(-71,18){\footnotesize $\displaystyle{\|\Br_k\|_2\over \|\Br_0\|_2}$}
\put(-87,72){\footnotesize \rotatebox{-43}{Ideal GMRES}}
\put(-42,83.5){\footnotesize $\eps\to\infty$}
\put(-42,77.5){\footnotesize \rotatebox{-9}{$\eps=10^{-3}$}}
\put(-42,69.5){\footnotesize \rotatebox{-17}{$\eps=10^{-4}$}}
\put(-42,62.5){\footnotesize \rotatebox{-23}{$\eps=10^{-5}$}}
\put(-42,55.25){\footnotesize \rotatebox{-28}{$\eps=10^{-6}$}}
\put(-42,42){\footnotesize \rotatebox{-33}{$\eps=10^{-7}$}}
\put(-46,-8){\footnotesize $k$}
\end{picture}
\end{center}
\caption{Adaptive convergence bounds for the convection-diffusion 
problem with $N=13$, 
generated at iteration $k=13$  on the left and $k=26$ on the right,
based on pseudospectra of $\BH_k$ shown in Figure~\ref{fig:psasupga}.}
\label{supg:adapt}
\end{figure}
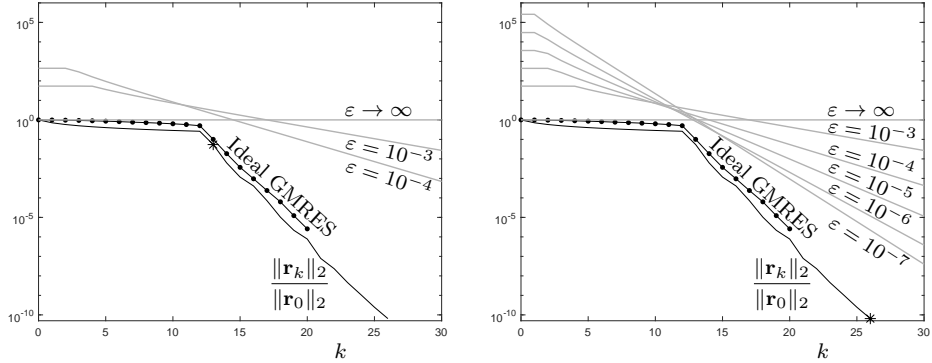
%%%%%%%%%%%%%%%%%%%%%%%%%%%%%%%%%%%%%%%%%%%%%%%%%%%%%%%%%%%%%%%%%%%%%%%%%%%%%%%%

%%%%%%%%%%%%%%%%%%%%%%%%%%%%%%%%%%%%%%%%%%%%%%%%%%%%%%%%%%%%%%%%%%%%%%%%%%%%%%%%
\begin{figure}
\begin{center}
\includegraphics[scale=0.395]{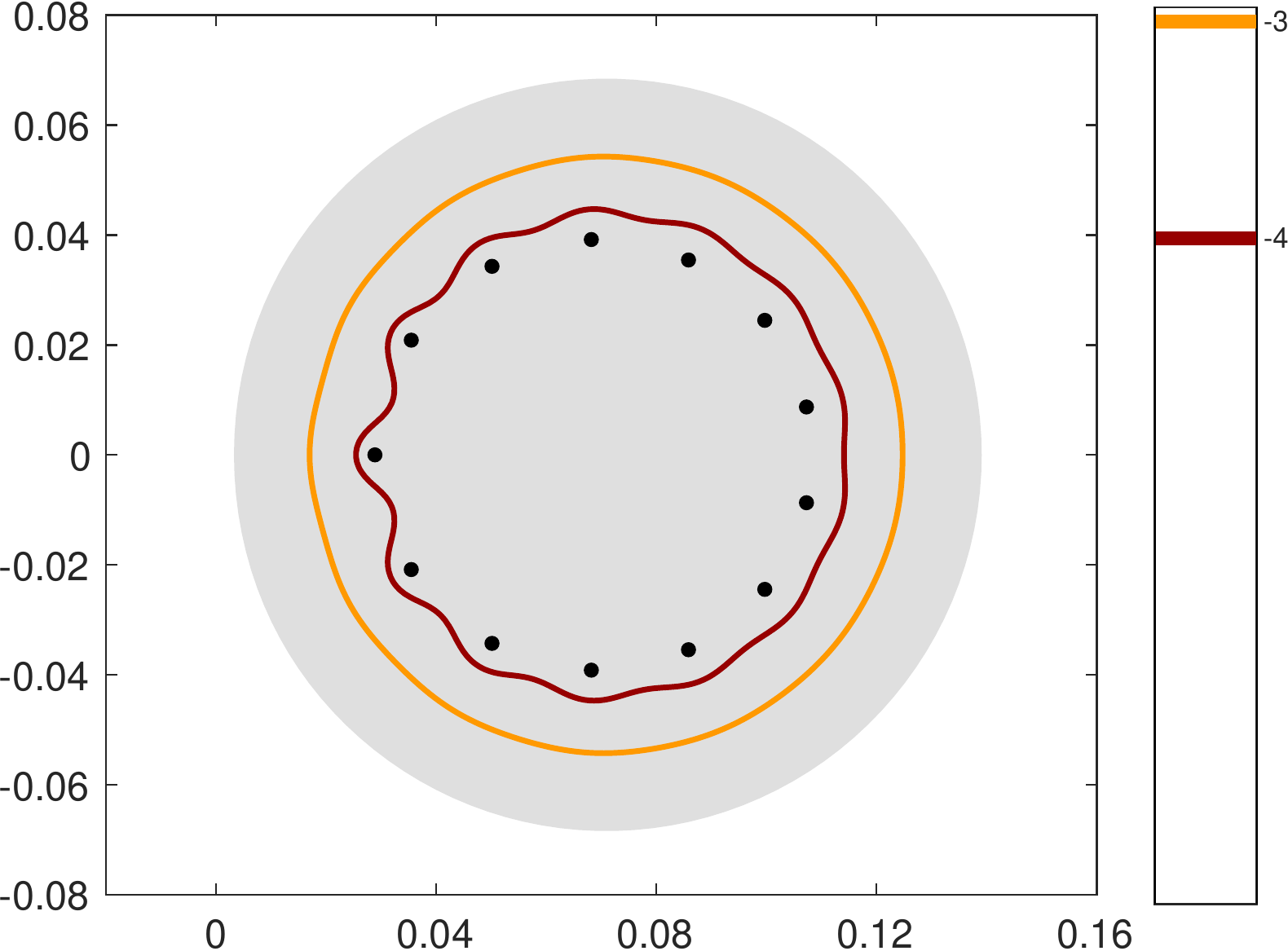}
\begin{picture}(0,0)
\put(-162,13){\footnotesize $k=13$}
\end{picture}
\includegraphics[scale=0.395]{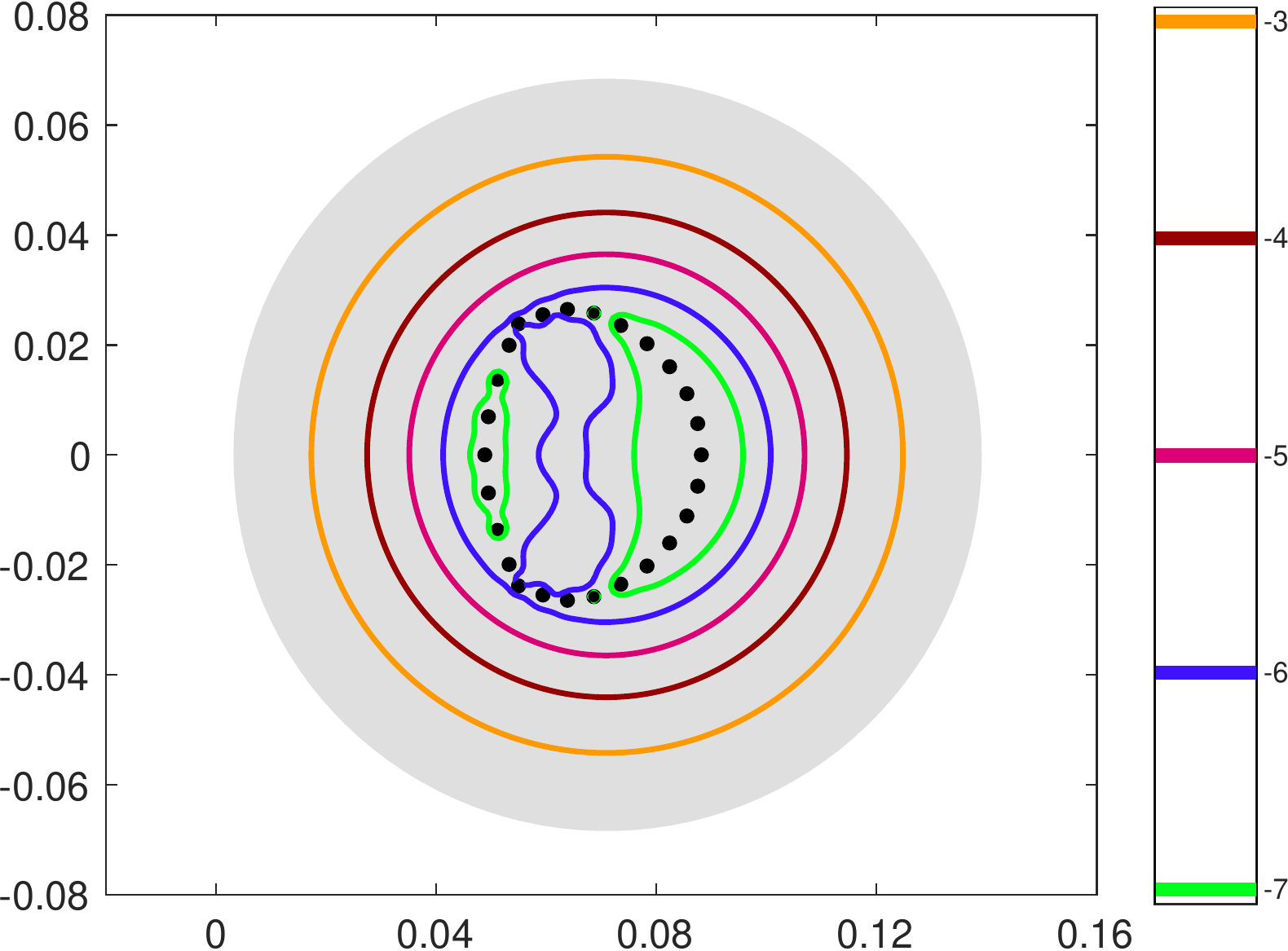}
\begin{picture}(0,0)
\put(-162,13){\footnotesize $k=26$}
\end{picture}
\end{center}

\vspace*{3pt}
\caption{\label{fig:psasupga}
At iterations $k=13$ and $k=26$,
the field of values $\FOV(\BH_k)$ (gray region) and $\eps$-pseudospectra $\PSA(\BH_k)$
(for $\eps=10^{-3}, 10^{-4}, 10^{-5}, 10^{-6}, 10^{-7}$)
for the convection-diffusion problem with $N=13$,
generated using the $\Br_0$ whose GMRES convergence
is illustrated in Figure~\ref{supg:adapt}.
The color levels use the same scale as those in Figure~\ref{supg:spec}
for the full matrix $\BA$, to facilitate comparison.}
\end{figure}
%%%%%%%%%%%%%%%%%%%%%%%%%%%%%%%%%%%%%%%%%%%%%%%%%%%%%%%%%%%%%%%%%%%%%%%%%%%%%%%%

We test the pseudospectral estimates described earlier in this
section for $N=13$ with the same initial residual described above.
The adaptive estimates taken during the initial phase of slow convergence
give little hint of future convergence behavior; the approximate
pseudospectra do not improve much from iteration to iteration
during these early steps.  For $k=13$, at the onset of
more rapid convergence, $\PSA(\BH_k)\approx \PSA(\BA)$ for
$\eps\ge 10^{-4}$ and one gets an indication of the improved
convergence to come.  By the time the convergence criterion
is satisfied at $k=26$, $\PSA(\BH_k)\approx \PSA(\BA)$ for
those values of $\eps$ relevant to the bound (PSA).\ \ 
Figure~\ref{supg:adapt} shows these convergence estimates
for $k=13$ and $k=26$, based on the pseudospectra shown in  
Figure~\ref{fig:psasupga}.  Notice how the agreement between
these pseudospectra of $\BH_k$ and those of the full matrix
$\BA$ (shown in Figure~\ref{supg:spec}) improve as $k$ increases,
in agreement with the observations of Toh and Trefethen~\cite{tt96}.

%%%%%%%%%%%%%%%%%%%%%%%%%%%%%%%%%%%%%%%%%%%%%%%%%%%%%%%%%%%%%%%%%%%%%%%%%%%%%%%%
\section{Summary}
%%%%%%%%%%%%%%%%%%%%%%%%%%%%%%%%%%%%%%%%%%%%%%%%%%%%%%%%%%%%%%%%%%%%%%%%%%%%%%%%

We have explored some of the relative merits of convergence bounds 
based on eigenvalues (with the eigenvector condition number),
the field of values, and pseudospectra.  In particular,
these bounds have distinct weaknesses that indicate situations
in which one bound may be preferred over the others.  
The standard bounds are global statements that can be
refined; for (EV$'$) and (FOV$'$) this localization
introduces spectral projector norms.
Pseudospectra provide a convenient tool for bridging
between the eigenvalues and the field of values, but they
can be expensive to compute.  Approximate pseudospectra
drawn from the Arnoldi process yield convergence estimates
at a fraction of the cost of full pseudospectral computation.

%%%%%%%%%%%%%%%%%%%%%%%%%%%%%%%%%%%%%%%%%%%%%%%%%%%%%%%%%%%%%%%%%%%%%%%%%%%%%%%%%
\section*{Acknowledgements}
%%%%%%%%%%%%%%%%%%%%%%%%%%%%%%%%%%%%%%%%%%%%%%%%%%%%%%%%%%%%%%%%%%%%%%%%%%%%%%%%%

I thank Andy Wathen for guiding this research and suggesting
numerous improvements to this presentation.  I am also grateful for 
Nick Trefethen's many helpful comments.  The title and the format of Section~3
were inspired by a paper by Nachtigal, Reddy, and Trefethen~\cite{nrt92}.
I thank Anne Greenbaum for providing a copy of reference~\cite{eiermann97},
and Chris Beattie, Bernd Fischer, and Henk van der Vorst for 
stimulating discussions related to this research.  
The helpful referees who reviewed the original version of this 
document pointed out several references, and made other suggestions 
that have improved this work.

\end{document}